\UseRawInputEncoding

\documentclass[a4paper,11pt]{article}
\usepackage{bbm}
\usepackage{mathrsfs}
\usepackage{amsfonts}
\usepackage{amssymb}
\usepackage{tabularx}
\usepackage{enumerate}
\usepackage{graphicx}
\usepackage{texdraw}
\usepackage{tikz}
\usepackage{cite}
\usepackage{hyperref}
\usepackage{amsthm}
\usepackage{amscd}
\usepackage[top=1.38in, bottom=1.39in, left=1.18in, right=1.17in]{geometry}
\usepackage[all]{xy}
\usepackage{mathrsfs}
\usepackage{amscd}
\usepackage{color}
\newtheorem{theo}{Theorem}[section]
\newtheorem{prop}[theo]{Proposition}
\newtheorem{lem}[theo]{Lemma}
\newtheorem{rem}[theo]{Remark}
\newtheorem{col}[theo]{Corollary}
\newtheorem{defi}[theo]{Definition}
\newtheorem{que}[theo]{Question}

\newcommand{\be}{\begin{equation}}
\newcommand{\ee}{\end{equation}}
\newcommand\bes{\begin{eqnarray}}
\newcommand\ees{\end{eqnarray}}
\newcommand{\bess}{\begin{eqnarray*}}
\newcommand{\eess}{\end{eqnarray*}}

\usepackage[titletoc,toc]{appendix}
\def\theequation{\arabic{section}.\arabic{equation}}

\begin{document}
 \setlength{\baselineskip}{15pt} \pagestyle{myheadings}

 \title{The Calabi-Yau Equation on Symplectic Manifolds\thanks{The work
 is supported by PRC grant NSFC 11701226 (Tan), 11771377, 12171417 (Wang);
  Natural Science Foundation of Jiangsu Province BK20170519 (Tan). }}

 \author{Qiang Tan and Hongyu Wang\thanks{E-mail:hywang@yzu.edu.cn}\\
  {\small }}

 \date{}
 \maketitle

 \noindent {\bf Abstract:} {
  By using the global deformation of  almost complex structures which are compatible with a symplectic form off a Lebesgue measure zero subset, we construct a (measurable) Lipschitz K\"{a}hler metric
  such that the one-form type Calabi-Yau equation on an open dense submanifold is reduced to the complex Monge-Amp\`{e}re equation with respect to the measurable K\"{a}hler metric.
 We give an  existence theorem for solutions to the one-form type Calabi-Yau equation on closed symplectic manifolds.
 }\\



 \section{Introduction}\label{Introduction}
  Yau's Theorem for closed K\"{a}hler manifolds \cite{Yau} states that one can prescribe the volume form
  of a K\"{a}hler metric within a given K\"{a}hler class.
  This result occupies a central place in the theory of K\"{a}hler manifolds (cf. Calabi \cite{Calabi1,Calabi2}), with wide-ranging applications in geometry
  and mathematical physics (cf. Joyce \cite{Joc}, Yau \cite{Yau2} and Fu-Yau \cite{FY}).

  More precisely, Yau's Theorem is as follows.
  Let $(M^{2n},\omega,J,g_J)$ be a closed K\"{a}hler manifold of real dimension $2n$, where $\omega$ denotes the  K\"{a}hler form,
  $J$ an $\omega$-compatible complex structure and $g_J(\cdot,\cdot)=\omega(\cdot,J\cdot)$.
  Then given a smooth volume form $e^F\omega^n$ on $M^{2n}$ there exists a unique smooth function $\varphi$ satisfying
  \begin{eqnarray}\label{Monge-Ampere}
    (\omega+\sqrt{-1}\partial_J\bar{\partial}_J\varphi)^n &=& e^F\omega^n, \nonumber\\
      \omega+\sqrt{-1}\partial_J\bar{\partial}_J\varphi  &>&  0, \,\,\,\sup_{M^{2n}} \varphi=0
  \end{eqnarray}
  as long as $F$ satisfies the necessary normalization condition that
  \begin{equation}
    \int_{M^{2n}}e^F\omega^n= \int_{M^{2n}}\omega^n.
  \end{equation}
   Equation (\ref{Monge-Ampere}) is known as the complex Monge-Amp\`{e}re equation for K\"{a}hler manifolds.
   It is based on the method of continuity and in any given situation requires laborious a priori estimates for the derivatives of
   the solution $\varphi$ up to third order (cf. S.-T. Yau \cite{Yau} or D. D. Joyce \cite[Chapter 6]{Joc}).
   The uniqueness of the solution $\varphi$ has an elementary proof found by Calabi \cite[p. 86]{Calabi2}.

   In a similar vein, L. Caffarelli, J. J. Kohn, L. Nirenberg and J. Spruck \cite{CKNS} have proved regularity results for the Dirichlet problem
   associated the following equation:
   \begin{equation}\label{Dir equ}
   \det(\frac{\partial^2\varphi}{\partial z_j\partial \bar{z}_k})=f
   \end{equation}
    in a strictly pseudoconvex domain.
   Also there has been interest in extending the Dirichlet problem for Equation (\ref{Dir equ}) in a strictly pseudoconvex domain with
   continuous boundary data.
   Again, in general $f$ must be replaced by Borel measure (cf. S. Ko{\l}odziej \cite{Kolo} and the references therein).

   \vskip 6pt

   Recall that a K\"{a}hler metric is positive definite real $(1,1)$ form $\omega$,
   namely a Hermitian metric, which satisfies
   $d\omega=0$.
   One extension of Yau's Theorem, initiated by Cherrier \cite{Cherr} in the 1980's, is to remove this closedness condition.
   This was carried out in full generality in Tosatti-Weinkove \cite{TosaWein} and the references therein where it was shown that (\ref{Monge-Ampere})
   has a unique solution for $\omega$ up to adding a (unique) constant to $F$.
   In \cite{FWW1,FWW2}, Fu, Wang and Wu consider the form -type Calabi-Yau equations.

   \vskip 6pt

    A different type of extension of Yau's Theorem is to the case when $J$ is a non-integrable almost complex structure.
    Recall that an almost complex structure $J$ is said to be tamed by symplectic from $\omega$ when the bilinear form $\omega(\cdot,J\cdot)$
    is positive definite. The almost complex structure $J$ is said to be compatible (or calibrate) with $\omega$ when the same bilinear form is also symmetric,
    that is, $\omega(\cdot,J\cdot)>0$ and $\omega(J\cdot,J\cdot)=\omega(\cdot,\cdot)$ (cf. M. Gromov \cite{Gro0} or McDuff-Salamon \cite{MS}).
    In 1990's, Gromov posed the following problem to P. Delano\"{e} \cite{Dela}:
    Let $(M^{2n},\omega)$ be a closed symplectic manifold, $J$ an almost complex structure compatible with $\omega$
    and $F$ a smooth function on $M^{2n}$ with
    $$
      \int_{M^{2n}}e^F\omega^n= \int_{M^{2n}}\omega^n.
    $$
  Can one find a smooth function $\varphi$ on $M$ such that $\omega+d(Jd\varphi)$ is a symplectic form taming $J$ and satisfying
  \begin{equation}\label{Gro CY}
    (\omega+d(Jd\varphi))^n=e^F\omega^n ?
  \end{equation}
 However, P. Delano\"{e} \cite{Dela} showed that when $n=2$ the answer to this question is negative, this was later extended to all
 dimensions by Wang-Zhu \cite{WZ}.
 The key ingredient of their results in the construction of a smooth function such that $\omega+d(Jd\varphi_0)$ is on the boundary of the
 set of taming symplectic forms (so its $(1,1)$ part is semipositive definite but not strictly positive),
  and yet $(\omega+d(Jd\varphi_0))^n>0$.
  This is possible because in this case the $(2,0)+(0,2)$ part of $d(Jd\varphi_0)$ contributes a strictly positive amount.
  This indicate that the problem with Gromov's suggestion is that the 2-form $d(Jd\varphi)$ is in general not of type $(1,1)$ with respect to $J$,
  due to the fact that $J$ may not be integrable.
  Its $(1,1)$ part will be denoted by
  $$
  \sqrt{-1}\partial_J\bar{\partial}_J\varphi =\frac{1}{2}(d(Jd\varphi))^{(1,1)}
  $$
  which agrees with the standard notation when $J$ is integrable.

  Therefore, one can consider the Monge-Amp\`{e}re equation for non-integrable almost complex structures.
  There are many important works when the base manifolds are almost complex manifolds.
  We refer to Harvey-Lawson \cite{HL}, Tosatti-Wang-Weinkove-Yang \cite{TWWY}, Chu-Tosatti-Weinkove \cite{CTW},
  and the references therein.

  \vskip 6pt


  S.K. Donaldson gave a conjecture in \cite{Donal}.
  Suppose that $J$ is an almost complex structure on a closed symplectic manifold $(M,\omega)$ of dimension $4$ and is tamed by the symplectic form.
  Let $\sigma$ be a smooth volume form on $M$ with
  \begin{equation}
    \int_{M}\sigma= \int_{M}\omega^2.
  \end{equation}
  Can one find a $J$-compatible symplectic form  $\tilde{\omega}$
 with $[\tilde{\omega}]=[\omega]$,
  and solve the Calabi-Yau equation
   \begin{equation}\label{CY equation}
  \tilde{\omega}^2=\sigma ?
  \end{equation}
  Donaldson conjectured that there are $C^\infty$ a prior bounds on $\tilde{\omega}$ depending only on $\omega$, $J$ and $\sigma$.
  This is related to his broaden program including Donaldson's ``tamed to compatible" conjecture.
  This project is partially confirmed by Taubes in \cite{Tau}
  and Tan-Wang-Zhou-Zhu in \cite{TWZZ}.

  Tosatti and Weinkove \cite{TosaWein1,TosaWein2}, Vezzoni \cite{Vezz}, Fino, Li, Salamon and Vezzoni \cite{FLSV},
   Buzano, Fino and Vezzoni \cite{BFV1,BFV2} studied the Calabi-Yau equation on the Kodaira-Thuston manifold or $\mathbb{T}^2$-bundles over 2-tori
   under certain condition.

   \vskip 6pt

   In \cite{Wein,TWY}, Tosatti, Weinkov and Yau attacked the problem by introducing an almost K\"{a}hler potential $\varphi$,
   defined by
   \begin{equation}\label{TWY equ}
    \tilde{\omega}=\omega+\frac{1}{2}dJd\varphi+da,\,\,\,  \tilde{\omega}\wedge da=0,\,\,\,  \sup_{M^4}\varphi=0
   \end{equation}
  using the sign convention in \cite{CTW} and where $a$ is a 1-form.
  It was shown in \cite{Wein,TWY} that $C^\infty$ estimates form $\tilde{\omega}$ in (\ref{CY equation})
  follow from an $L^\infty$ bounded on $\varphi$, and this was reduced in \cite{TWY} to a bound on
  $$\int_{M^4}e^{-\alpha\varphi}\omega^2$$ for some $\alpha>0$.


  \vskip 6pt

  Suppose that $(M^{2n},\omega,J,g_J)$ is a closed almost K\"{a}hler manifold of dimension $2n$.
  By using the deformation of $\omega$-compatible almost complex structures,
  we can construct a measurable K\"{a}hler flat structure $(\omega,J_0,g_0)$ on $M^{2n}$ which is homotopy equivalent to
  the given smooth almost K\"{a}hler structure $(\omega,J,g_J)$ and satisfies Lipschitz condition (\ref{metric equ})
  (more details see Appendix \ref{App A}).
  It is similar to K\"{a}hler case (cf. Yau \cite{Yau} or Tosatti-Weinkove-Yau \cite[Theorem 1.3]{TWY}, Wang-Zhu \cite[Theorem 1.2]{WZ}),
  we have the following generalized Calabi-Yau equation of Donaldson type (that is, the one-form type Calabi-Yau equation) on symplectic manifolds:
  \begin{theo}\label{main result}
  Suppose that $(M^{2n},\omega)$ is a closed symplectic manifold of dimension $2n$.
  Let $F\in C^\infty(M^{2n},\mathbb{R})$ satisfying
  $$
    \int_{M^{2n}}e^F\omega^n= \int_{M^{2n}}\omega^n.
  $$
  Then, there exists a smooth $1$-form, $a(F)$, on $M^{2n}$ satisfying
  $$
  (\omega+da(F))^n=e^F\omega^n.
  $$
  Therefore, $\omega+da(F)$ is a new symplectic form on $M^{2n}$ which is cohomologous to $\omega$.
  \end{theo}

  \begin{rem}
  Theorem \ref{main result} can be regarded a generalization of Yau's Theorem \cite{Yau} in symplectic case.
   Suppose that $(M^{2n}, \omega)$ is a closed symplectic manifold of dimension $2n$.
   If there exists an $\omega$-compatible, integrable complex structure $J_0$ on $M^{2n}$,
   let $g_{J_0}(\cdot,\cdot):= \omega(\cdot,J_0\cdot)$, then $(\omega,J_0,g_{J_0})$ is a K\"{a}hler structure on $M^{2n}$.
   Thus, in K\"{a}hler case, the Calabi-Yau equation is reduced to the complex Monge-Amp\`{e}re equation.
   By Yau's Theorem \cite{Yau}, there exists a unique solution to the complex Monge-Amp\`{e}re equation for closed K\"{a}hler manifolds
   (see also Aubin \cite[Chapter 7]{Au2}).
   In general, the $1$-form $a(F)$ in Theorem \ref{main result} is not unique.
   Notice that in \cite{STW}, Sz\'{e}kelyhidi, Tosatti and Weinkove pointed that for Gauduchon conjecture \cite{Gau2}
   (that is, there is a Gauduchon metric $\omega$ on closed complex manifold with prescribed volume form), there can be no uniqueness.
   Donaldson \cite{Donal0} conjectured that the space of K\"{a}hler metrics is geodesic convex by smooth geodesic and that it is a metric space.
   Following Donaldson's program, X. X. Chen \cite{Chen1,Chen2}, E. Calabi and X. X. Chen \cite{CaCh} studied the space of K\"{a}hler metrics.
  \end{rem}

  Notice that by the solution of the Calabi-Yau equation on a closed symplectic manifold $(M^{2n}, \omega)$,
   we obtain a new symplectic form on $M^{2n}$ which is cohomologous to the given symplectic form $\omega$.
   Hence, on closed symplectic manifolds, uniqueness of the solution to the Calabi-Yau equation on given symplectic manifold can be regarded
   as the uniqueness of the obtained symplectic structure on the given symplectic manifold in the sense of following definition.

   Let $M^{2n}$ be a closed manifold of dimension $2n$.
   For two symplectic forms $\omega_0$ and $\omega_1$ on $M^{2n}$,
   we say that $\omega_0$ and $\omega_1$ have equivalent relation if and only if $\omega_0$ and $\omega_1$
   are connected by a path of cohomologous symplectic forms (cf. D. A. Salamon \cite{Sala}).
   Hence, the equivalent relation of solutions to the  Calabi-Yau equation is defined as follows:
   \begin{defi}
   Let $(M^{2n}, \omega)$ be a closed symplectic manifold of dimension $2n$.
   Suppose that $a_0(F)$ and $a_1(F)\in \Omega^1(M^{2n})$ are two solutions of the Calabi-Yau equation on $(M^{2n}, \omega)$:
   $$
   (\omega+da)^n=e^F\omega^n,
   $$
   where
    $$a\in\Omega^1(M^{2n}), \,\,\,\int_{M^{2n}}\omega^n=\int_{M^{2n}}e^F\omega^n.$$
    $a_0(F)$ and $a_1(F)$ have an equivalent relation if and only if $\omega+da_0(F)$ and $\omega+da_1(F)$
     are connected by a path of cohomologous symplectic forms.
     $\omega+da_0(F)$ and $\omega+da_1(F)$ are called isotopic.
   \end{defi}

   Notice that by Moser isotopy (cf. Moser \cite{Moser} or McDuff-Salamon \cite[\S 3.3]{MS}),
   $\omega_0$ and $\omega_1$ have equivalent relation if and only if there exists a diffeomorphism
   $\psi:M^{2n}\rightarrow M^{2n}$ isotopic to the identity such that $\psi^*\omega_1=\omega_0$.
   Denote the space of symplectic forms on $M^{2n}$ representing the class $[\omega]$ by
   $$
   \Psi(M^{2n},[\omega]):=\{\rho\in\Omega^2(M^{2n})\,|\,\rho^n/\omega^n>0, d\rho= 0, [\rho]=[\omega]\}.
   $$
   This is an open set in the Fr\'{e}chet space of all closed 2-forms on $M^{2n}$ representing the class $[\omega]$.
   Hence, the uniqueness of the solution to the Calabi-Yau equation on closed symplectic manifolds is equivalent to the following question
   (cf. McDuff-Salamon \cite[Chapter 13]{MS} or Salamon \cite{Sala}):
   \begin{que}\label{ques}
    Suppose that $(M^{2n},\omega)$ is a closed symplectic manifold of dimension $2n$.
   Is $\Psi(M^{2n},[\omega])$ connected?
   \end{que}

   \begin{rem}
   (1) McDuff \cite{McD} (see also McDuff-Salamon \cite[Theorem 9.7.4]{MS}) considered symplectic manifold $(T^2\times S^2\times S^2,\omega)$.
   Here the set $\Psi(T^2\times S^2\times S^2,[\omega])$ is disconnected.
   However, by an arbitrarily small pertubation of the cohomology class $[\omega]$, the two known distinct connected components of
   $\Psi(T^2\times S^2\times S^2,[\omega])$
   merge to a single connected component.

   (2) Question \ref{ques} is completely open for closed symplectic four-manifolds (cf. McDuff-Salamon \cite[Chapter 13]{MS}).
    The Donaldson geometric flow approach to uniqueness problem for hyperK\"{a}hler surfaces and general symplectic four manifolds.
    The Donaldson geometric flow was introduced by S. K. Donaldson in \cite{Donal2}.
    For a detailed exposition see R. Krom and D. A. Salamon \cite{KS}.
   \end{rem}

   The remainder of the paper is organized as follows:

   \vskip 6pt

   Section \ref{local CY} Preliminaries. This section is devoted to studying local theory of the Calabi-Yau equation on symplectic manifolds.

   \vskip 6pt

   Section \ref{global CY} Global theory of the Calabi-Yau equation.
   This section is devoted to studying global theory of the one-form type Calabi-Yau equation on symplectic manifolds.
   By using a (measurable) Lipschitz K\"{a}hler flat metric which is quasi isometric to the original almost K\"{a}hler metric off a Lebesgue measure zero subset,
    reduce the one-form type Calabi-Yau equation to
   the Monge-Amp\`{e}re equation on an open dense submanifold with respect to a (measurable) Lipschitz K\"{a}hler flat metric. Then we give an existence theorem for Calabi-Yau equation on symplectic manifold.
   Our approach is along the lines used by Yau to give a proof for Calabi-conjecture,
   and Theorems \ref{main the1},  \ref{main the2},  \ref{main the3} make up our proof of Theorem \ref{main result}.

   \vskip 6pt

    Section \ref{priori estimates} A priori estimates for the derivatives of the K\"{a}hler potential.
    By the method of Moser iteration, similar to Yau's proof of  Calabi conjecture,
    give the proofs of Theorem \ref{main the1} and Theorem \ref{main the2}.

   \vskip 6pt

     Section \ref{existence}
 The differentiability and the existence of solutions of the Calabi-Yau equation. In this section, we will prove Theorem \ref{main the3}.

   \vskip 6pt

 Appendix \ref{App A} Measurable K\"{a}hler flat structures on symplectic manifolds. In this appendix, we will construct a (measurable) Lipschitz K\"{a}hler flat metric on a closed almost K\"{a}hler manifold which is quasi isometric to the given  almost K\"{a}hler  metric off a Lebesgue measure zero subset.
 By the (measurable) Lipschitz K\"{a}hler flat metric, we will reduce the one-form type Calabi-Yau equation to the the Monge-Amp\`{e}re equation on an open dense submanifold.

   \vskip 6pt

   Appendix \ref{App B}  Calculations at a point. In this appendix, we will give several technical lemmas and several inequalities that will be useful.

 \section{Preliminaries}\label{local CY}
 \setcounter{equation}{0}
  This section is devoted to studying local theory of the Calabi-Yau equation on symplectic manifolds.
   Suppose that $(M^{2n},\omega)$ is a closed symplectic manifold of dimension $2n$.
   We can find an $\omega$-compatible almost complex structure $J$ on $M^{2n}$.
   Set $g_J(\cdot,\cdot):=\omega(\cdot,J\cdot)$ which is a $J$-invariant Riemannian metric on $M^{2n}$,
   then the triple $(\omega,J,g_J)$ is an almost K\"{a}hler structure on $M^{2n}$ (cf. \cite{Gau,MS}).
   The $\mathbb{C}$-extension of the almost complex structure $J$ on the complexified tangent bundle $TM^{2n}\otimes\mathbb{C}$
   is defined by
   $$J(X+\sqrt{-1}Y)=JX+\sqrt{-1}JY$$
   for $X,Y\in TM^{2n}$.
   Obviously, $J^2=-id$ on $TM^{2n}\otimes\mathbb{C}$.
   Therefore, $$TM^{2n}\otimes\mathbb{C}=T^{(1,0)}\oplus T^{(0,1)},$$
    where, $$T^{(1,0)}:=\{X\in TM^{2n}\otimes\mathbb{C}\,|\,JX=\sqrt{-1}X \}.$$
   $J$ induces an almost complex structure on $\Lambda^kT^*M^{2n}\otimes\mathbb{C}$.
   Hence
   $$
   \Lambda^2T^*M^{2n}\otimes\mathbb{C}=\Lambda^{(2,0)}\oplus \Lambda^{(1,1)}\oplus \Lambda^{(0,2)}
   $$
    where $\Lambda^{(p,q)}$ is the space of $(p,q)$-forms.
    Let $P_{2,0}$, $P_{1,1}$ and $P_{0,2}$ be the projections to $\Lambda^{(2,0)}$, $\Lambda^{(1,1)}$ and $\Lambda^{(0,2)}$, respectively.
    For each symplectic potential $\phi\in C^\infty(M^{2n},\mathbb{R})$, set
    \begin{eqnarray}
      \omega(\phi)&=& \omega+d(Jd\phi) \\
      &=& \tau(\phi)+ H(\phi)+\overline{\tau(\phi)}
    \end{eqnarray}
    where,
    \begin{eqnarray}
      \tau(\phi) &=& P_{2,0}(\omega(\phi)), \nonumber\\
       H(\phi)&=&   P_{1,1}(\omega(\phi)).
    \end{eqnarray}
   \begin{prop}\label{deco prop}
   (cf. Wang-Zhu \cite[Proposition 2.1]{WZ})
   Suppose that $(M^{2n},\omega,J,g_J)$ is a closed almost K\"{a}hler manifold of dimension $2n$.
   Then
   $$
   \omega(\phi)= \tau(\phi)+ H(\phi)+\overline{\tau(\phi)}
   $$
  tames $J$ if and only if $H(\phi)$ is strictly positive definite.
   \end{prop}
   \begin{rem}
   Observe that the taming $(1,1)$-form $H(\phi)$ in Proposition \ref{deco prop} is not necessarily closed,
   it is, when $J$ is integrable (K\"{a}hler case) in which case it coincides with $\omega(\phi)$ (that is, $\tau(\phi)=0$).
   \end{rem}

   In 1990's Gromov posed the following question to P. Delano\"{e} \cite{Dela}:
   \begin{que}\label{ques 1}
 Does there exists a smooth function  $\phi\in C^\infty(M^{2n},\mathbb{R})$ satisfying the following conditions?
 \begin{equation}\label{Gromov type CY equ}
   \left\{
  \begin{array}{ll}
    \omega(\phi)^n =\sigma,  \\

  \omega(\phi)=\omega+dJd\phi \,\,\, {\rm tames}\,\,\, J,
  \end{array}
  \right.
  \end{equation}
  where $\sigma$ is a given volume form in $[\omega^n]$.
   \end{que}
   However Delano\"{e} \cite{Dela}, Wang-Zhu \cite{WZ} showed the answer to Question \ref{ques 1} is negative.
 Hence, in general, we can not use almost K\"{a}hler potential to solve Calabi-Yau equation on closed symplectic manifolds.
 In this paper, we consider a generalized Calabi-Yau equation of Donaldson type on closed symplectic manifolds \cite{Donal2}.
       In order to study local theory of Calabi-Yau equation on closed symplectic manifold $(M^{2n},\omega)$,
       we consider the following question:
       \begin{que}\label{ques 2}
       Given a smooth function $F\in C^\infty(M^{2n};\mathbb{R})$ satisfying
        $$
       \int_{M^{2n}}e^F\omega^n= \int_{M^{2n}}\omega^n,
       $$
       and $F$ is very small.
       Can we find a small symplectic potential $\phi\in C^\infty(M^{2n};\mathbb{R})$ satisfying the following conditions?
       \begin{equation}\label{key equ}
         \omega(\phi)^n =e^F\omega^n,
       \end{equation}
       where $\omega(\phi)=\omega+d(Jd\phi)$ is a new symplectic form on $M^{2n}$.
       \end{que}
     Let $(M^{2n},\omega)$ be a closed symplectic manifold of dimension $2n$.
     Then there exists an almost K\"{a}hler structure $(\omega,J,g_J)$ on $M^{2n}$,
     where $J$ is an $\omega$-compatible almost complex structure on $M^{2n}$, $g_J(\cdot,\cdot)=\omega(\cdot,J\cdot)$.
      For solving Question \ref{ques 2}, we need some notations (cf. \cite{Dela,WZ}):
  \begin{defi}\label{large defi}
   Suppose that $(M^{2n},\omega,J,g_J)$ is a closed almost K\"{a}hler manifold of dimension $2n$.
   The sets $A$, $B$, $A_+$, $\bar{A}_+$ and $B_+$ are defined as follows:
   $$A:= \{\phi\in C^\infty(M^{2n},\mathbb{R})\,|\, \int_{M^{2n}}\phi\omega^n=0\};$$
   $$B:= \{f\in C^\infty(M^{2n},\mathbb{R})\,|\, \int_{M^{2n}}f\omega^n= \int_{M^{2n}}\omega^n\}; $$
    $$A_+:= A\cap\{\phi\in C^\infty(M^{2n},\mathbb{R})\,|\, \omega(\phi)\,\, {\rm tames}\,\, J\}; $$
    $$
    \bar{A}_+:= A\cap\{\phi\in C^\infty(M^{2n},\mathbb{R})\,|\, H(\phi)\,\, {\rm is}\,\, {\rm a}\,\,{\rm semi-positive}\,\, {\rm (1,1)}\,\,{\rm form}\};
    $$
     $$B_+:= B\cap\{f\in C^\infty(M^{2n},\mathbb{R})\,|\, f>0\}.$$
     $$\bar{B}_+:= B\cap\{f\in C^\infty(M^{2n},\mathbb{R})\,|\, f\geq0\}.$$
   \end{defi}

   Notice that $ A_+$ can be regarded as an open convex set of symplectic potential functions
   (analogue of K\"{a}hler potential functions) and $A_+$, $\bar{A}_+$ are convex sets in $A$.
   We define an operator $\mathcal{F}$ from $C^\infty(M^{2n},\mathbb{R})$ to $C^\infty(M^{2n},\mathbb{R})$ as follows:
   $$
   \phi\mapsto \mathcal{F}(\phi),
   $$
  where
  \begin{equation}\label{key equ1}
     \mathcal{F}(\phi)\omega^n=(\omega(\phi))^n.
  \end{equation}
  Restricting the operator $ \mathcal{F}$ to $A_+$ (resp.  $\bar{A}_+$),
  we get $\mathcal{F}(A_+)\subset B_+$ (resp.  $\mathcal{F}(\bar{A}_+)\subset \bar{B}_+$).
  Thus, the existence of a solution to Equation (\ref{Gro CY}) is equivalent to that the restricted operator
  \begin{equation}
    \mathcal{F}|_{A_+}: A_+\rightarrow  B_+
  \end{equation}
  is surjective.
  \begin{rem}
  (cf. Wang-Zhu \cite[Theorem 1.2]{WZ}) Suppose that $(M^{2n},\omega,J,g_J)$ is a closed almost K\"{a}hler manifold of dimension $2n$.
  Then the restricted operator $\mathcal{F}|_{A_+}: A_+\rightarrow  \mathcal{F}(A_+)$
  is diffeomorphism.
  Moreover, the restricted operator  $\mathcal{F}|_{A_+}: A_+\rightarrow  B_+$ is a surjectivity map if and only if $J$ is integrable.
   \end{rem}

   \vskip 6pt

   In order to solve Equation (\ref{key equ1}), we need to compute $\tau(\phi)$ and $H(\phi)$.
   Choose a local coordinate system and the second canonical connection on $M^{2n}$ with respect to $g_J$ (cf. \cite{Cha,Gau,WZ})

   There exists a local orthonormal basis $\{\varepsilon_1,J\varepsilon_1,\cdot\cdot\cdot,\varepsilon_n,J\varepsilon_n\}$ of $M^{2n}$.
   Set $$e_j=\frac{\varepsilon_j-\sqrt{-1}J\varepsilon_j}{\sqrt{2}}.$$
   Then $\{e_1,\cdot\cdot\cdot,e_n,\bar{e}_1,\cdot\cdot\cdot,\bar{e}_n\}$ is a local basis of $TM^{2n}\otimes\mathbb{C}$
   and $g_J(e_i,e_j)=g_J(\bar{e}_i,\bar{e}_j)=0$, $g_J(e_i,\bar{e}_j)=\delta_{ij}$.
    Obviously, $\{e_1,\cdot\cdot\cdot,e_n\}$ is a local basis of $T^{(1,0)}$.
    Let $\{\theta^1,\cdot\cdot\cdot,\theta^n\}$ be its dual basis.
    After complexification, $J$, $g_J$ and $\omega$ can be expressed locally as follows:
   \begin{equation}\label{local rep}
   \left\{
  \begin{array}{ll}
     J=\sqrt{-1}(\theta^\alpha\otimes e_\alpha-\bar{\theta}^\alpha\otimes \bar{e}_\alpha),  \\
     ~\\
     g_J= \sum^n_{\alpha=1}(\theta^\alpha\otimes\bar{ \theta}^\alpha+\bar{ \theta}^\alpha\otimes\theta^\alpha)  ,  \\
     ~\\
     \omega= \sqrt{-1}\sum^n_{\alpha=1}\theta^\alpha\wedge\bar{ \theta}^\alpha.
  \end{array}
  \right.
  \end{equation}
   Notice that the second canonical connection, $\nabla^1$, on $(M^{2n},\omega,J,g_J)$ with respect to $g_J$ is an affine connection satisfying
   \begin{equation}\label{canonical con}
     \nabla^1g_J=0= \nabla^1J
   \end{equation}
  and the $(1,1)$-component of its torsion vanishes \cite{Gau,Cha}.
  We must redo with $\nabla^1$ calculations analogous to those presented in Wang-Zhu \cite{WZ} or
  Delano\"{e} \cite{Dela}, Tosatti-Weinkove-Yau \cite{TWY}.
  From Equality (\ref{canonical con}), we obtain that
  $$
  \nabla^1: T^{(1,0)}\longrightarrow (T^*M^{2n}\otimes\mathbb{C})\otimes T^{(1,0)},
  $$
  $$
  e_\alpha\mapsto \omega^{\beta}_{\alpha}\otimes e_\beta.
  $$
  The second canonical connection, $\nabla^1$, induces
  $$
  \nabla^1: \Lambda^{(1,0)}\longrightarrow (T^*M^{2n}\otimes\mathbb{C})\otimes \Lambda^{(1,0)},
  $$
  $$
  \theta^\alpha\mapsto -\omega^{\alpha}_{\beta}\otimes \theta^\beta.
  $$
  Let $\phi\in C^\infty(M^{2n};\mathbb{R})$. Then
  \begin{equation}\label{diff equ}
    d\phi=\phi_\alpha\theta^\alpha+\bar{\phi}_\alpha\bar{\theta}^\alpha.
  \end{equation}
   Thus,
   \begin{equation}
     Jd\phi=\sqrt{-1}(\phi_\alpha\theta^\alpha-\bar{\phi}_\alpha\bar{\theta}^\alpha).
   \end{equation}
  Let
  \begin{equation}
   d\phi_\alpha-\phi_\beta\omega^{\beta}_{\alpha}=\phi_{\alpha\beta}\theta^\beta+\phi_{\alpha\bar{\beta}}\bar{\theta}^\beta.
  \end{equation}
   Then, equality $d^2\phi=0$ and Equation (\ref{diff equ}) implies that
   \begin{equation}
    d(\phi_\alpha\theta^\alpha)+d(\bar{\phi}_\alpha\bar{\theta}^\alpha)=0.
   \end{equation}
   Let
   \begin{equation}\label{torsion rep}
     \Theta^\alpha=d\theta^\alpha+\omega^{\alpha}_{\beta}\wedge\theta^\beta
   \end{equation}
  be the torsion of the second canonical connection $\nabla^1$.
   Thus, $\Theta^\alpha$ contains $(2,0)$ and $(0,2)$ components only \cite{Gau}.
   Therefore,
   \begin{equation}
      \Theta^\alpha=T^{\alpha}_{\beta\gamma}\theta^\beta\wedge\theta^\gamma+N^{\alpha}_{\bar{\beta}\bar{\gamma}}\bar{\theta}^\beta\wedge\bar{\theta}^\gamma,
   \end{equation}
  with $T^{\alpha}_{\beta\gamma}=-T^{\alpha}_{\gamma\beta}$ and $N^{\alpha}_{\bar{\beta}\bar{\gamma}}=-N^{\alpha}_{\bar{\gamma}\bar{\beta}}$.

     Indeed, the $(0,2)$ component of the torsion is independent of the choice of a metric can be regarded as the
     Nijenhuis tensor, $N$, of the almost complex structure $J$ \cite{Gau,MS}.
     From Equation (\ref{local rep}) and (\ref{torsion rep}), we obtain that for $\omega$ being a symplectic $2$-form,
     \begin{eqnarray*}
      0=d\omega  &=& d( \sqrt{-1}\sum^n_{\alpha=1}\theta^\alpha\wedge\bar{ \theta}^\alpha) \\
        &=&  \sqrt{-1}\sum^n_{\alpha=1}(d\theta^\alpha\wedge\bar{ \theta}^\alpha-\theta^\alpha\wedge d\bar{ \theta}^\alpha) \\
        &=& \sqrt{-1}\sum^n_{\alpha=1}(\Theta^\alpha\wedge\bar{ \theta}^\alpha-\theta^\alpha\wedge \bar{\Theta}^\alpha) \\
        &=& \sqrt{-1}\sum^n_{\alpha,\beta,\gamma=1}(T^{\alpha}_{\beta\gamma}\bar{\theta}^\alpha\wedge\theta^\beta\wedge\theta^\gamma
        -\overline{T^{\alpha}_{\beta\gamma}}\theta^\alpha\wedge\bar{\theta}^\beta\wedge\bar{\theta}^\gamma \\
        && + N^{\alpha}_{\bar{\beta}\bar{\gamma}}\bar{\theta}^\alpha\wedge\bar{\theta}^\beta\wedge\bar{\theta}^\gamma
        -\overline{N^{\alpha}_{\bar{\beta}\bar{\gamma}}}\theta^\alpha\wedge\theta^\beta\wedge\theta^\gamma).
     \end{eqnarray*}
  Hence, we have $T^{\alpha}_{\beta\gamma}=0$ and
  \begin{equation}
    N^{\alpha}_{\bar{\beta}\bar{\gamma}}+N^{\beta}_{\bar{\gamma}\bar{\alpha}}+N^{\gamma}_{\bar{\alpha}\bar{\beta}}=0.
  \end{equation}
  Combining with the above equations, we obtain that
  $$
  dJd\phi=2\sqrt{-1}(\phi_{\alpha\beta}\theta^\beta\wedge\theta^\alpha+\phi_{\alpha\bar{\beta}}\bar{\theta}^\beta\wedge\theta^\alpha
    +\phi_{\alpha}N^{\alpha}_{\bar{\beta}\bar{\gamma}}\bar{\theta}^\beta\wedge\bar{\theta}^\gamma).
  $$
  The reality of $dJd\phi$ implies that
  $$
  \phi_{\alpha\bar{\beta}}=\overline{\phi_{\beta\bar{\alpha}}}
  $$
  and
  $$
  \phi_{\alpha\beta}\theta^\beta\wedge\theta^\alpha=-\overline{\phi_{\alpha}N^{\alpha}_{\bar{\beta}\bar{\gamma}}\bar{\theta}^\beta\wedge\bar{\theta}^\gamma}.
  $$
  Set
   \begin{equation}
     h(\phi)(X,Y)=\frac{1}{2}(H(\phi)(X,JY)+H(\phi)(Y,JX)).
   \end{equation}
  Thus
  \begin{equation}\label{tau-phi equ}
   \tau(\phi)=-2\sqrt{-1}\bar{\phi}_\alpha \overline{N^{\alpha}_{\bar{\beta}\bar{\gamma}}}\theta^\beta\wedge\theta^\gamma,
  \end{equation}
   \begin{equation}\label{H-phi equ}
   H(\phi)=\sqrt{-1}(\delta_{\alpha\bar{\beta}}-2\phi_{\alpha\bar{\beta}})\theta^\alpha\wedge\bar{\theta}^\beta,
  \end{equation}
  \begin{equation}
   h(\phi)=(\delta_{\alpha\bar{\beta}}-2\phi_{\alpha\bar{\beta}})(\theta^\alpha\otimes\bar{\theta}^\beta+\bar{\theta}^\beta\otimes\theta^\alpha),
  \end{equation}
  $$
  dJd\phi=-2\sqrt{-1}(\bar{\phi}_\alpha \overline{N^{\alpha}_{\bar{\beta}\bar{\gamma}}}\theta^\beta\wedge\theta^\gamma
  +\phi_{\alpha\bar{\beta}}\theta^\alpha\wedge\bar{\theta}^\beta-\phi_\alpha N^{\alpha}_{\bar{\beta}\bar{\gamma}}\bar{\theta}^\beta\wedge\bar{\theta}^\gamma),
  $$
  where $\delta_{\alpha\bar{\beta}}=1$ if $\alpha=\beta$;
  otherwise, $\delta_{\alpha\bar{\beta}}=0$.
  Hence we have the following lemma:
  \begin{lem}
  (cf. \cite[Lemma 2.3]{WZ})
   Suppose that $(M^{2n},\omega,J,g_J)$ is an almost K\"{a}hler manifold of dimension $2n$ and
   $\phi$ is a smooth real function on $M^{2n}$.
   Then
   $$
   \omega(\phi)=\tau(\phi)+H(\phi)+\bar{\tau}(\phi),
   $$
   where $\tau(\phi)$ and $H(\phi)$ are given by Equation (\ref{tau-phi equ}) and (\ref{H-phi equ}), respectively.
  \end{lem}

  Let $[\frac{n}{2}]$  denote the integral part of a positive number $\frac{n}{2}$.
  Following \cite{Dela,WZ}, let us define operators $\mathcal{F}_j$ ($j=0,1,\cdot\cdot\cdot,[\frac{n}{2}]$) as follows:
  $$
  \mathcal{F}_j:C^\infty(M^{2n},\mathbb{R})\longrightarrow C^\infty(M^{2n},\mathbb{R}),
  $$
  $$
  \phi\mapsto \mathcal{F}_j(\phi),
  $$
  where
  \begin{equation}\label{component rep}
   \mathcal{F}_j(\phi)\omega^n=\frac{n!}{j!^2(n-2j)!}(\tau(\phi)\wedge \bar{\tau}(\phi))^j \wedge H(\phi)^{n-2j}.
  \end{equation}
  Then
  \begin{equation}\label{rep of F}
    \sum_{j=0}^{[\frac{n}{2}]}\mathcal{F}_j(\phi)\omega^n=(\tau(\phi)+H(\phi)+\bar{\tau}(\phi))^n=\omega(\phi)^n=\mathcal{F}(\phi)\omega^n.
  \end{equation}
  We have the following result:
  \begin{prop}\label{F inequ}
  (cf. \cite[Proposition 5]{Dela} or \cite[Proposition 2.4]{WZ})
  Suppose that $(M^{2n},\omega,J,g_J)$ is a closed almost K\"{a}hler manifold of dimension $2n$.
  Then
  $$
  \mathcal{F}(\phi)=\mathcal{F}_0(\phi)+\cdot\cdot\cdot+\mathcal{F}_{[\frac{n}{2}]}(\phi).
  $$

  (1) If $\phi\in A_+$, then $\mathcal{F}_0(\phi)>0$ and $\mathcal{F}_j(\phi)\geq0$ for $j=1,\cdot\cdot\cdot,[\frac{n}{2}]$,
  hence $\mathcal{F}(\phi)>0$;

  (2) if $\phi\in \bar{A}_+$,  $\mathcal{F}_0(\phi)\geq0$, then
  $\mathcal{F}_j(\phi)\geq0$ for $j=1,\cdot\cdot\cdot,[\frac{n}{2}]$;

  (3) define $$B_{\varepsilon}(0):=\{\phi\in A\,|\, \|\phi\|_{C^2}\leq \varepsilon\},$$
  where $C^2$ is $C^2$-norm introduced by the metric $g_J$;
  if $\varepsilon<< 1$, then $B_{\varepsilon}(0)\subset A_+$.
  \end{prop}

  For any $\phi\in A_+$, it is easy to see that the tangent space at $\phi$, $T_{\phi}A_+$, is $A$.
  For $u\in T_{\phi}A_+=A$, define $L(\phi)(u)$ by
  $$
  L(\phi)(u)=\frac{d}{dt}\mathcal{F}(\phi+tu)|_{t=0}.
  $$
  Indeed, by simple calculation, we can obtain
  \begin{eqnarray*}
     L(\phi)(u)&=& \frac{n\omega(\phi)^{n-1}\wedge(dJdu)}{\omega^n} \\
     &=& \frac{n\omega(\phi)^{n-1}\wedge(P_{1,1}(dJdu)+P_{2,0}(dJdu)+P_{0,2}(dJdu))}{\omega^n}.
  \end{eqnarray*}
 In the follows we will concretely calculate $L(\phi)(u)$ according to the parity of $n$:

 If $n=2m+1$,
 \begin{eqnarray*}
    \omega(\phi)^{n-1}&=&\omega(\phi)^{2m}=(\tau(\phi)+H(\phi)+\bar{\tau}(\phi))^{2m} \\
    &=&\sum^m_{j=0}\frac{(2m)!}{j!^2(2m-2j)!}(\tau(\phi)\wedge \bar{\tau}(\phi))^j \wedge H(\phi)^{2m-2j}  \\
    &=& H(\phi)^{2m}+\sum^m_{j=1}\frac{(2m)!}{j!^2(2m-2j)!}(\tau(\phi)\wedge \bar{\tau}(\phi))^j \wedge H(\phi)^{2m-2j} .
 \end{eqnarray*}
  Then
  \begin{eqnarray*}
     L(\phi)(u)&=&\{n\omega(\phi)^{n-1}\wedge(P_{1,1}(dJdu)+P_{2,0}(dJdu)+P_{0,2}(dJdu))\}/ \omega^n \\
       &=& \{n[H(\phi)^{2m}+\sum^m_{j=1}\frac{(2m)!}{j!^2(2m-2j)!}(\tau(\phi)\wedge \bar{\tau}(\phi))^j \wedge H(\phi)^{2m-2j}] \\
       && \,\,\,\,\,\,\,\,\,\,\,\,\,\,\,\,\,\,\,\,\,\,\,\,\wedge P_{1,1}(dJdu)\}/\omega^n.
     \end{eqnarray*}
     Notice that $H(\phi)^{2m}$ is a positive $(n-1,n-1)$-form
    and $$\sum^m_{j=1}\frac{(2m)!}{j!^2(2m-2j)!}(\tau(\phi)\wedge \bar{\tau}(\phi))^j \wedge H(\phi)^{2m-2j}$$
  is a non-negative $(n-1,n-1)$-form.
  Hence, $L(\phi)$ is a linear elliptic differential operator of second order without first term and zero-th term.

  If $n=2m$, $m>1$,
 \begin{eqnarray*}
    \omega(\phi)^{n-1}&=&\omega(\phi)^{2m-1}=\omega(\phi)^{2m-2}\wedge\omega(\phi)  \\
    &=&\{H(\phi)^{2m-2}+\sum^{m-1}_{j=1}\frac{(2m-2)!}{j!^2(2m-2-2j)!}(\tau(\phi)\wedge \bar{\tau}(\phi))^j \wedge H(\phi)^{2m-2-2j} \} \\
    && \wedge(H(\phi)+\tau(\phi)+\bar{\tau}(\phi))\\
     &=&\{H(\phi)^{2m-1}+\sum^{m-1}_{j=1}\frac{(2m-2)!}{j!^2(2m-2-2j)!}(\tau(\phi)\wedge \bar{\tau}(\phi))^j \wedge H(\phi)^{2m-1-2j}\} \\
     &&+\{H(\phi)^{2m-2}+\sum^{m-1}_{j=1}\frac{(2m-2)!}{j!^2(2m-2-2j)!}(\tau(\phi)\wedge \bar{\tau}(\phi))^{j} \wedge H(\phi)^{2m-2-2j}\}\\
      &&\,\,\,\,\,\,\,\,\,\,\,\,\,\,\,\,\,\,\wedge(\tau(\phi)+\bar{\tau}(\phi)).
 \end{eqnarray*}
   Then
  \begin{eqnarray*}
     &&  L(\phi)(u)\\
    &=& \{n[H(\phi)^{n-1}+\sum^{\frac{n}{2}-1}_{j=1}\frac{(n-2)!}{j!^2(n-2-2j)!}(\tau(\phi)\wedge \bar{\tau}(\phi))^j \wedge H(\phi)^{n-1-2j}]\wedge P_{1,1}(dJdu)\}/\omega^n\\
       &&+\{n[H(\phi)^{n-2}+\sum^{\frac{n}{2}-1}_{j=1}\frac{(n-2)!}{j!^2(n-2-2j)!}(\tau(\phi)\wedge \bar{\tau}(\phi))^{j} \wedge H(\phi)^{n-2-2j}]\wedge(\tau(\phi)+\bar{\tau}(\phi)) \\
      &&  \,\,\,\,\,\,\,\,\,\,\,\,\,\,\,\,\,\,\,\,\,\,\,\,\,\,\,\,\,\,\,\,\,\,\,\,\wedge P_{2,0+0,2}\}/\omega^n,
     \end{eqnarray*}
  where $$P_{2,0+0,2}=P_{2,0}(dJdu)+P_{0,2}(dJdu).$$
  Here $$P_{2,0}(dJdu)=\overline{-u_{\alpha}N^{\alpha}_{\bar{\beta}\bar{\gamma}}\bar{\theta}^\beta\wedge\bar{\theta}^\gamma},$$
  $$P_{0,2}(dJdu)=\overline{P_{2,0}(dJdu)}$$ are first term of $L(\phi)(u)$.
  Notice that
  $$
  H(\phi)^{n-1}+\sum^{\frac{n}{2}-1}_{j=1}\frac{(n-2)!}{j!^2(n-2-2j)!}(\tau(\phi)\wedge \bar{\tau}(\phi))^j \wedge H(\phi)^{n-1-2j}
  $$ is
  a positive $(n-1,n-1)$-form, and
  $$
  [H(\phi)^{n-2}+\sum^{\frac{n}{2}-1}_{j=1}\frac{(n-2)!}{j!^2(n-2-2j)!}(\tau(\phi)\wedge \bar{\tau}(\phi))^{j} \wedge H(\phi)^{n-2-2j}]\wedge(\tau(\phi)+\bar{\tau}(\phi))
  $$
  is a $(n-2,n)+(n,n-2)$-form.
   Hence, $L(\phi)$ is a linear elliptic differential operator of second order without zero-th term.

  In summary,
  it is easy to see that $L(\phi)$ is a linear elliptic differential operator of second order without zero-th term.

  For a local theory of the Calabi-Yau equation, we should study the tangent map of the restricted operator $\mathcal{F}|_{A_+}$.
  \begin{lem}\label{nozero map}
  (cf. \cite[Proposition 1]{Dela} or \cite[Lemma 2.5]{WZ})
   Suppose that $(M^{2n},\omega,J,g_J)$ is a closed almost K\"{a}hler manifold of dimension $2n$.
   Then the restricted operator
   $$ \mathcal{F}:A_+\longrightarrow B_+ $$
   is elliptic type on $A_+$.
   Moreover, the tangent map,
   $d\mathcal{F}(\phi)=L(\phi)$, of $\mathcal{F}$ at $\phi\in A_+$ is a linear elliptic differential operator of second order without zero-th term.
  \end{lem}

  Suppose that $\phi_0\in A_+$.
  By Proposition \ref{F inequ}, there exists a small neighborhood $U(\phi_0)$ in $A$ such that $U(\phi_0)\subset A_+$,
 and if $\phi_1\in A$, $\varepsilon<< 1$, we can obtain that
 $\mathcal{F}_0(\phi_0+\varepsilon\phi_1)>0$ and $\mathcal{F}_j(\phi_0+\varepsilon\phi_1)\geq0$.
 Indeed, by (\ref{component rep}), we have
           \begin{eqnarray*}
              &&   \mathcal{F}_0(\phi_0+\varepsilon\phi_1)\omega^n  \\
             &=& H(\phi_0+\varepsilon\phi_1)^n   \\
              &=&P_{1,1}(\omega(\phi_0+\varepsilon\phi_1))^n  \\
              &=&P_{1,1}(\omega+dJd\phi_0+\varepsilon dJd\phi_1)^n    \\
              &=& [P_{1,1}(\omega+dJd\phi_0)+ P_{1,1}(\varepsilon dJd\phi_1)]^n   \\
              &=& [H(\phi_0)+\varepsilon P_{1,1}(dJd\phi_1)]^n  \\
              &=&H(\phi_0)^n+\varepsilon H(\phi_0)^{n-1}\wedge P_{1,1}(dJd\phi_1)+\varepsilon^2 H(\phi_0)^{n-2}\wedge P_{1,1}(dJd\phi_1)^2+\cdot\cdot\cdot.
           \end{eqnarray*}
           Notice that $H(\phi_0)>0$ and $\varepsilon<< 1$, it is easy to get $\mathcal{F}_0(\phi_0+\varepsilon\phi_1)>0$.
           Then by Proposition \ref{F inequ}, we can obtain $\mathcal{F}_j(\phi_0+\varepsilon\phi_1)\geq0$.

  By Lemma \ref{nozero map},
  $d\mathcal{F}(\phi)=L(\phi)$ for any $\phi\in A_+$,
  where $L(\phi)$ is a linear elliptic differential operator of second order without zero-th term.
  Making use of maximal principle of Hopf (for example, see Gauduchon \cite{Gau0}),
  we obtain that the kernel of $L(\phi)$ consists of constant functions, thus $\ker L(\phi)=0$ in $A$ (cf. Proof of Theorem 1.2 in \cite{WZ}).
  Therefore $\mathcal{F}:A_+\rightarrow B_+$ is an injectivity map, and by nonlinear analysis, we obtain the following result:
  \begin{theo}\label{diff map}
  (cf. \cite[Theorem 2]{Dela} or \cite[Proposition 2.6]{WZ})
  Suppose that $(M^{2n},\omega,J,g_J)$ is a closed almost K\"{a}hler manifold of dimension $2n$.
  Then the restricted operator
  $$
  \mathcal{F}:A_+\rightarrow  \mathcal{F}(A_+)\subset B_+
  $$
  is a diffeomorphic map.
  \end{theo}

  \begin{rem}\label{diff rem}

  From the definition of $ \mathcal{F}(\phi)$ (cf. Equality (\ref{rep of F})), we see that
  $$
  \mathcal{F}(0)\omega^n=(\omega(0))^n=\omega^n.
  $$
  Thus, $\mathcal{F}(0)=1$.
  Therefore, Proposition \ref{F inequ}, Theorem \ref{diff map} and Theorem $C4$ in \cite[\S 6.2]{Joc} imply that the solution to Equation (\ref{Gromov type CY equ})
  (that is the Calabi-Yau equation of Gromov type)
  exists and is unique if $\sigma$ is a small perturbation of $\omega^n$.
  \end{rem}

  Notice that Theorem 1.2 in \cite{WZ} shows that if the almost K\"{a}hler structure $(\omega,J,g_J)$ is not  K\"{a}hler, then the restricted operator
  $$
  \mathcal{F}:A_+\longrightarrow\mathcal{F}(A_+)\subset B_+
  $$
  is not surjective.
  More precisely, by Proposition 3.8 and Lemma 3.9 in \cite{WZ}, we can construct a smooth function $\phi_0\in A$
  such that $\mathcal{F}_1(\phi_0)>0$,  $\mathcal{F}_0(\phi_0)\geq0$.
  Thus $\mathcal{F}(\phi_0)>0$, that is, $\phi_0$ in $\partial A_+$.
  Therefore $\mathcal{F}(\phi_0)\in B_+\setminus \mathcal{F}(A_+)$.
 More details, see Proof of Theorem 1.2 in \cite{WZ}.
 If $(\omega,J,g_J)$ is a  K\"{a}hler structure, then $\mathcal{F}$ is a surjectivity map that is the famous Yau's theorem
 \cite{Yau}.
 Hence, Equation (\ref{Gromov type CY equ}) (resp. Equation (\ref{key equ})) has a unique solution for each $\sigma\in[\omega^n]$ if and only if $J$ is integrable. Notice that if $J$ is integrable, then $\omega(\phi)=H(\phi)$ and $\tau(\phi)=0$.

   \begin{rem}\label{diff rem2}
   For any $\phi\in\bar{ A}_+$ and $\omega(\phi)^n/\omega^n>0$, then $H(\phi)\not\equiv 0$ and $H(\phi)\geq 0$, otherwise it implies that
   $$
   0<\int_{M^{2n}}\omega^n=\int_{M^{2n}}\omega^{n-1}\wedge\omega(\phi)=\int_{M^{2n}}\omega^{n-1}\wedge(\tau(\phi)+\bar{\tau}(\phi))=0.
   $$
   This is giving a contradiction.
   \end{rem}

   By Theorem \ref{diff map}, Remark \ref{diff rem} and \ref{diff rem2}, we have local existence theorem for solution of one-form type Calabi-Yau equation
   on closed symplectic manifold $(M^{2n},\omega)$.
   \begin{theo}\label{main theo}
   Suppose that $(M^{2n},\omega,J,g_J)$ is a closed almost K\"{a}hler manifold.
   Let $F\in C^{\infty}(M^{2n},\mathbb{R})$ satisfying
   $$
   \int_{M^{2n}}e^F\omega^n=\int_{M^{2n}}\omega^n,
   $$
   and
   $$
   F\in\mathcal{F}(A_+)\subset B_+.
   $$
   Then, there exists a smooth one-form,
   $$
   a(F)=Jd\phi(F)\in\Omega^1(M^{2n})
   $$
   satisfying
   $$
   (\omega+da(F))^n=e^F\omega^n.
   $$
   Therefore, $\omega+da(F)$ is a new symplectic form on $M^{2n}$ which is cohomologous to $\omega$.

   \end{theo}

 \section{Global theory of the Calabi-Yau equation}\label{global CY}
 \setcounter{equation}{0}
           This section is devoted to studying global theory of the Calabi-Yau equation on symplectic manifolds.

   Suppose that $(M^{2n},\omega)$ is a closed symplectic manifold of dimension $2n$.
  Given $F\in C^\infty(M^{2n},\mathbb{R})$ satisfying
  $$
  \int_{M^{2n}}e^F\omega^n= \int_{M^{2n}}\omega^n,
  $$
   we want to find a $1$-form $a\in\Omega^1(M^{2n})$ such that $a$ is a solution of the following
    Calabi-Yau equation:
    \begin{equation}\label{3CY equ}
     (\omega+da)^n=e^F\omega^n.
    \end{equation}

   For $(M^{2n},\omega)$, we can find an almost K\"{a}hler structure, $(\omega,J,g_J)$, on $M^{2n}$ \cite{MS}.
   Where $J$ is an $\omega$-compatible almost complex structure on $M^{2n}$, $g_J(\cdot,\cdot):=\omega(\cdot,J\cdot)$.
   If $a$ is very small, then
    $a=Jd\phi$, where $\phi\in C^\infty(M^{2n},\mathbb{R})$ satisfying that $$\int_{M^{2n}}\phi\omega^n=0,$$
   then Equation (\ref{3CY equ}) is reduced to the following complex Monge-Amp\`{e}re equation:
   \begin{equation}\label{complex MA equ}
     (\omega+dJd\phi)^n=e^F\omega^n.
   \end{equation}
   It is easy to see that if $J$ is integrable,
   then $$dJd\phi=2\sqrt{-1}\partial_J\bar{\partial}_J\phi, \,\,\,dJd(\Omega^0(M^{2n}))=d\Omega^1(M^{2n}),$$ otherwise, $dJd(\Omega^0(M^{2n}))\subsetneqq d\Omega^1(M^{2n})$.
   In Section \ref{local CY}, we point out that the nonexistence of the complex Monge-Amp\`{e}re equation (that is the Calabi-Yau equation of Gromov type)
   and local existence theorem for one-form type Calabi-Yau equation (\ref{3CY equ}) (Theorem \ref{main theo}).
   For studying global theory of the one-form type Calabi-Yau equation on closed symplectic manifold $(M^{2n},\omega)$, by
   using Darboux's coordinate subatlas \cite{MS},
   we can deform an $\omega$-compatible almost K\"{a}hler structure to an $\omega$-compatible (measurable) Lipschitz K\"{a}hler flat structure off a Lebesgue measure zero subset,
   then the Calabi-Yau equation on an open dense submanifold of symplectic manifold can be reduced to the complex
   Monge-Amp\`{e}re equation with the (measurable) Lipschitz K\"{a}hler flat structure (cf. Proposition \ref{CYequ Lipschitz} in \ref{App A}).
   Thus, we can use Yau's method \cite{Yau2} to solve the Calabi-Yau equation on symplectic manifolds.

   \vskip 6pt

   Suppose that $(M^{2n},\omega,J,g_J)$ is a closed almost K\"{a}hler manifold of dimension $2n$.
   If there exists a $1$-form $a\in\Omega^1(M^{2n})$ (one may assume that $a$ is very small)
   such that $\omega(a)=\omega+da$ is a new symplectic form on $M^{2n}$,
   then there exists a new almost K\"{a}hler structure $(\omega(a),J(a),g_{J(a)})$ on $M^{2n}$,
   where $J(a)$ is an $\omega(a)$-compatible almost complex structure on $M^{2n}$,
   $g_{J(a)}(\cdot,\cdot)=\omega(a)(\cdot,J(a)\cdot)$.
   Notice that in general, if $a$ is not very small, $J$ is not tamed by the new symplectic form $\omega(a)$.
   Let $e^F=\omega(a)^n/\omega^n$, then $a,F$ satisfy the Calabi-Yau equation $\omega(a)^n=e^F\omega^n$.
   Let $h(1)=g_J-\sqrt{-1}\omega$ (resp. $h(a)(1)=g_{J(a)}-\sqrt{-1}\omega(a)$) be an almost Hermitian metric with respect to $(\omega,J,g_J)$
   (resp. $(\omega(a),J(a),g_{J(a)})$) on $M^{2n}$.
   By using the global deformation of almost complex on $M^{2n}$ off a Lebesgue measure zero subset,
   one can deform $h(1)$ (resp. $h(a)(1)$) to a (measurable) Lipschitz K\"{a}hler flat
   metric $h(0)$ (resp. $h(a)(0)$) on $M^{2n}$ off a Lebesgue measure zero subset.
   More precisely, by Proposition \ref{CYequ Lipschitz}, one can define a partition, $P(\omega(a),J(a),J)$, of $M^{2n}$, and
    find a finite Darboux's coordinate subatlas $\mathcal{V}(a)=\{V_k(a)\}_{1\leq k\leq N(a)}$.
   On each $V_k(a)$, $1\leq k\leq N(a)$, choose a holomorphic coordinates $\{z_1,\cdot\cdot\cdot,z_n\}$ such that
   the Calabi-Yau equation (\ref{3CY equ}) is reduced to the complex Monge-Amp\`{e}re equation (cf. Appendix \ref{App A})
   with respect to the K\"{a}hler flat metric $h(0)$,
   \begin{equation}\label{mat equ1}
     \det(h_{i\bar{j},k}+\frac{\partial^2\varphi_k(a)}{\partial z_i\partial \bar{z}_j})=e^F\det(h_{i\bar{j},k}),
   \end{equation}
  where $\varphi_k(a)$ is the K\"{a}hler potential on $V_k(a)$, $1\leq k\leq N(a)$,
  \begin{equation}\label{mat equ2}
    -\sqrt{-1}da|_{V_k(a)}=(\frac{\partial^2\varphi_k(a)}{\partial z_i\partial \bar{z}_j})dz_i\wedge d\bar{z}_j
  \end{equation}
    is a complex matrix,
    $h_{i\bar{j},k}$ and
    \begin{equation}\label{mat equ3}
    h_{i\bar{j},k}(a)=  h_{i\bar{j},k}+(\frac{\partial^2\varphi_k(a)}{\partial z_i\partial \bar{z}_j})
  \end{equation}
     are two Hermitian matrices with respect to $h(0)$ and $h(a)(0)$, respectively (cf. Lemma \ref{Joc lemma 1'}).

     \vskip 6pt

     For symplectic forms $\omega$ and $\omega(a)$ on $M^{2n}$, in Proposition \ref{CYequ Lipschitz} ( in Appendix \ref{App A}),
     for a partition, $P(\omega(a),J(a),J)$, of $M^{2n}$,
     we have chosen a finite Darboux's coordinate subatlas
       $\mathcal{V}(a)=\{V_k(a)\}_{1\leq k\leq N(a)}$.
     Set $W_1(a)=V_1(a), T_1(a)=V_1(a)$, and as $2\leq k\leq N(a) $,
    $$
   W_k(a)=V_k(a)-\cup^{k-1}_{j=1}\overline{V}_j(a),\,\,\,
   T_k(a)=V_k(a)-\cup^{k-1}_{j=1}V_j(a).
   $$
  Then, $W_k(a)\subset T_k(a)\subset V_k(a)$, is an open subset of $M^{2n}$.
  $\{T_k(a)\}_{1\leq k\leq N(a)}$ is a disjoint partition of $M^{2n}$, that is,
   $M^{2n}=\cup^{N(a)}_{k=1}T_k(a)$. Let
  $$\mathring{M}^{2n}(a,P)=\cup^{N(a)}_{k=1}W_k(a).$$
   $M^{2n}\setminus \mathring{M}^{2n}(a,P)$ has Hausdorff dimension $\leq2n-1$ with Lebesgue measure zero.
   Define, by Equation (\ref{mat equ2}),
   \begin{equation}\label{}
     \varphi(a)=\sum^{N(a)}_{k=1}\varphi_k(a)|_{ T_k(a)}.
   \end{equation}
   We require
   \begin{equation}\label{}
   \int_{\mathring{M}^{2n}(a,P)}\varphi(a)\omega^n= \int_{M^{2n}}\varphi(a)\omega^n=0.
   \end{equation}
   Therefore, by partion $P(a)$, we construct a (measurable) Lipschitz K\"{a}hler flat structure $(\omega,J_0,g_0)$
   on $M^{2n}$ which is Lipschitz equivalent to almost K\"{a}hler structure $(\omega,J_1,g_1)=(\omega,J,g_J)$
    (that is, Lipschitz condition (\ref{metric equ1})-(\ref{metric equ0})).
   Restricted to the open dense submanifold $\mathring{M}^{2n}(a,P)$ of $M^{2n}$,
   $g_1$ and $g_0$ are quasi isometry (cf. (\ref{metric equ}) ).
   In fact $g_0$ is a measurable  K\"{a}hler metric which is regarded as a singular K\"{a}hler metric.
   We also define two H\"{o}lder spaces $C^{k,\alpha}_i$, $i=0,1$ on  $\mathring{M}^{2n}(a,P)$, for $\alpha\in [0,1)$ and
   a nonnegative integer $k$ with respect to metrics $g_i$, $i=0,1$, respectively.
   By (\ref{neq estimate}), $C^{k,\alpha}_i$, $i=0,1$, are with equivalent norms on  $\mathring{M}^{2n}(a,P)$ and $C^{k,\alpha}_0$ can be extended to $M^{2n}$.

  In order to solve Equation (\ref{3CY equ}), we begin by stating three results,
  Theorem \ref{main the1}-\ref{main the3} (cf. Joyce \cite[Theorem C1-C3 in \S 6.2]{Joc}) which will be proved later.
  These are the three main theorems which make up our proof of Theorem \ref{main result}.

  Using (\ref{mat equ1})-(\ref{mat equ3}) we have the following theorems:
  \begin{theo}\label{main the1}
  (cf. \cite[Theorem C1 in \S 6.2]{Joc})
 Let $(M^{2n},\omega,J,g_J)$, $(M^{2n},\omega(a),J(a),g_{J(a)})$ be closed almost K\"{a}hler manifolds of dimension $2n$
 with almost K\"{a}hler structures $(\omega,J,g_J)$, $(\omega(a),J(a),g_{J(a)})$ on $M^{2n}$,
 where $\omega(a)=\omega+da$, $a\in\Omega^1(M^{2n})$, and $$\omega(a)^n=Ae^f\omega^n,\,\,\,A>0.$$
  Let $Q_1\geq 0$.
  Then there exist $Q_2,Q_3,Q_4\geq 0$ depending only on $M^{2n}$, $\omega$, $J$ and $Q_1$, such that the following holds:

  Suppose that $$f\in C^3(M^{2n},g_1),\, \varphi(a)\in C^5(\mathring{M}^{2n}(a,P),g_1), \,a\in C^4(T^*M^{2n},g_1),\, \|f\|_{C^3_1}\leq Q_1$$
   satisfying equations
   $$\int_{M^{2n}}\varphi(a)\omega^n=0,\,\,\,
   -\sqrt{-1}da|_{\mathring{M}^{2n}(a,P)}=\partial\bar{\partial}\varphi(a) $$
  and
  $$
  (\omega+da)^n=Ae^f\omega^n,\,\,\, \int_{M^{2n}}\omega^n=\int_{M^{2n}}Ae^f\omega^n ,\,\,\,A>0.
  $$
  Then $$\|\varphi(a)\|_{C^0_1(\mathring{M}^{2n}(a,P))}\leq Q_2, \,\|da\|_{C^0_1(\Lambda^2T^*M^{2n})}=\|dd^c\varphi(a)\|_{C^0_1}\leq Q_3,$$
  and $$\|\nabla_1da\|_{C^0_1(T^*\otimes \Lambda^2T^*M^{2n})}=\|\nabla_1dd^c\varphi(a)\|_{C^0_1}\leq Q_4.$$
   Where $\nabla_1$ is the Levi-Civita connection with respect to the metric $g_1=g_J$,
   $C^k_1$ norm on $M^{2n}$ is defined by the metric $g_1$ and its Levi-Civita connection $\nabla_1$.
  \end{theo}

  \begin{theo}\label{main the2}
  (cf. \cite[Theorem C2 in \S 6.2]{Joc})
   Suppose that $(M^{2n},\omega,J,g_J)$ is a closed almost K\"{a}hler manifold of dimension $2n$.
   Let $Q_1,Q_2,Q_3,Q_4\geq 0$ and $\alpha\in (0,1)$. Then there exists $Q_5,Q_6> 0$ depending only on
   $M^{2n},\omega,J,Q_1,Q_2,Q_3,Q_4$ and $\alpha$, such that the following holds:

   Suppose $$f\in C^{3,\alpha}(M^{2n},g_1), \,\varphi(a)\in C^5(\mathring{M}^{2n}(a,P),g_1),\,\, a\in C^{4,\alpha}(T^*M^{2n},g_1) \,\, and \,\,A>0$$ satisfy
   $$-\sqrt{-1}da|_{\mathring{M}^{2n}(a,P)}=\partial\bar{\partial}\varphi(a),$$
   $$ (\omega+da)^n=Ae^f\omega^n,\,\,\,\int_{M^{2n}}\omega^n=\int_{M^{2n}}Ae^f\omega^n $$
   and the inequalities
   $$\|f\|_{C^{3,\alpha}_1}\leq Q_1,\,\,\, \|\varphi(a)\|_{C^0_1}\leq Q_2,\,\,\,\|da\|_{C^0_1}=\|dd^c\varphi(a)\|_{C^0_1}\leq Q_3,$$
   $$\|\nabla_1(da)\|_{C^0_1}=\|\nabla_1dd^c\varphi(a)\|_{C^0_1}\leq Q_4.$$
  Then
  $$\varphi(a)\in C^{5,\alpha}(\mathring{M}^{2n}(a,P),g_1), \,\,\,\|\varphi(a)\|_{C^{5,\alpha}_1}\leq Q_5$$
   and $$ \|a\|_{C^{4,\alpha}_1}, \|da\|_{C^{3,\alpha}_1}=\|dd^c\varphi(a)\|_{C^{3,\alpha}_1}\leq Q_6.$$
  Also, if $f\in C^{k,\alpha}(M^{2n},g_1)$ for $k\geq3$ then
  $$\varphi(a)\in C^{k+2,\alpha}(\mathring{M}^{2n}(a,P),g_1), \,\,\, a\in C^{k+1,\alpha}(T^*M^{2n},g_1),\,\,\, da\in C^{k,\alpha}(\Lambda^2T^*M^{2n},g_1)$$
   and if $f\in C^\infty(M^{2n},g_1)$,
  then
  $$\varphi(a)\in C^\infty(\mathring{M}^{2n}(a,P),g_1),\,\,a\in C^\infty(T^*M^{2n},g_1),\,\,da\in C^\infty(\Lambda^2T^*M^{2n},g_1).$$
  \end{theo}

  \begin{theo}\label{main the3}
  (cf. \cite[Theorem C3 in \S 6.2]{Joc})
  Let $(M^{2n},\omega,J,g_J)$ be a closed almost K\"{a}hler manifold of dimension $2n$.
  Fix $\alpha\in(0,1)$, and suppose that $$f'\in C^{3,\alpha}(M^{2n},g_1),\,da'\in C^{3,\alpha}(\Lambda^2T^*M^{2n},g_1),\,\varphi (a')\in C^{5,\alpha}(\mathring{M}^{2n}(a',P'), g_1)\,\,
  and \,\, A'>0$$
   satisfy the equations
    $$
    \int_{M^{2n}}\varphi(a')\omega^n=0, \,\,\,
   -\sqrt{-1}da'|_{\mathring{M}^{2n}(a',P')}=\partial\bar{\partial}\varphi(a')
   $$
  and
  $$
  (\omega+da')^n=A'e^{f'}\omega^n.
  $$
  Then whenever $$f\in C^{3,\alpha}(M^{2n},g_J)\,\, and \,\,\|f-f'\|_{C^{3,\alpha}_1}$$ is sufficiently small, there exists a $1$-form $a$ on $M^{2n}$
  satisfying that
   $$
  (\omega+da)^n=Ae^f\omega^n,\,\,\,A>0
  $$
  and $da\in C^{3,\alpha}(\Lambda^2T^*M^{2n},g_1)$.
  We can find a function $\varphi(a)$ on $\mathring{M}^{2n}(a,P)$ such that $$-\sqrt{-1}da|_{\mathring{M}^{2n}(a,P)}=\partial\bar{\partial}\varphi(a)\,\,\,
  {\rm and}\,\,\, \varphi(a)\in C^{5,\alpha}(\mathring{M}^{2n}(a,P),g_1).$$
  \end{theo}
  Theorems \ref{main the1}-\ref{main the3} will be proved later.

  \begin{rem}
  It is still an open problem, does there exist at most one solution $a$ such that on $M^4$
   $$
  (\omega+da)^2=Ae^f\omega^2,\,\,\,A>0,\,\,\,\int_{M^{2n}}\omega^2=\int_{M^{2n}}Ae^f\omega^2.
  $$
  In general, the above question has no uniqueness.
  It is possible to approach to this problem by using Donaldson geometric flow (cf. Donaldson \cite{Donal2} or Krom-Salamon \cite{KS}.)
  \end{rem}
    In the remainder of this section, using Theorems \ref{main the1}-\ref{main the3} we now prove Theorem \ref{main result}.
    \begin{defi}\label{defi equ}
    Let $(M^{2n},\omega)$ be a closed symplectic manifold of dimension $2n$.
    Then by $\omega$, there exists an almost K\"{a}hler structure $(\omega,J,g_J)$ on $M^{2n}$, where $J$ is an $\omega$-compatible
    almost complex structure, $g_J(\cdot,\cdot)=\omega(\cdot,J\cdot)$.
    Fix $\alpha\in(0,1)$ and $f\in C^{3,\alpha}(M^{2n},g_J)$.
    Define $S$ to the set of all $s\in[0,1]$ for which there exists $da_s\in C^{3,\alpha}(\Lambda^2T^*M^{2n},g_J)$ and $A_s>0$,
    where $a_s\in C^{4,\alpha}(T^*M^{2n},g_J)$ such that
     $$
  \int_{M^{2n}}\omega^n=\int_{M^{2n}}A_se^{sf}\omega^n,\,\,\,(\omega+da_s)^n=A_se^{sf}\omega^n.e'g
  $$
    \end{defi}
    Now using Theorems \ref{main the1} and \ref{main the2}, similar to Yau's proof for  Calabi conjecture in \cite{Yau},
     we will show that this set $S$ is closed.
     And using Theorem \ref{main the3}, we will show that $S$ is open.

     \begin{theo}\label{closed theo}
     In Definition \ref{defi equ}, the set is a closed subset of $[0,1]$.
     \end{theo}
     \begin{proof}
     Set $(\omega,J_1,g_1)=(\omega,J,g_J)$ on $M^{2n}$.
     It must be shown that $S$ contains its limit points,
     and therefore is closed.
     Let $\{s_j\}^\infty_{j=0}$ be a sequence in $S$, which converges to some $s'\in[0,1]$.
     We will prove that $s'\in S$, by the definition there exists $a_j\in C^{4,\alpha}(T^*M^{2n},g_1)$, $da_j\in C^{3,\alpha}(\Lambda^2T^*M^{2n},g_1)$ and
     $A_j>0$ such that
     \begin{equation}\label{sequ equ}
       (\omega+da_j)^n=A_je^{s_jf}\omega^n,\,\,\,\int_{M^{2n}}\omega^n=\int_{M^{2n}}A_je^{s_jf}\omega^n.
     \end{equation}
     Define $Q_1$ by $Q_1=\|f\|_{C^{3,\alpha}}$.
     Let $Q_2,Q_3,Q_4$ be the constants given by Theorem \ref{main the1}, which depend on $Q_1$,
     and $Q_5,Q_6$ the constants given by Theorem \ref{main the2}, which depend on $Q_1,Q_2,Q_3,Q_4$.

     As $s_j\in[0,1]$, $\|s_jf\|_{C^3_1}\leq Q_1$.
     So applying Theorem \ref{main the1} with $\varphi(a_j)$ in place of $\varphi(a)$,
     $a_j$ in place of $a$ and $s_jf$ in place of $f$,
     we see that $$\|\varphi(a_j)\|_{C^0_1}\leq Q_2,\,\,\|da_j\|_{C^0_1}\leq Q_3\,\,\, {\rm and}  \,\,\,\|\nabla_1(da_j)\|_{C^0_1}\leq Q_4$$ for all $j$.
     Here $\nabla_1$ is the Levi-Civita connection with respect to the metric $g_1$.
     Thus, by Theorem \ref{main the2},  $$\varphi(a_j)\in C^{5,\alpha}(\mathring{M}^{2n}(a_j,P_j),g_1), \,\,
      \|\varphi(a_j)\|_{ C^{5,\alpha}_1}\leq Q_5, \,\,a_j\in C^{4,\alpha}(T^*M^{2n},g_1)$$
     and  $$\|a_j\|_{C^{4,\alpha}_1}=\|da_j\|_{C^{3,\alpha}_1}\leq Q_6.$$
      Now the Kondrakov Theorem (cf. \cite[Theprem 1.2.3]{Joc}) says that the inclusion
      $$C^{4,\alpha}(T^*M^{2n},g_1)\rightarrow C^4(T^*M^{2n},g_1)$$ is compact.
      It follows that as the sequence $\{a_j\}^\infty_{j=0}$ is bounded in $C^{4,\alpha}(T^*M^{2n},g_1)$, it lies in a compact subset of $C^4(T^*M^{2n},g_1)$.
      Let $a'\in C^4(T^*M^{2n},g_1)$ be the limit of the subsequence.

      Define $A_{s'}>0$ by $$A_{s'}\int_{M^{2n}}e^{s'f}\omega^n=\int_{M^{2n}}\omega^n.$$
      Then $A_{i_j}\rightarrow A_{s'}$ as $j\rightarrow\infty$, because $s_{i_j}\rightarrow s'$ as $j\rightarrow\infty$.
      Since $\{da_{i_j}\}^\infty_{j=0}$ converges in $C^3(\Lambda^2T^*M^{2n},g_1)$ we may take the limit in (\ref{sequ equ}) giving
      \begin{equation}\label{limit sequ equ}
       (\omega+da')^n=A_{s'}e^{s'f}\omega^n.
     \end{equation}
     It is easy to see that $(\omega+da')$ is a new symplectic form on $M^{2n}$.
     By Proposition \ref{CYequ Lipschitz}, there exists a finite Darboux's coordinate subatlas
     $\mathcal{V}(a')=\{V_k(a')\}_{1\leq k\leq \overline{N}(a')}$
     for symplectic forms $\omega$ and $\omega(a')=\omega+da'$,
     by $\mathcal{V}(a')$ define a new open submanifold $\mathring{M}^{2n}(a',P')$ of $M^{2n}$,
     $M^{2n}\setminus\mathring{M}^{2n}(a',P')$ has Lebesgue measure zero,
     and a corresponding function $\varphi(a')$ on $ C^{5,\alpha}(\mathring{M}^{2n}(a',P'))$
     such that
     \begin{equation}\label{}
       \int_{M^{2n}}\varphi(a')\omega^n=0,\,\,\,(\omega+da')^n=A'e^{s'f}\omega^n,\,\,\,-\sqrt{-1}da'|_{ \mathring{M}^{2n}(a',P')}=\partial\bar{\partial}\varphi(a')
     \end{equation}
      in terms of holomorphic coordinates $\{z_1,\cdot\cdot\cdot,z_n\}$.
      Theorem \ref{main the1} and \ref{main the2} then show that $a'\in C^{4,\alpha}(T^*M^{2n},g_1)$, $\varphi(a')\in C^{5,\alpha}(\mathring{M}^{2n}(a',P'),g_1)$.
      Therefore $s'\in S$.
      So $S$ contains its limit points, and is closed.
       \end{proof}

       \begin{theo}\label{open theo}
        In Definition \ref{defi equ}, the set $S$ is an open subset of $[0,1]$.
       \end{theo}
       \begin{proof}
         Suppose $s'\in S$.
         Then by the definition there exist $$A_{s'}>0,\,\,a_{s'}\in C^{4,\alpha}(T^*M^{2n},g_1),\,\,\omega(a_{s'})=\omega+da_{s'}$$
         and there exists an open dense submanifold $\mathring{M}^{2n}(a',P')\subset M^{2n}$, where $M^{2n}\setminus \mathring{M}^{2n}(a',P')$ has Lebesgue measure zero.
         Also there exists
         a of function $\varphi(a')\in C^{5,\alpha}(\mathring{M}^{2n}(a',P'))$
          such that
           $$-\sqrt{-1}da_{s'}|_{ \mathring{M}^{2n}(a_{s'},P_{s'})}=\partial\bar{\partial}\varphi(a')$$
            in terms of holomorphic coordinates $\{z_1,\cdot\cdot\cdot,z_n\}$,
          $$\int_{M^{2n}}\varphi(a_{s'})\omega^n=0\,\,\,
          {\rm and} \,\,\,(\omega+da_{s'})^n=A_{s'}e^{s'f}\omega^n\,\,\, {\rm on} \,\,\,M^{2n}.$$
          Apply Theorem \ref{main the3}, with $s'f$ in place of $f'$, and $sf$ in place of $f$ for $s\in[0,1]$.
          The theorem shows that  whenever $|s-s'|\|f\|_{C^{3,\alpha}_1}$ in sufficiently small,
           there exist $$a_s\in C^{4,\alpha}(T^*M^{2n},g_1),\,\,\,
          \varphi(a_s)\in C^{5,\alpha}(\mathring{M}^{2n}(a_s,P_s),g_1),$$
           an open dense submanifold $\mathring{M}^{2n}(a_s,P_s)\subset M^{2n}$ where $M^{2n}\setminus \mathring{M}^{2n}(a_s,P_s)$ has Lebesgue measure zero,
           and $A_s>0$ such that
          $$-\sqrt{-1}da_s|_{ \mathring{M}^{2n}(a_s,P_s)}=\partial\bar{\partial}\varphi(a_s),\,\,\int_{M^{2n}}\varphi(a_s)\omega^n=0$$
           and $$ (\omega+da_s)^n=A_se^{sf}\omega^n\,\,\,{\rm on}\,\,\,M^{2n}.$$
          Then $s\in S$. Thus, if $s\in [0,1]$ is sufficiently close to $s'$ then $s\in S$, and $S$ contains an open
          neighbourhood in $[0,1]$ of each $s'$ in $S$. So $S$ is open .
          \end{proof}
          Using Theorem \ref{closed theo} and  \ref{open theo}, we get an existence result for the Calabi-Yau equation on closed symplectic manifolds.

           \begin{theo}\label{main the4}
          Let $(M^{2n},\omega)$ be a closed symplectic manifold of dimension $2n$, $(\omega,J,g_J)$
        an almost K\"{a}hler structure on  $(M^{2n},\omega)$.
        Choose $\alpha\in(0,1)$, and let $f\in C^{3,\alpha}(M^{2n},g_J)$.
        Then exists a $1$-form $a\in C^{4,\alpha}(T^*M^{2n},g_J)$ such that $\omega+da$ is a new symplectic form on $M^{2n}$.
        There exists an open dense submanifold $\mathring{M}^{2n}(a,P)$ of $M^{2n}$ where $M^{2n}\setminus \mathring{M}^{2n}(a,P)$ has Lebesgue measure zero,
        a function $\varphi(a)\in  C^{5,\alpha}(\mathring{M}^{2n}(a,P),g_J)$ and $A>0$ such that

          (1) $\int_{M^{2n}}\varphi(a)\omega^n=0$,

          (2) $-\sqrt{-1}da|_{\mathring{M}^{2n}(a,P)}=\partial\bar{\partial}\varphi(a)$ on $\mathring{M}^{2n}(a,P)$,

          (3)$(\omega+da)^n=Ae^f\omega^n$ on $M^{2n}$.
          \end{theo}

          Finally using Theorem \ref{main the2} and \ref{main the4}, we can prove Theorem \ref{main result}.

            \vskip 6pt

         {\bf Proof of Theorem \ref{main result}} Suppose that $(M^{2n},\omega)$ is a closed symplectic manifold of dimension $2n$.
         Let $F\in C^\infty(M^{2n},\mathbb{R})$ satisfying
         $$
         \int_{M^{2n}}e^F\omega^n=\int_{M^{2n}}\omega^n.
         $$
             Then by Theorem \ref{main the4} there exists a $1$-form $a\in C^{4,\alpha}(T^*M^{2n},g_J)$,
             and Theorem \ref{main the2} shows that $a$ is smooth.
             This proves Theorem \ref{main result}.\qed

         \vskip 6pt

             In Section \ref{priori estimates}, we will prove Theorem \ref{main the1} and Theorem \ref{main the2}, in Section \ref{existence},
             we will prove  Theorem \ref{main the3}.

\section{A priori estimates for the derivatives of the K\"{a}hler potential }\label{priori estimates}
 \setcounter{equation}{0}
   This section is devoted to giving a priori estimates for the derivatives of the
    K\"{a}hler potential by the method of Moser iteration (cf. Yau \cite{Yau2} or Joyce \cite[Chapter 6]{Joc}).
    Suppose that $(M^{2n},\omega,J,g_J)$ be a closed almost K\"{a}hler manifold of dimension $2n$.
    When $J$ is an $\omega$-compatible almost complex structure on $M^{2n}$, $g_J(\cdot,\cdot):=\omega(\cdot,J\cdot)$,
    $(\omega,J,g_J)$ is an almost K\"{a}hler strucutre on $M^{2n}$.
    Suppose that $a\in\Omega^1(M^{2n})$ is a $1$-form on $M^{2n}$ which is a solution of the following Calabi-Yau equation
    \begin{eqnarray}\label{4CYequ}
      (\omega+da)^n &=& Ae^f\omega^n, \,\,A>0, \,\,f\in C^\infty(M^{2n},\mathbb{R}), \,\,\int_{M^{2n}}f\omega^n=0,  \nonumber \\
     && \int_{M^{2n}}Ae^f\omega^n=\int_{M^{2n}}\omega^n.
    \end{eqnarray}
    Thus, $\omega(a):=\omega+da$ is a new symplectic form on $M^{2n}$ by (\ref{4CYequ}).
    Similarly, we can find a new almost  K\"{a}hler structure $(\omega(a),J(a),g_{J(a)})$ on $M^{2n}$.
    Then, by using Darboux's coordinate charts for symplectic forms $\omega$ and  $\omega(a)=\omega+da$ as follows:
    We can find a finite Darboux's coordinate subatlas
     $\mathcal{V}(a)=\{V_k(a)\}_{1\leq k\leq N(a)}$.
     On each $V_k(a)$, $1\leq k\leq N(a)$, find two families almost  K\"{a}hler structures:

   (1) $(\omega|_{V_k(a)},J_k(t),g_k(t))$, $t\in [0,1]$;

   (2) $(\omega(a)|_{V_k(a)},J_k(a)(t),g_k(a)(t))$, $t\in [0,1]$.\\
   When $t=1$, on $V_{k_1}(a)\cap V_{k_2}(a)\neq \emptyset$,
   $$(\omega|_{V_{k_1}(a)},J_{k_1}(1),g_{k_1}(1))=(\omega|_{V_{k_2}(a)},J_{k_2}(1),g_{k_2}(1))=(\omega,J,g_J).$$
   When $t=0$,  $(\omega|_{V_{k}(a)},J_{k}(0),g_{k}(0))$ is a K\"{a}hler flat structure on $V_k(a)$.
   We have the same conclusion for $\omega(a)$.
   More details, see Appendix \ref{App A}.
   On $V_k(a)$, $1\leq k\leq N(a)$, for $(\omega|_{V_k(a)},J_k(0),g_k(0))$ and $(\omega(a)|_{V_k(a)},J_k(a)(0),g_k(a)(0))$,
   choose a holomorphic coordinates $\{z_1,\cdot\cdot\cdot,z_n\}$ such that the Calabi-Yau equation
   $$
   \omega(a)^n=Ae^f\omega^n
   $$
  is reduced to the Monge-Amp\`{e}re equations on $V_k(a)$ (see Equation  (\ref{MA equ2}) and  Proposition \ref{CYequ Lipschitz} )
  \begin{equation}\label{4MAequ}
    \det(h_{i\bar{j},k}+\frac{\partial^2\varphi_k(a)}{\partial z_i\partial \bar{z}_{j}})=Ae^f\det(\frac{\partial^2\varphi_k(a)}{\partial z_i\partial \bar{z}_{j}}),
  \end{equation}
  where $\varphi_k(a)$ is the K\"{a}hler potential on $V_k(a)$,
  \begin{equation}
  h_{i\bar{j},k}=-\sqrt{-1}\omega|_{V_k(a)},\,\,\, -\sqrt{-1}da|_{V_k(a)}=\frac{\partial^2\varphi_k(a)}{\partial z_i\partial \bar{z}_{j}},
   \end{equation}
  \begin{equation}
h(a)_{i\bar{j},k}=-\sqrt{-1}\omega(a)|_{V_k(a)}=h_{i\bar{j},k}+\frac{\partial^2\varphi_k(a)}{\partial z_i\partial \bar{z}_{j}},
 \end{equation}
$(h_{i\bar{j},k})$, $(h(a)_{i\bar{j},k})$, $1\leq k\leq N(a)$, are Hermitian matrices associated with $h(0)$ and $h(a)(0)$, respectively.(see Lemma \ref{Joc lemma 1'}).

We have a K\"{a}hler potential $\varphi(a)$ on
 $\mathcal{V}(a)=\{V_k(a)\}_{1\leq k\leq N(a)}$.
 Define $W_1(a)=V_1(a)$, $T_1(a)=V_1(a)$.
 For $2\leq k\leq  N(a)$,
 $$W_k(a)=V_k(a)\backslash \bigcup_{j=1}^{k-1}\overline{V_j}(a),\,\,\,
 T_k(a)=V_k(a) \backslash\bigcup_{j=1}^{k-1}V_j(a).$$
 Then $$\mathring{M}^{2n}(a,P)=\bigcup_{k=1}^{N(a)}W_k(a),\,\,\,M^{2n}=\bigcup_{k=1}^{N(a)}T_k(a).$$
 $\mathring{M}^{2n}(a,P)$ is an open dense submanifold and $M^{2n}\backslash\mathring{M}^{2n}(a,P)$ has Hausdorff dimension $\leq 2n-1$
 with Lebesgue measure zero.
 $\{T_k(a)\}_{1\leq k\leq N(a)}$ is a disjoint partition of $M^{2n}$.
 Hence, on  $\mathring{M}^{2n}(a,P)$ we have  K\"{a}hler flat structure associated to $\omega$:
 \begin{equation}\label{}
   (\omega|_{V_k(a)},J_k(0),g_k(0)), 1\leq k\leq  N(a).
 \end{equation}
 Thus, we define two (measurable) Lipschitz Riemannian metrics, $g_1$, $g_0$ on $(M^{2n},\omega)$ which are equivalent norms (that is, Lipschitz condition (\ref{metric equ1})-(\ref{metric equ0})), where $g_1=g_J$\ is an almost
 K\"{a}hler metric, $g_0$ is a (measurable) Lipschitz K\"{a}hler metric.
 Restricted to the open dense submanifold $\mathring{M}^{2n}(a,P)$ of $M^{2n}$, $g_1$ and $g_0$ are quasi isometric (see (\ref{metric equ}) ).
 We also define spaces $C^{k,\alpha}_i$, $i=0,1$, for $\alpha\in [0,1)$ and $k$ is a nonnegative integer which are with equivalent norms (see (\ref{neq estimate})).
 In particular, $C^{k,\alpha}_0$ can be extended to $M^{2n}$.

 In the remainder of this section, we give proofs of Theorem \ref{main the1} and Theorem \ref{main the2}.
 Here constants $Q_1,\cdot\cdot\cdot,Q_4$ are defined in Theorem \ref{main the1} and \ref{main the2}.
 Let $\nabla_0$ and $\nabla_1$ be Levi-Civita connections with respect to the metrics $g_0$ and $g_1$, respectively.

 \vskip 6pt

 {\bf Estimates of order zero.}

 By Equation (\ref{4MAequ}), on $\mathring{M}^{2n}(a,P)$, since
 $$\omega-\omega(a)=-dJ_0d\varphi(a)=2\sqrt{-1}\partial\bar{\partial}\varphi(a),$$
 $$dvol_{g_0} =dvol_{g_1}=dvol_{g_J}=\frac{\omega^n}{n!},\,\,\,\omega(a)^n =Ae^f\omega^n,$$
              and $M^{2n}\backslash\mathring{M}^{2n}(a,P)$ has Lebesgue measure zero,
             it follows that
             \begin{equation}\label{vol equ}
                 \int_{\mathring{M}^{2n}(a,P)}  \varphi(a)dvol_{g_0}= \int_{\mathring{M}^{2n}(a,P)}  \varphi(a)dvol_{g_1}= \int_{M^{2n}}  \varphi(a)dvol_{g_1} =0.
             \end{equation}
 We see that
   \begin{eqnarray}
      (1-Ae^f)\omega^n &=&\omega^n- \omega(a)^n \nonumber \\
      &=&   -dJ_0d\varphi(a)\wedge(\omega^{n-1}+\omega^{n-2}\wedge\omega(a)+\cdot\cdot\cdot+\omega(a)^{n-1}).
   \end{eqnarray}
   By the construction of $\varphi(a)$, it is easy to see that $\varphi(a)\in L^\infty_k(M^{2n})\cap C^\infty(\mathring{M}^{2n}(a,P))$
   with respect to the metric $g_1$ (cf. Appendix \ref{App A}).
   Hence, there exists a sequence of smooth functions $\{\varphi(a)^{(j)}\}_{j\geq 1}$ on $M^{2n}$ such that $\{\varphi(a)^{(j)}\}_{j\geq 1}$
   converges to $\varphi(a)$ in $L^\infty_k(M^{2n})$, for $k\geq 0$,  with respect to $g_1$.
 Now, since $g_1$ and $g_0$ are quasi isometry on $\mathring{M}^{2n}(a,P)$ and $g_1$ and $g_0$ are equivalent (measurable) Lipschitz Riemannian metrics on $M^{2n}$
 due to Lipschitz condition (\ref{metric equ}),
 by Stokes Theorem for (measurable) Lipschitz K\"{a}hler metric on $M^{2n}$ ( see Lemma \ref{adjoint lemma}),
 \begin{eqnarray*}
    &&\int_{M^{2n}}d[\varphi(a)|\varphi(a)|^{p-2}J_0d\varphi(a)\wedge(\omega^{n-1}+\omega^{n-2}\wedge\omega(a)+\cdot\cdot\cdot+\omega(a)^{n-1})]  \\
    && =\lim_{j\rightarrow \infty}\int_{M^{2n}}d[\varphi(a)^{(j)}|\varphi(a)^{(j)}|^{p-2}J_0d\varphi(a)^{(j)}\wedge(\omega^{n-1}+\omega^{n-2}\wedge\omega(a)+\cdot\cdot\cdot+\omega(a)^{n-1})]   \\
    && =0.
 \end{eqnarray*}
 Therefore,
 we find
 \begin{eqnarray*}
    &&\int_{M^{2n}}\varphi(a)|\varphi(a)|^{p-2}(1-Ae^f)\omega^n  \\
   &&=(p-1) \int_{M^{2n}}|\varphi(a)|^{p-2}d\varphi(a)\wedge J_0d\varphi(a)\wedge(\omega^{n-1}+\omega^{n-2}\wedge\omega(a)+\cdot\cdot\cdot+\omega(a)^{n-1}).
 \end{eqnarray*}
Lemma \ref{wedge equ} gives expressions for $$d\varphi(a)\wedge J_0d\varphi(a)\wedge\omega^{n-1}\,\,\,
  {\rm and} \,\,\,d\varphi(a)\wedge J_0d\varphi(a)\wedge\omega^{n-j-1}\wedge\omega(a)^j,$$
then
$$
 \int_{M^{2n}}(1-Ae^f)\varphi(a)|\varphi(a)|^{p-2}\omega^n=\frac{p-1}{n} \int_{M^{2n}}|\varphi(a)|^{p-2}(|\nabla_0\varphi(a)|^2_{g_0}+G_1+\cdot\cdot\cdot+G_{n-1})\omega^n,
$$
where $\nabla_0$ is Levi-Civita connection with respect to the metric $g_0$, $G_1,\cdot\cdot\cdot,G_{n-1}$ are nonnegative real functions on $\mathring{M}^{2n}(a,P)$.
Thus (\ref{vol equ}) yields
$$
\int_{M^{2n}}|\varphi(a)|^{p-2}(|\nabla_0\varphi(a)|^2_{g_0}+G_1+\cdot\cdot\cdot+G_{n-1})dvol_{g_0}=
\frac{n}{p-1}\int_{M^{2n}}(1-Ae^f)\varphi(a)|\varphi(a)|^{p-2}dvol_{g_0}.
$$
Combining this with the equation
$$
\frac{1}{4}p^2|\varphi(a)|^{p-2}|\nabla_0\varphi(a)|^2_{g_0}=|\nabla_0|\varphi(a)|^{\frac{p}{2}}|^2_{g_0}
$$
and the inequality $0\leq|\varphi(a)|^{p-2}G_j$ gives
$$
\int_{M^{2n}}|\nabla_0|\varphi(a)|^\frac{p}{2}|^2_{g_0}dvol_{g_0}\leq\frac{np^2}{4(p-1)}\cdot\int_{M^{2n}}(1-Ae^f)\varphi(a)|\varphi(a)|^{p-2}dvol_{g_0}.
$$
Notice that restricted to the open dense submanifold, $\mathring{M}^{2n}(a,P)$, of $M^{2n}$, $g_0$ and $g_1$ are quasi isometric,
$g_0$ and $g_1$ are equivalent (measurable) Lipschitz Riemannian metrics on $M^{2n}$ due to Lipschitz condition (\ref{metric equ}).
By Lemma \ref{quasi isometry} (that is, $L^2_k\Omega^q(t)$ and $L^2_k\Omega^q(1)$ are quasi isometry due to (\ref{quasi isometry2}),
$t\in [0,1)$, $0\leq q\leq 2n$), we have
\begin{equation}\label{int leq}
  \int_{M^{2n}}|\nabla_1|\varphi(a)|^\frac{p}{2}|^2_{g_1}dvol_{g_1}\leq C\frac{np^2}{4(p-1)}\cdot\int_{M^{2n}}(1-Ae^f)\varphi(a)|\varphi(a)|^{p-2}dvol_{g_1},
\end{equation}
where $\nabla_1$ is Levi-Civita connection with respect to the metric $g_1$, $C$ is a positive constant depending on $M^{2n}$, $\omega$ and $J$.

By the Sobolev Embedding Theorem (cf. T. Aubin \cite{Au2}), we have the following lemma:
\begin{lem}\label{Sob leq}
There are constant $C_1$, $C_2$ depending on $M^{2n}$,  $\omega$ and $J$ such that if $\psi\in L^2_1(M^{2n})$ with respect to the metric $g_1$,
then $$\|\psi\|^{2\varepsilon}_{L^2(M^{2n},g_1)}\leq C_1(\|\nabla_1\psi\|^2_{L^2(M^{2n},g_1)}+\|\psi\|^2_{L^2(M^{2n},g_1)}),$$
and if $$\int_{M^{2n}}\psi dvol_{g_1}=0,$$ then $$\|\psi\|_{L^2(M^{2n},g_1)}\leq C_2\|\nabla_1\psi\|_{L^2(M^{2n},g_1)},$$ where $\varepsilon=\frac{n}{n-1}$.
\end{lem}

Putting $p=2$ in (\ref{int leq}) and using Lemma \ref{Sob leq}, we have the following lemma which is the same to Lemma $6.4.3$ in \cite{Joc}.
\begin{lem}\label{}
There is a constant $C_3$ depending on $M^{2n}$,  $\omega$, $J$  and  $Q_1$ (defined in Theorem \ref{main the1} and \ref{main the2} ) such that
if $p\in[2,2\varepsilon]$, then $\|\psi\|_{L^2(M^{2n},g_1)}\leq C_3$.
\end{lem}

 As done in Proposition $6.4.4$ in  \cite{Joc}, we can give a proposition by a form of induction on $p$.
\begin{prop}\label{Q_2C_4leq}
There are constants $Q_2$, $C_4$ depending  on $M^{2n}$,  $\omega$, $J$  and  $Q_1$  such that for each $p\geq2$, we have
$$
\|\varphi(a)\|_{L^p(M^{2n},g_1)}\leq Q_2(C_4p)^{-\frac{n}{p}}.
$$
\end{prop}
Now we can get the first part of Theorem \ref{main the1}.
\begin{prop}\label{Q_2 leq}
The function $\varphi(a)$ satisfies
$$
\|\varphi(a)\|_{L^\infty(M^{2n},g_1)}\leq Q_2,\,\,\,\|\varphi(a)\|_{C^0(\mathring{M}^{2n}(a),g_1)}\leq Q_2.
$$
\end{prop}
\begin{proof}
As $\varphi(a)$ is continuous on $\mathring{M}^{2n}(a,P)$, $M^{2n}\backslash\mathring{M}^{2n}(a,P)$ has Lebesgue measure zero,
and $M^{2n}$ is a compact manifold,
$$
\|\varphi(a)\|_{L^\infty(M^{2n},g_1)}=\lim_{p\rightarrow\infty}\|\varphi(a)\|_{L^p(M^{2n},g_1)},\,\,\,
\|\varphi(a)\|_{C^0(\mathring{M}^{2n}(a,P),g_1)}=\lim_{p\rightarrow\infty}\|\varphi(a)\|_{L^p(M^{2n},g_1)}.
$$
But, by Proposition \ref{Q_2C_4leq},
$$
\|\varphi(a)\|_{L^p(M^{2n},g_1)}\leq Q_2(C_4p)^{-\frac{n}{p}},
$$
and $\lim_{p\rightarrow\infty}(C_4p)^{-\frac{n}{p}}=1$, so the result follows.
\end{proof}

\begin{rem}
The Moser iteration method was also used by Delano\"{e} \cite{Dela2} for the Gromov type Calabi-Yau equation (\ref{Gromov type CY equ}) on closed symplectic manifolds to prove the zero order estimate for the symplectic potential.
\end{rem}

 {\bf  Second-order estimates.}

 Here is some notation that will be used for the next calculations. We have two almost K\"{a}hler metrics $g_J$ and $g_{J(a)}$ on $M^{2n}$.
 Also, we have measurable K\"{a}hler metrics $h(0)$ and $h(a)(0)$ on  $\mathring{M}^{2n}(a,P)$ (cf. Proposition \ref{CYequ Lipschitz}).
 Let $\nabla$ be the Levi-Civita connection of $h(0)$, that is $\nabla=\nabla_0$ defined in before. If $T$ is a tensor on $\mathring{M}^{2n}(a,P)$, we will write $\nabla_{a_1\cdot\cdot\cdot a_k}T$
 in place of $\nabla_{a_1}\cdot\cdot\cdot\nabla_{a_k}T$, the $k^{th}$ derivative of $T$ using $\nabla$.
 For $\psi\in C^2(\mathring{M}^{2n}(a,P))$, let $\Delta^L\psi$ be the Laplacian of $\psi$ with respect to $h(0)$,
  and let $\Delta'^{L}\psi$ be the Laplacian  of $\psi$ with respect to $h(a)(0)$.
  Then in this notation,
  $$\Delta^L\psi=-h(0)^{c\bar{d}}\nabla_{c\bar{d}}\psi$$
  and
  $$
  \Delta'^{L}\psi=-h(a)(0)^{c\bar{d}}\nabla_{c\bar{d}}\psi.
  $$
  The second formula defines the Laplacian with respect to $h(a)(0)$ using the Levi-Civita connection of $h(0)$ on  $\mathring{M}^{2n}(a,P)$.

 Using this notation we shall prove:
 \begin{prop}\label{Delta Delta}
 In the situation above, we have
 \begin{eqnarray}\label{Delta equ0}
   \Delta'^{L}(\Delta^L\varphi(a)) &=& -\Delta^L f+h(0)^{\alpha\bar{\lambda}}h(a)(0)^{\mu\bar{\beta}}h(a)(0)^{\gamma\bar{\nu}}\nabla_{\alpha\bar{\beta}\gamma}\varphi(a)\nabla_{\bar{\lambda}\mu\bar{\nu}}\varphi(a).
 \end{eqnarray}
 \end{prop}
 \begin{proof}
 Taking the log of equation (\ref{4MAequ}) and applying $\nabla$ gives $$\nabla_{\bar{\lambda}}f=h(a)(0)^{\mu\bar{\nu}}\nabla_{\bar{\lambda}\mu\bar{\nu}}\varphi(a).$$
 Therefore, as $\Delta^Lf=-h(0)^{\alpha\bar{\lambda}}\nabla_{\alpha\bar{\lambda}}f$, we have
 \begin{equation}\label{Delta equ}
   \Delta^Lf=-h(0)^{\alpha\bar{\lambda}}(\nabla_{\alpha}h(a)(0)^{\mu\bar{\nu}})\nabla_{\bar{\lambda}\mu\bar{\nu}}\varphi(a)
   -h(0)^{\alpha\bar{\lambda}}h(a)(0)^{\mu\bar{\nu}}\nabla_{\alpha\bar{\lambda}\mu\bar{\nu}}\varphi(a).
 \end{equation}
 But $$h(a)(0)^{\mu\bar{\beta}}h(a)(0)_{\bar{\beta}\gamma}=\delta^{\mu}_{\gamma},\,\,\, {\rm and} \,\,\,\nabla_{\alpha}h(a)(0)_{\bar{\beta}\gamma}=\nabla_{\alpha\bar{\beta}\gamma}\varphi(a)$$
 as $$h(a)(0)_{\bar{\beta}\gamma}=h(0)_{\bar{\beta}\gamma}+\nabla_{\bar{\beta}\gamma}\varphi(a),$$
 so that
 $$
 0=\nabla_{\alpha}\delta^{\mu}_{\gamma}=h(a)(0)_{\bar{\beta}\gamma}\nabla_{\alpha}h(a)(0)^{\mu\bar{\beta}}+h(a)(0)^{\mu\bar{\beta}}\nabla_{\alpha\bar{\beta}\gamma}\varphi(a).
 $$
 Contracting with $h(a)(0)^{\gamma\bar{\nu}}$ shows that $$\nabla_{\alpha}h(a)(0)^{\mu\bar{\nu}}=-h(a)(0)^{\mu\bar{\beta}}h(a)(0)^{\gamma\bar{\nu}}\nabla_{\alpha\bar{\beta}\gamma}\varphi(a).$$
 Substituting this into (\ref{Delta equ}) yieds
 $$
 \Delta^Lf=h(0)^{\alpha\bar{\lambda}}h(a)(0)^{\mu\bar{\beta}}h(a)(0)^{\gamma\bar{\nu}}\nabla_{\alpha\bar{\beta}\gamma}\varphi(a)\nabla_{\bar{\lambda}\mu\bar{\nu}}\varphi(a)
   -h(0)^{\alpha\bar{\lambda}}h(a)(0)^{\mu\bar{\nu}}\nabla_{\alpha\bar{\lambda}\mu\bar{\nu}}\varphi(a),
 $$
 and rearranging and changing some indices gives
 \begin{equation}\label{Delta equ1}
h(a)(0)^{\alpha\bar{\beta}}h(0)^{\gamma\bar{\delta}}\nabla_{\gamma\bar{\delta}\alpha\bar{\beta}}\varphi(a)=   -\Delta^Lf+h(0)^{\alpha\bar{\lambda}}h(a)(0)^{\mu\bar{\beta}}h(a)(0)^{\gamma\bar{\nu}}\nabla_{\alpha\bar{\beta}\gamma}\varphi(a)\nabla_{\bar{\lambda}\mu\bar{\nu}}\varphi(a).
 \end{equation}
 However, it can be shown that
 $$
 h(a)(0)^{\alpha\bar{\beta}}h(0)^{\gamma\bar{\delta}}\nabla_{\alpha\bar{\beta}\gamma\bar{\delta}}\varphi(a)-h(a)(0)^{\alpha\bar{\beta}}h(0)^{\gamma\bar{\delta}}\nabla_{\gamma\bar{\delta}\alpha\bar{\beta}}\varphi(a)
 \,\,\,\,\,\,\,\,\,\,\,\,\,\,\,\,\,\,\,\,\,\,\,\,\,\,\,\,\,\,\,\,\,\,\,\,\,\,\,\,\,\,
 $$
 $$
 \,\,\,\,\,\,\,\,\,\,\,\,\,\,\,\,\,\,\,\,\,=h(a)(0)^{\alpha\bar{\beta}}h(0)^{\gamma\bar{\delta}}
 (R^{\bar{\epsilon}}_{\bar{\delta}\gamma\bar{\beta}}\nabla_{\alpha\bar{\epsilon}}\varphi(a)
 -R^{\bar{\epsilon}}_{\bar{\beta}\alpha\bar{\delta}}\nabla_{\gamma\bar{\epsilon}}\varphi(a)),
 $$
 where $R^a_{bcd}$ is the Riemann curvature of $h(0)$.
 Notice that $h(0)$ is K\"{a}hler flat on $\mathring{M}^{2n}(a,P)$. So $R^a_{bcd}=0$ and
 $$
 h(a)(0)^{\alpha\bar{\beta}}h(0)^{\gamma\bar{\delta}}\nabla_{\alpha\bar{\beta}\gamma\bar{\delta}}\varphi(a)-h(a)(0)^{\alpha\bar{\beta}}h(0)^{\gamma\bar{\delta}}\nabla_{\gamma\bar{\delta}\alpha\bar{\beta}}\varphi(a)
 =0,
 $$
    and this combines with (\ref{Delta equ1}) and
   $$\Delta'^{L} \Delta^{L}\varphi(a)=h(a)(0)^{\alpha\bar{\beta}}h(0)^{\gamma\bar{\delta}}\nabla_{\alpha\bar{\beta}\gamma\bar{\delta}}\varphi(a)$$
   to give (\ref{Delta equ0}), as we want.
 \end{proof}
The next result gives the second part of Theorem \ref{main the1}.

\begin{prop}\label{sec order}
There are constants $c_1,c_2$ and $Q_3$ depending on $M^{2n},\omega,J_0$ and $Q_1$, such that

\begin{eqnarray}\label{h estimate}
 &&\|h(a)(0)_{ab}\|_{C^0_1( \mathring{M}^{2n}(a,P))} \leq c_1, \,\,\,\|h(a)(0)^{ab}\|_{C^0_1( \mathring{M}^{2n}(a,P))}\leq c_2,  \nonumber\\
  &&\,\,\,\,\,\,\,\,\,\,\,\,\,\, and \,\,\,\,\,\,\,\, \|da\|_{C^0_1(M^{2n})}=\|dJ_0d\varphi(a)\|_{C^0_1( M^{2n})} \leq Q_3.
\end{eqnarray}
\end{prop}
\begin{proof}
Define a function $F$ on $\mathring{M}^{2n}(a,P)$ by
$$F=\log(n-\Delta^{L}\varphi(a))-k\varphi(a),$$
where $k$ is a constant to be chosen later.
Note that  $n-\Delta^{L}\varphi(a)>0$ by Proposition \ref{Joc prop 1}, so $F$ is well-defined.
We shall find an expression for $\Delta'^{L}F$. It is easy to show that
\begin{eqnarray*}
  \Delta'^{L}F &=&  -(n-\Delta^{L}\varphi(a))^{-1}\Delta'^{L}\Delta^{L}\varphi(a)\\
   && +
(n-\Delta^{L}\varphi(a))^{-2}h(a)(0)^{\alpha\bar{\lambda}}h(0)^{\mu\bar{\beta}}
h(0)^{\gamma\bar{\nu}}\nabla_{\alpha\bar{\beta}\gamma}\varphi(a)\nabla_{\bar{\lambda}\mu\bar{\nu}}\varphi(a)-k\Delta'^{L}\varphi(a).
\end{eqnarray*}
Now
\begin{eqnarray*}
   \Delta'^{L}\varphi(a)&=& -h(a)(0)^{\alpha\bar{\beta}}\nabla_{\alpha\bar{\beta}}\varphi(a) \\
   &=& h(a)(0)^{\alpha\bar{\beta}}(h(0)_{\alpha\bar{\beta}}-h(a)(0)_{\alpha\bar{\beta}})\\
   &=& h(a)(0)^{\alpha\bar{\beta}}h(0)_{\alpha\bar{\beta}}-n,
\end{eqnarray*}
so (\ref{Delta equ0}) and the equation above give
\begin{equation}\label{Delta' equ}
   \Delta'^{L}F =  (n-\Delta^{L}\varphi(a))^{-1}\Delta^{L}f+k(n-h(a)(0)^{\alpha\bar{\beta}}h(0)_{\alpha\bar{\beta}})-(n-\Delta^{L}\varphi(a))^{-1}G,
\end{equation}
where $G$ is defined by
\begin{eqnarray}
   G&=&h(0)^{\alpha\bar{\lambda}}h(a)(0)^{\mu\bar{\beta}}h(a)(0)^{\gamma\bar{\nu}}\nabla_{\alpha\bar{\beta}\gamma}\varphi(a)\nabla_{\bar{\lambda}\mu\bar{\nu}}\varphi(a) \nonumber\\
   && -(n-\Delta^{L}\varphi(a))^{-1}h(a)(0)^{\alpha\bar{\lambda}}h(0)^{\mu\bar{\beta}}h(0)^{\gamma\bar{\nu}}\nabla_{\alpha\bar{\beta}\gamma}\varphi(a)\nabla_{\bar{\lambda}\mu\bar{\nu}}\varphi(a).
\end{eqnarray}
Now expanding the inequality
$$
h(0)^{\alpha\bar{\lambda}}h(a)(0)^{\mu\bar{\beta}}h(a)(0)^{\gamma\bar{\nu}}
[(n-\Delta^{L}\varphi(a))\nabla_{\alpha\bar{\beta}\gamma}\varphi(a)-h(a)(0)_{\alpha\bar{\beta}}\nabla_{\gamma}\Delta^{L}\varphi(a))]\cdot\,\,\,\,\,\,\,\,\,\,\,\,\,\,
\,\,\,\,\,\,\,\,\,\,\,\,\,\,\,\,\,\,\,\,\,\,\,\,\,\,\,\,$$
$$
\,\,\,\,\,\,\,\,\,\,\,\,\,\,\,\,\,\,\,\,\,
[(n-\Delta^{L}\varphi(a))\nabla_{\bar{\lambda}\mu\bar{\nu}}\varphi(a)-h(a)(0)_{\bar{\lambda}\mu}\nabla_{\bar{\nu}}\Delta^{L}\varphi(a))]\geq 0,
$$
and dividing by $(n-\Delta^{L}\varphi(a))^2$ shows that $G\geq 0$.
Also, using the inequalities $$|\nabla_{\alpha\bar{\beta}}\varphi(a)|_{g_0}\leq (n-\Delta^{L}\varphi(a)))$$ and
$$|h(a)(0)^{\alpha\bar{\beta}}|_{g_0}\leq h(a)(0)^{\alpha\bar{\beta}}h(0)_{\alpha\bar{\beta}},$$
one can see that there is a constant $C_5\geq 0$ depending on $n$. Substituting these inequalities and $\Delta^{L}f\leq Q_1$ into (\ref{Delta' equ})
gives
\begin{equation}\label{Delta' equ1}
  \Delta'^{L}F \leq (n-\Delta^{L}\varphi(a))^{-1}Q_1+k(n-h(a)(0)^{\alpha\bar{\beta}}h(0)_{\alpha\bar{\beta}})+C_5h(a)(0)^{\alpha\bar{\beta}}h(0)_{\alpha\bar{\beta}}.
\end{equation}
At a point $p$ where $F$ is maximum, $ \Delta'^{L}F\geq 0$, and so by (\ref{Delta' equ1}) at $p$ we get
\begin{equation}\label{k inequ}
  (k-C_5)h(a)(0)^{\alpha\bar{\beta}}h(0)_{\alpha\bar{\beta}}\leq nk+Q_1(n-\Delta^{L}\varphi(a))^{-1}.
\end{equation}
Rearranging the inequality of Proposition \ref{Joc prop 1} gives $n-\Delta^{L}\varphi(a)\geq ne^{\frac{f}{n}}$, and as
$\|f\|_{C^0_1}\leq Q_1$ this shows that $(n-\Delta^{L}\varphi(a))^{-1}\geq \frac{1}{n}e^{\frac{Q_1}{n}}$.
Now choose $k=C_5+1$. Then (\ref{k inequ}) implies that $h(a)(0)^{\alpha\bar{\beta}}h(0)_{\alpha\bar{\beta}}\leq C_6$ at $p$,
where $C_6=nk+\frac{1}{n}Q_1e^{\frac{Q_1}{n}}$.
   Let us apply Lemma \ref{Joc lemma 1} to find expressions for $h(0)$ and $h(a)(0)$ at $p$ in terms of positive constants $a_1,\cdot\cdot\cdot,a_n$.
   In this notation, using Lemma \ref{Joc lemma 2}  we have
   $$
   n-\Delta^{L}\varphi(a)=\sum^n_{j=1}a_j,\,\,\,h(a)(0)^{\alpha\bar{\beta}}h(0)_{\alpha\bar{\beta}}=\sum^n_{j=1}a^{-1}_j,\,\,\, {\rm and}\,\,\,\Pi^n_{j=1}a_j=e^{f(p)}.
   $$
   It easily follows that at $p$ we have
   $$
   n-\Delta^{L}\varphi(a)\leq e^{f(p)}(h(a)(0)^{\alpha\bar{\beta}}h(0)_{\alpha\bar{\beta}})^{n-1}\leq e^{Q_1}C^{n-1}_6.
   $$
   Therefore at a maximum of $F$, we see that $$F\leq Q_1+(n-1)\log C_6-k\inf\varphi(a), $$ and as $\|\varphi(a)\|_{C^0}\leq Q_2$
   by Proposition \ref{Q_2 leq} this shows that $$F\leq Q_1+(n-1)\log C_6-kQ_2 $$ everywhere on $M^{2n}$.
    Substituting for $F$ and exponentiating gives
    $$
    0<n-\Delta^{L}\varphi(a)\leq C^{n-1}_6 \exp(Q_1+kQ_2+k\varphi(a))\leq C^{n-1}_6\exp(Q_1+2kQ_2)
    $$
   on $M^{2n}$.
   This gives an a priori estimate for $\|\Delta^{L}\varphi(a)\|_{C^0_0}$.
   But Proposition \ref{Joc prop 1}  gives estimates for $\|h(a)(0)_{c\bar{d}}\|_{C^0_0}$, $\|h(a)(0)^{c\bar{d}}\|_{C^0_0}$
   and $\|dJ_0d\varphi(a)\|_{C^0_0}$ in terms of upper bounds for $\|f\|_{C^0_0}$ and $\|\Delta^{L}\varphi(a)\|_{C^0_0}$,
   so (\ref{h estimate}) holds for some constants $c_1,c_2$ and $Q_3$.
   All the constants in this proof, including $c_1,c_2$ and $Q_3$, depends only on $M^{2n},\omega,J$ and $Q_1$.
   Notice that $g_0$ and $g_1$ are quasi isometric on $\mathring{M}^{2n}(a,P)$,
    $g_0$ and $g_1$ are equivalent (measurable) Lipschitz Riemannian metrics (that is, Lipschitz condition (\ref{metric equ1})-(\ref{metric equ0})) on $M^{2n}$, and $C^k_1$, $C^k_0$ are with equivalent norms (see (\ref{neq estimate})).
  Thus, we complete the proof of Proposition \ref{sec order}

\end{proof}

 {\bf  Third-order estimates.}

 Define a function  $S\geq 0$ on $\mathring{M}^{2n}(a,P)$ by $\|\nabla dJ_0d\varphi(a)\|^2_{g_0}$ on $\mathring{M}^{2n}(a,P)$,
  where $\nabla$ is the Levi-Civita connection with respect to $g_0$,
 so that in index notation
 $$
 S^2=h(a)(0)^{\alpha\bar{\lambda}}h(a)(0)^{\mu\bar{\beta}}h(a)(0)^{\gamma\bar{\nu}}\nabla_{\alpha\bar{\beta}\gamma}\varphi(a)\nabla_{\bar{\lambda}\mu\bar{\nu}}\varphi(a).
 $$
Our goal is to find an a priori upper bound for $S$,
and this will be done by finding a formula for $\Delta'^{L}(S^2)$ and then using an argument similar to that used in Proposition \ref{sec order}.
First, here is some notation that will be used for the next result.
 Suppose $A, B, C$ are tensors on $M^{2n}$. Let us write $P^{a,b,c}(A,B,C)$ for any polynomial in the tensors $A, B, C$
 alone, that is homogeneous of degree $a$ in $A$, degree $b$ in $B$ and degree $c$ in $C$.
 Using this notation, we shall give in the next proposition an expression for  $\Delta'^{L}(S^2)$.
 \begin{prop}\label{Delta S2}
 In the notation above, we have:
 \begin{eqnarray}\label{Delta S2 equ}
   -\Delta'^{L}(S^2) &=&|\nabla_{\bar{\alpha}\beta\bar{\gamma}\delta}\varphi(a)-h(a)(0)^{\lambda\bar{\mu}}
   \nabla_{\bar{\alpha}\lambda\bar{\gamma}}\varphi(a)\nabla_{\beta\bar{\mu}\delta}\varphi(a)|^2_{g_0(a)}  \nonumber\\
    &&+|\nabla_{\alpha\beta\bar{\gamma}\delta}\varphi(a)-h(a)(0)^{\lambda\bar{\mu}}
   \nabla_{\alpha\bar{\gamma}\lambda}\varphi(a)\nabla_{\beta\bar{\mu}\delta}\varphi(a)\nonumber\\
   &&-h(a)(0)^{\lambda\bar{\mu}}
   \nabla_{\alpha\bar{\mu}\delta}\varphi(a)\nabla_{\lambda\bar{\gamma}\beta}\varphi(a)|^2_{g_0(a)}  \nonumber\\
    &&+P^{4,2,1}(h(a)(0)^{\alpha\bar{\beta}},\nabla_{\alpha\bar{\beta}\gamma}\varphi(a),\nabla_{\alpha\bar{\beta}}f)
    +P^{3,1,1}(h(a)(0)^{\alpha\bar{\beta}},\nabla_{\alpha\bar{\beta}\gamma}\varphi(a),\nabla_{\bar{\alpha}\beta\bar{\gamma}}f).\nonumber\\
 \end{eqnarray}
 Where $h(a)(0)$ is the K\"{a}hler metric with respect to $(\omega(a),J_0(a))$ on $\mathring{M}^{2n}(a,P)$.
 \end{prop}
 Now there is a certain similarity
   between Proposition \ref{Delta S2} and Proposition \ref{Delta Delta}. Here is one way to see it.
   In (\ref{Delta equ0}), the left hand side $\Delta'^{L}(\Delta^{L}\varphi(a))$ involves $\nabla^4\varphi(a)$.
   However, the right hand side involves only $\nabla^2\varphi(a)$ and $\nabla^3\varphi(a)$, together with $h(a)(0),\Delta^{L}f $ and $R$.
    Moreover, the terms on the right hand side involving $\nabla^3\varphi(a)$ are nonnegative. In the same way, the left hand side of
    (\ref{Delta S2 equ}) involves $\nabla^5\varphi(a)$, but the right hand side involves only $\nabla^2\varphi(a)$ and $\nabla^3\varphi(a)$,
     and the terms on the right hand side involving $\nabla^4\varphi(a)$ are nonnegative.

     So, both (\ref{Delta equ0}) and (\ref{Delta S2 equ})  express a derivative of $\varphi(a)$ in terms of lower derivatives of $\varphi(a)$
     and $h(a)(0),f$ and $R$, and the highest derivative of  $\varphi(a)$ on the right hand side contributes only nonnegative terms.
     Now Proposition \ref{sec order} used (\ref{Delta equ0}) to find an a priori bound for $\|\Delta^{L}\varphi(a)\|_{C^0_1}$.
     Because of the similarities between (\ref{Delta equ0}) and (\ref{Delta S2 equ}), we will be able to use
 the same method to find an a priori bound for $S$, and hence for $\|\nabla dJ_0d\varphi(a)\|^2_{C^0_0}$.

   It follows quickly from (\ref{Delta S2 equ})  that
   \begin{col}\label{C7 inequ}
   There is a constant $C_7$ depending on $Q_1,c_1,c_2$  with
   $$
   \Delta'^{L}(S^2) \leq C_7(S^2+S).
   $$
   \end{col}
   Here $c_1, c_2$ are the constants from Proposition \ref{sec order}.
    To prove Corollary \ref{C7 inequ}, observe
 that the first two terms on the right hand side of (\ref{Delta S2 equ})  are nonnegative, and can be
 neglected. Of the four terms of type $P^{a,b,c}$, the first two are quadratic in $\nabla_{\alpha\bar{\beta}\gamma}\varphi(a)$ and
must be estimated by a multiple of $S^2$,  and the second two are linear in $\nabla_{\alpha\bar{\beta}\gamma}\varphi(a)$ and
must be estimated by a multiple of $S$. As $\|f\|_{C^3_1}\leq Q_1$, the factors of $\nabla_{\alpha\bar{\beta}}f$ and
$\nabla_{\bar{\alpha}\beta\bar{\gamma}}f$  are estimated using $Q_1$, and by using the estimates of Proposition \ref{sec order} for
$h(a)_{ab}$ and $h(a)^{ab}$, Corollary \ref{C7 inequ} quickly follows.

The next result, together with Proposition \ref{Q_2 leq} and Proposition \ref{sec order},
 completes the proof of  Theorem \ref{main the1}.

\begin{prop}\label{thi order}
$\|\nabla dJ_0d\varphi(a)\|_{C^0_1}\leq Q_4$ for $Q_4$ depending only on $M^{2n},\omega,J$ and $Q_1$.
\end{prop}
\begin{proof}
Using the formulae for $h(a)(0)$ in Appendix \ref{App B}, it is easy to show that
$$
h(0)^{\alpha\bar{\lambda}}h(a)(0)^{\mu\bar{\beta}}h(a)(0)^{\gamma\bar{\nu}}\nabla_{\alpha\bar{\beta}\gamma}\varphi(a)\nabla_{\bar{\lambda}\mu\bar{\nu}}\varphi(a)
\geq c_2^{-1}S^2,
$$
where $c_2>0$ comes from Proposition \ref{sec order}.
So Proposition  \ref{Delta Delta} gives
\begin{equation}\label{C8 inequ}
  \Delta'^{L}(\Delta^{L}\varphi(a))\geq c_2^{-1}S^2-C_8,
\end{equation}
where $C_8$ depending only on $M^{2n},\omega,J$ and $Q_1$.
Consider the function $S^2-2c_2C_7\Delta^{L}\varphi(a)$ on $M^{2n}$.
From Corollary \ref{C7 inequ} and (\ref{C8 inequ})  it follows that
\begin{eqnarray*}
   \Delta'^{L}(S^2-2c_2C_7\Delta^{L}\varphi(a))&\leq& C_7(S^2+S)-2c_2C_7(c^{-1}_2S^2-C_8) \\
   &=& -C_7(S-\frac{1}{2})^2+2c_2C_7C_8+\frac{1}{4}C_7.
\end{eqnarray*}
  At a maximum $p$ of $S^2-2c_2C_7\Delta^{L}\varphi(a)$, we have $ \Delta'^{L}(S^2-2c_2C_7\Delta^{L}\varphi(a))\geq 0$, and thus
  $$
  (S-\frac{1}{2})^2\leq2c_2C_8+\frac{1}{4}
  $$
 at $p$. Soo, there is a constant $C_9 > 0$ depending on $c_2, C_8$ and the a priori estimate for $\|\Delta^{L}\varphi(a)\|_{C^0_0}$
  found in  Proposition \ref{sec order}, such that $S^2-2c_2C_7\Delta^{L}\varphi(a)\leq C_9$ at $p$.
  As $p$ is a maximum, $S^2-2c_2C_7\Delta^{L}\varphi(a)\leq C_9$ holds on $M^{2n}$.
  Using the estimate on $\|\Delta^{L}\varphi(a)\|_{C^0_0}$ again,
  we find an a priori estimate for $\|S\|_{C^0_0}$ .

  Now $2S=|\nabla dJ_0d\varphi(a)|_{g_0}$.
  It  can be seen that $$|\nabla dJ_0d\varphi(a)|_{h(0)}\leq c_1^{\frac{3}{2}}|\nabla dJ_0d\varphi(a)|_{g_0},$$
  where $c_1$ comes from Proposition \ref{sec order}.
  Thus, an a  priori bound for $\|S\|_{C^0_0}$ gives one for $\|\nabla dJ_0d\varphi(a)\|_{C^0_0}$,
  and there is $Q_4$ depending only on $M^{2n},\omega,J$ and $Q_1$ with $$\|\nabla dJ_0d\varphi(a)\|_{C^0_0}\leq Q_4.$$
 Notice that $g_0$ and $g_1$ are quasi isometric on $\mathring{M}^{2n}(a,P)$,
    $g_0$ and $g_1$ are equivalent (measurable) Lipschitz Riemannian metrics on $M^{2n}$ which satisfy Lipschitz condition (\ref{metric equ}), and $C^0_1$, $C^0_0$ are with equivalent norms (that is, with (\ref{neq estimate}) ).
  Thus, we complete the proof of Proposition \ref{thi order} and Theorem \ref{main the1}.
 \end{proof}

 \vskip 6pt

 {\bf The proof of Theorem \ref{main the2}.}

 Suppose that $V$ is a bounded, simply connected open subset of $(M^{2n},\omega,J,g_J)$ which is diffeomorphic to a contractible  relatively compact and strictly pseudoconvex domain
 in $\mathbb{C}^n$.
 Let us begin by stating three elliptic regularity results for $\Delta^{L}$ and $\Delta'^{L}$ on $V$ with respect to the metric $g_0$,
 which come from $\S1.4$ in D.D. Joyce \cite[Theorem 1.4.1]{Joc}.

 \begin{lem}\label{regularity lem1}
 Let $k\geq 0$ and $\alpha\in (0,1)$.
 Then there exists a constant $E_{k,\alpha}>0$ depending on $k,\alpha,M^{2n},\omega$ and $J$, such that if $\psi\in C^2_0(V)$,
 $\xi\in C^{k,\alpha}_0(V)$ and $\Delta^{L}\psi=\xi$, then $\psi\in C^{k+2,\alpha}_0(V)$
 and $\|\psi\|_{ C^{k+2,\alpha}_0}\leq E_{k,\alpha}(\|\xi\|_{C^{k,\alpha}_0}+\|\psi\|_{C^0_0})$.
 \end{lem}

   \begin{lem}\label{regularity lem2}
 Let $\alpha\in (0,1)$.
 Then there exists a constant $E'_{\alpha}>0$ depending on $\alpha,M^{2n},\omega,J$ and the norms $\|h(a)(0)_{cd}\|_{C^0_0}$ and  $\|h(a)(0)^{cd}\|_{C^{0,\alpha}_0}$,
 such that if $\psi\in C^2_0(V)$,
 $\xi\in C^0_0(V)$ and $\Delta'^{L}\psi=\xi$, then $\psi\in C^{1,\alpha}_0(V) $
 and $\|\psi\|_{ C^{1,\alpha}_0}\leq E'_{\alpha}(\|\xi\|_{C^0_0}+\|\psi\|_{C^0_0})$.
 \end{lem}

 \begin{lem}\label{regularity lem3}
 Let $k\geq 0$ be an integer, and $\alpha\in (0,1)$.
 Then there exists a constant $E'_{k,\alpha}>0$ depending on $k,\alpha,M^{2n},\omega,J$ and the norms $\|h(a)(0)_{cd}\|_{C^0_0}$ and  $\|h(a)(0)^{cd}\|_{C^{k,\alpha}_0}$, such that if $\psi\in C^2_0(V)$,
 $\xi\in C^{k,\alpha}_0(V)$ and $\Delta'^{L}\psi=\xi$, then $\psi\in C^{k+2,\alpha}_0(V) $
 and $\|\psi\|_{C^{k+2,\alpha}_0}\leq E'_{k,\alpha}(\|\xi\|_{C^{k,\alpha}_0}+\|\psi\|_{C^0_0})$.
 \end{lem}

 The proof of Theorem \ref{main the2} is based on equation (\ref{Delta equ0}) of Proposition \ref{Delta Delta}
 and $g_0$, $g_1$ are quasi isometry on $\mathring{M}^{2n}(a,P)$ (see (\ref{metric equ1})-(\ref{metric equ0})).
 Applying Lemma \ref{regularity lem1}, \ref{regularity lem2} and \ref{regularity lem3} to this equation, we prove the theorem by an inductive
  process known as `bootstrappin'. First we find an a priori estimate for $\|\varphi(a)\|_{C^{3,\alpha}_0}$.
  \begin{prop}\label{D1 inequ}
  There is a constant $D_1$ depending on $M^{2n}, \omega, J, Q_1,\cdot\cdot\cdot,Q_4$ and $\alpha$
  such that $\|\varphi(a)\|_{C^{3,\alpha}_0}\leq D_1$.
  \end{prop}
 \begin{proof}
 In this proof, all estimates and all constants depend only on $M^{2n}, \omega, J, Q_1,\cdot\cdot\cdot,Q_4$ and $\alpha$.
 By Proposition \ref{sec order}, $\|dJ_0d\varphi(a)\|_{C^0_0}\leq Q_3$ implies $\|h(a)(0)_{cd}\|_{C^0_0}\leq c_1$ and $\|h(a)(0)^{cd}\|_{C^0_0}\leq c_2$.
 Also, $\nabla h(a)(0)^{cd}$ can be expressed in terms of $\nabla dJd\varphi(a)$, $h(a)(0)^{cd}$ and $J$,
 so that the estimates for $\|h(a)(0)^{cd}\|_{C^0_0}$ and $\|\nabla dJ_0d\varphi(a)\|_{C^0_0}$ yield an estimate for $\|\nabla h(a)(0)^{cd}\|_{C^0_0}$.
  Combining the estimates for $\|h(a)(0)^{cd}\|_{C^0_0}$ and $\|\nabla h(a)(0)^{cd}\|_{C^0_0}$, we can find an estimate for $\|h(a)(0)^{cd}\|_{C^{0,\alpha}_0}$.

  By the construction of $\mathring{M}^{2n}(a,P)$, we choose finite cover of $M^{2n}$, $V_i(a)$, $1\leq i\leq N(a)$.
  Each $V_i(a)$ is a contractible, relatively compact, and strictly pseudoconvex domain in $\mathbb{C}^n$.
  Note that if $V_i(a)\cap V_j(a)=\varnothing$, $\varphi_i(a)-\varphi_j(a)$ is a constant on $V_i(a)\cap V_j(a)$ by Lemma \ref{Joc lemma 1'}.
  Thus there are a priori estimates for $\|h(a)(0)_{cd}\|_{C^0_0}$ and $\|h(a)(0)^{cd}\|_{C^{0,\alpha}_0}$.
   Therefore, Lemma \ref{regularity lem2}  holds with a constant $E'_{\alpha}$  depend on $M^{2n}, \omega, J, Q_1,\cdot\cdot\cdot,Q_4$ and $\alpha$.
   Put $\psi=\Delta^{L}\varphi(a)$, and let $\xi$ be the right hand side of  (\ref{Delta equ0}), so that $\Delta^{L}\psi=\xi$.
    Now, combining a priori estimates of the $C^0_0$ norms of $h(a)(0)^{cd}$, $\nabla_{\alpha\bar{\beta}}\varphi(a)$ and $\nabla_{\alpha\bar{\beta}\gamma}\varphi(a)$, the inequality $\|f\|_{C^2_1}\leq Q_1$ , and
 a bound for $R^a_{bcd}$ (in fact, here, $R^a_{bcd}=0$), we can find a constant $D_2$ such that $\|\xi\|_{C^0_0}\leq D_2$.

 So, by Lemma \ref{regularity lem2}, $\Delta^{L}\varphi(a)=\psi$ lies in $C^{1,\alpha}_0(M^{2n})$ and
 $\|\Delta^{L}\varphi(a)\|_{C^{1,\alpha}_0}\leq E'_\alpha(D_2+Q_3)$, as $\|\Delta^{L}\varphi(a)\|_{C^0_0}\leq \|dJ_0d\varphi(a)\|_{C^0_0}\leq Q_3$.
  Therefore by Lemma \ref{regularity lem1}, $\varphi(a)\in C^{3,\alpha}_0(M^{2n})$ and
  $$
  \|\varphi(a)\|_{C^{3,\alpha}_0}\leq E_{1,\alpha}(\|\Delta^{L}\varphi(a)\|_{C^{1,\alpha}_0}+\|\varphi(a)\|_{C^0_0})\leq E_{1,\alpha}(E'_{\alpha}(D_2+Q_3)+Q_2).
  $$
  Thus, putting $D_1=E_{1,\alpha}(E'_{\alpha}(D_2+Q_3)+Q_2)$, the proof of this proposition is completed.
 \end{proof}

  Theorem \ref{main the2} will follow from the next proposition.
  \begin{prop}\label{k3 estimate}
   For each $k\geq 3$, if $f\in C^{k,\alpha}(M^{2n})$ then $\varphi(a)\in  C^{k+2,\alpha}_0( \mathring{M}^{2n}(a,P))$, and there
   exists an a priori bound for $\|\varphi(a)\|_{C^{k+2,\alpha}_0}$  depend on $M^{2n}, \omega, J, Q_1,\cdot\cdot\cdot,Q_4,k,\alpha$
   and a bound for $\|f\|_{C^{k,\alpha}_1}$.
  \end{prop}
  \begin{proof}
  The proof is by induction on $k$. The result is stated for $k\geq 3$ only, because
   Theorem  \ref{main the1} uses $\|f\|_{C^3_1}$ to bound $\|\nabla dJ_0d\varphi(a)\|_{C^0_0}$.
   However, it is convenient to start the
     induction at $k=2$. In this proof, we say a constant `depends on the $k$-data' if it depends
  only on $M^{2n}, \omega, J, Q_1,\cdot\cdot\cdot,Q_4,k,\alpha$
   and  bounds for $\|f\|_{C^{k,\alpha}_1}$ and $\|f\|_{C^3_1}$ .
   Our inductive hypothesis is that $f\in C^{k,\alpha}_1(M^{2n})$ and $\varphi(a)\in C^{k+1,\alpha}_0( \mathring{M}^{2n}(a,P))$,
   and that there is an a priori bound for $\|\varphi(a)\|_{C^{k+1,\alpha}_0}$ depending on the $(k-1)$-data.
    By Proposition \ref{D1 inequ}, this holds for
 $k=2$, and this is the first step in the induction.
   Write $\psi=\Delta^{L}\varphi(a)$, and let $\xi$ be the right hand side of (\ref{Delta equ0}), so that
   $\Delta'^{L}\psi=\xi$. Now the term $-\Delta^{L}f$ on the right hand side of (\ref{Delta equ0}) is bounded in
   $C^{k-2,\alpha}(M^{2n})$ by $\|f\|_{C^{k,\alpha}_1}$, and hence in terms of the $k$-data.
    It is easy to see that every other term on the right
 hand side of (\ref{Delta equ0}) can be bounded in   $C^{k-2,\alpha}(M^{2n})$ in terms of $M^{2n},\omega,J$
 and bounds for $\|h(a)(0)^{cd}\|_{C^0_0}$ and $\|\varphi(a)\|_{C^{k+1,\alpha}_0}$.
    By the inductive hypothesis, these are all bounded in terms of
the $(k-1)$-data. Therefore, we can find a bound $F_{k,\alpha}$ depending on the $k$-data, such that
   $\|\xi\|_{C^{k-2,\alpha}_0}\leq F_{k,\alpha}$.

   Using the inductive hypothesis again, we may bound $\|h(a)(0)_{cd}\|_{C^0_0}$ and $\|h(a)(0)^{cd}\|_{C^{k,\alpha}_0}$
    in terms of the $(k-1)$-data.
    Therefore we may apply Lemma \ref{regularity lem3} to $\psi=\Delta'^{L}\varphi(a)$, which shows that
    $\Delta^{L}\varphi(a)\in C^{k,\alpha}_0( \mathring{M}^{2n}(a,P))$ and
    $$
    \|\Delta^{L}\varphi(a)\|_{C^{k,\alpha}_0}\leq E'_{k-2,\alpha}(\|\xi\|_{C^{k-2,\alpha}_0}+\|\Delta^{L}\varphi(a)\|_{C^0_0})\leq E'_{k-2,\alpha}(F_{k,\alpha}+Q_3),
    $$
    since $\|\Delta^{L}\varphi(a)\|_{C^0_0}\leq Q_3$, where $E'_{k-2,\alpha}$ depends on the $(k-1)$-data. Thus by Lemma \ref{regularity lem1},
    $\varphi(a)\in C^{k+2,\alpha}( \mathring{M}^{2n}(a,P))$ as we have to prove, and
    $$
     \|\varphi(a)\|_{C^{k+2,\alpha}_0}\leq   E_{k,\alpha}(\|\Delta^{L}\varphi(a)\|_{C^{k,\alpha}_0}+\|\varphi(a)\|_{C^0_0})\leq  E_{k,\alpha}(E'_{k-2,\alpha}(F_{k,\alpha}+Q_3)+Q_2),
    $$
    since $\|\varphi(a)\|_{C^0_0}\leq Q_2$.  This is an a priori bound for $\|\varphi(a)\|_{C^{k+2,\alpha}_0}$ depending only on the
 $k$-data. Therefore by induction, the result holds for all $k$.
  \end{proof}
  Now  we prove Theorem \ref{main the2}.

   {\bf Proof of Theorem \ref{main the2}.}
   Notice that $g_0$ and $g_1$ are quasi isometry on  $\mathring{M}^{2n}(a,P)$, and with equivalent (measurable) Lipschitz Riemannian metrics on $M^{2n}$ (see Lipschitz contidion (\ref{metric equ})).
   The space $C^{k,\alpha}_i$, $i=0,1$, for $\alpha\in[0,1)$ and $k$ nonnegative integer defined on the open dense submanifold $\mathring{M}^{2n}(a,P)$
   of $M^{2n}$ which are with equivalent norms (see (\ref{neq estimate})), and can be extended to  $M^{2n}$.
   First put $k=3$ in Proposition \ref{k3 estimate}.
  This shows that $\varphi(a)\in C^{5,\alpha}_1( \mathring{M}^{2n}(a,P))$ and gives an a priori bound for  $\|\varphi(a)\|_{C^{5,\alpha}_1}$.
  Let $Q_5$ be this bound. Then $\|\varphi(a)\|_{C^{5,\alpha}_1}\leq Q_5$, and $Q_5$ depending only on
   $M^{2n}, \omega, J, Q_1,\cdot\cdot\cdot,Q_4,k$, and $\alpha$, since $Q_1$ is a bound for $\|f\|_{C^{3,\alpha}_1}$.
   Also, if $k\geq 3$ and $f\in C^{k,\alpha}_1(M^{2n})$, Proposition \ref{k3 estimate} shows that  $\varphi(a)\in C^{k+2,\alpha}_1( \mathring{M}^{2n}(a,P))$.
   Since this holds for all $k$, if $f\in C^{\infty}_1(M^{2n})$ then $\varphi(a)\in C^{\infty}_1(\mathring{M}^{2n}(a,P))$, and the theorem is proved. \qed

\section{The differentiability and the existence of solutions of the Calabi-Yau equation }\label{existence}
 \setcounter{equation}{0}
 This section is devoted to proving Theorem \ref{main the3}.

 First, we consider the Gromov type Calabi-Yau equation on closed almost K\"{a}hler manifold $(M^{2n},\omega,J,g_J)$, where coefficients of
 the almost  K\"{a}hler structure $(\omega,J,g_J)$ are in $C^{k,\alpha}$, $k\in\mathbb{N}$, $k\geq 3$, $\alpha\in(0,1)$.
 Find a function $\phi\in C^{k+2,\alpha}(M^{2n},\mathbb{R})$ with respect to the metric $g_J$ satisfying the following conditions:
\begin{equation}
   \left\{
  \begin{array}{ll}
    \omega(\phi)^n =\sigma,  \\

  \omega(\phi)=\omega+dJd\phi \,\,\, {\rm tames}\,\,\, J,
  \end{array}
  \right.
  \end{equation}
  where $\sigma=e^F\omega^n$ is a given volume form in $[\omega^n]$,
  that is
  \begin{equation}\label{CY int equ}
  \int_{M^{2n}}e^F\omega^n=\int_{M^{2n}}\omega^n,
  \end{equation}
  $F\in C^{k,\alpha}(M^{2n},\mathbb{R})$.
  As done in Section \ref{local CY}, we need some notations (cf. \cite{Dela,WZ}).
  \begin{defi}
  The sets $A(k,\alpha)$, $B(k,\alpha)$, $A_+(k,\alpha)$ and $B_+(k,\alpha)$ are defined as follows:
  $$
  A(k,\alpha):=\{\phi\in C^{k+2,\alpha}(M^{2n},\mathbb{R}) \mid  \int_{M^{2n}}\phi\omega^n=0\};
  $$
   $$
  B(k,\alpha):=\{f\in C^{k,\alpha}(M^{2n},\mathbb{R}) \mid  \int_{M^{2n}}f\omega^n=\int_{M^{2n}}\omega^n\};
  $$
  $$
  A_+(k,\alpha):=A(k,\alpha)\cap\{\phi\in C^{k+2,\alpha}(M^{2n},\mathbb{R}) \mid  \omega(\phi)\,\,tames\,\, J\};
  $$
   $$
  B_+(k,\alpha):= B(k,\alpha)\cap\{f\in C^{k,\alpha}(M^{2n},\mathbb{R}) \mid f>0\}.
  $$
  \end{defi}
   It is easy to see that $A(k,\alpha)$ and  $B(k,\alpha)$ are Banach manifolds,
   $A_+(k,\alpha)$ is an open convex set of $A(k,\alpha)$.
   As done in smooth case (see Section \ref{local CY}, Delano\"{e} \cite{Dela} or Wang-Zhu \cite{WZ} ),
   we define an operator $\mathcal{F}$ from $A(k,\alpha)$ to $B(k,\alpha)$ as follows:
   \begin{equation}\label{}
     \phi\mapsto\mathcal{F}(\phi),
   \end{equation}
where
 \begin{equation}\label{}
   \mathcal{F}(\phi)\omega^n=\omega(\phi)^n, \,\,\,\omega(\phi)=\omega+dJd\phi.
   \end{equation}
Restricting the operator $\mathcal{F}$ to $A_+(k,\alpha)$, we get $\mathcal{F}(A_+(k,\alpha))\subset B_+(k,\alpha)$.
Thus, the existence of a solution to Equation (\ref{CY int equ}) is a equivalent to that the restricted operator
\begin{equation}\label{}
 \mathcal{F}|_{A_+(k,\alpha)}:A_+(k,\alpha)\rightarrow B_+(k,\alpha)
\end{equation}
is surjective.

Let $[\frac{n}{2}]$ denote the integral part of positive number $\frac{n}{2}$.
Following Section \ref{local CY}, let us define operator  $\mathcal{F}_j$ $(j=0,1,\cdot\cdot\cdot,[\frac{n}{2}])$
as follows:
$$
\mathcal{F}_j:A(k,\alpha)\rightarrow B(k,\alpha)
$$
$$
\phi\mapsto \mathcal{F}_j(\phi),
$$
such that
\begin{equation}\label{}
  \sum^{[\frac{n}{2}]}_{j=0}\mathcal{F}_j(\phi)\omega^n=\omega(\phi)^n=\mathcal{F}(\phi)\omega^n.
\end{equation}
We have the following result:
\begin{prop}
(cf. Proposition \ref{F inequ} )
$\mathcal{F}(\phi)=\mathcal{F}_0(\phi)+\mathcal{F}_1(\phi)+\cdot\cdot\cdot+\mathcal{F}_{[\frac{n}{2}]}(\phi)$

(1) $\phi\in A_+(k,\alpha)$, then $\mathcal{F}_0(\phi)>0$ and $\mathcal{F}_j(\phi)\geq0$ for $j=1,\cdot\cdot\cdot,[\frac{n}{2}]$,
hence $\mathcal{F}(\phi)>0$, in particular, $\mathcal{F}(0)=1$;

(2)Define $B_{\varepsilon}(0):=\{\phi\in A(k,\alpha)\mid \|\phi\|_{C^2_1}\leq\varepsilon\}$,
where $C^2_1$ is $C^2_1$-norm on $C^2(M^{2n})$ induced by the metric $g_J=g_1$ (cf. Appendix \ref{App A} ),
if $\varepsilon<< 1$, then $B_{\varepsilon}(0)\subset  A_+(k,\alpha)$.
\end{prop}

For any $\phi\in A_+(k,\alpha)$, it is easy to see that the tangent at $T_{\phi}A_+(k,\alpha)$ is $A$.
For $u\in T_{\phi}A_+(k,\alpha)=A(k,\alpha)$, define $L(\phi)(u)$ by
\begin{equation}\label{}
 L(\phi)(u)=\frac{d}{dt}\mathcal{F}(\phi+tu)\mid_{t=0},
\end{equation}
thus $L(\phi)(u)\in B_0(k,\alpha)$,
where
\begin{equation}\label{}
  B_0(k,\alpha):=\{f\in C^{k,\alpha}_1(M^{2n},\mathbb{R})\mid \int_{M^{2n}}f\omega^n=0\},
\end{equation}
where $C^{k,\alpha}_1$ is the H\"{o}lder space for $\alpha\in(0,1)$ and $k$ is a nonnegative integer with respect to the metric $g_1=g_J$ (cf. Appendix \ref{App A}).
$L(\phi)$ is a linear elliptic operator of second  order without zero-th term.
Notice that $B(k,\alpha)=1+B_0(k,\alpha)$ is an affine space.
Therefore, we have the following lemma:
\begin{lem}\label{restricted operator}
(cf. Wang-Zhu \cite[Proposition 1]{WZ} or Delano\"{e} \cite[Lemma 2.5]{Dela})
The restricted operator
$$
\mathcal{F}: A_+(k,\alpha)\rightarrow B_+(k,\alpha)
$$
is elliptic type on $A_+(k,\alpha)$.
Moreover, the tangent map,
$$
d\mathcal{F}=L(\phi): A(k,\alpha)\rightarrow B_0(k,\alpha),
$$
of $\mathcal{F}$ at $\phi\in A_+(k,\alpha)$ is a linear elliptic differential operator of second order without zero-th term which is invertible.
\end{lem}
By Lemma \ref{restricted operator}, we have the following theorem:
\begin{theo}
(cf. \cite[Chapter 3]{Au2})
Suppose that $(M^{2n},\omega,J,g_J)$ is a closed, $C^{k,\alpha}$ almost K\"{a}hler manifold, $k\geq 3$, $\alpha\in(0,1)$.
For any $\phi_0\in  A_+(k,\alpha)$, that is, symplectic form $\omega(\phi_0)$ taming $J$, then $\mathcal{F}(\phi_0)=f_0\in B_+(k,\alpha)$
(in particular, $\phi_0=0$, $\omega(0)=\omega$, $\mathcal{F}(0)=1$).
Hence, there exists a constant $\varepsilon(\omega,\phi_0,J,M^{2n})<< 1$ such that if $$\|f_0-f\|_{C^{k,\alpha}_1(M^{2n},g_J)}<\varepsilon(\omega,\phi_0,J,M^{2n}),$$
then exists a unique $\phi(f)\in A_+(k,\alpha)$ satisfying $\mathcal{F}(\phi(f))=f$,
and $\omega(\phi(f))$ taming almost complex structure $J$.
\end{theo}
In the remainder of this section, we give a proof of Theorem \ref{main the3}.
Notice that $\omega+da$ is a symplectic form on $M^{2n}$, we make a global deformation of almost complex structures on $M^{2n}$
 off a Lebesgue measure zero subset, $M^{2n}\setminus \mathring{M}^{2n}(a,P)$, of $M^{2n}$, where $\mathring{M}^{2n}(a,P)$ is an open dense
 submanifold of $M^{2n}$.
 Let $g_1=g_J$, $g_0$ is a K\"{a}hler flat metric on  $\mathring{M}^{2n}(a,P)$.
 In particular, $g_0$ and $g_1$ are quasi isometry on $\mathring{M}^{2n}(a,P)$, and  equivalent (measurable) Lipscchitz Riemannian metrics on $M^{2n}$ (cf. (\ref{metric equ})).

 {\bf Proof of Theorem \ref{main the3}.}
Let $(\omega,J,g_J)$ and $(\omega(a'),J(a'),g_{J(a')})$  are two $C^{k,\alpha}_1$-almost K\"{a}hler structures on closed manifold $M^{2n}$,
$k\geq 3$, $\alpha\in(0,1)$.
Where $$f'\in C^{k,\alpha}(M^{2n},g_J), \,\,\,da'\in C^{k,\alpha}(\Lambda^2T^*M^{2n},g_J),\,\,\,\varphi(a')\in C^{k+2,\alpha}(\mathring{M}^{2n}(a',P'),g_J)$$
and $A'>0$ satisfying the following equations
$$
\int_{M^{2n}}\varphi(a')\omega^n=0,\,\,\,-\sqrt{-1}da'|_{\mathring{M}^{2n}(a',P')}=\partial\bar{\partial}\varphi(a'),
$$
$$
\omega(a')^n=(\omega+da')^n=A'e^{f'}\omega^n,\,\,\, A'>0,
$$
and
$$
\int_{M^{2n}}\omega(a')^n=\int_{M^{2n}}A'e^{f'}\omega^n=\int_{M^{2n}}\omega^n.
$$
Notice that, in general, $\omega(a'):=\omega+da'$ is not taming the almost complex structure $J$.
But there exists an almost complex structure $J(a')$ on $M^{2n}$ such that $J(a')$ is $\omega(a')$ compatible, hence
$g_{J(a')}(\cdot,\cdot)=\omega(a')(\cdot,J(a')\cdot)$ is an almost K\"{a}hler metric.
As done in before, we define
$$
A(k,\alpha,a'):=\{\phi'\in C^{k+2,\alpha}(M^{2n},\mathbb{R})\mid \int_{M^{2n}}\phi'(\omega+da')^n=0\};
$$
$$
B(k,\alpha,a'):=\{f'\in C^{k,\alpha}(M^{2n},\mathbb{R})\mid \int_{M^{2n}}f'\omega(a')^n=\int_{M^{2n}}\omega^n\};
$$
$$
A_+(k,\alpha,a')=A(k,\alpha,a')\cap\{\phi'\in C^{k+2,\alpha}(M^{2n},\mathbb{R})\mid \omega(a')(\phi')=\omega+da'+dJ(a')d\phi'\,\,\, tames \,\,\,J(a')\};
$$
and
$$
B_+(k,\alpha,a')=B(k,\alpha,a')\cap\{f'\in C^{k,\alpha}(M^{2n},\mathbb{R})\mid f'>0\}.
$$
By Lemma \ref{restricted operator}, there is a constant
$\varepsilon(\omega,J,a',J(a'),f',M^{2n})<<1$
if
\begin{equation}\label{}
  \|f'-f\|_{C^{k,\alpha}(M^{2n},g_{J(a)})}<\varepsilon(\omega,J,a',J(a'),f',M^{2n})
\end{equation}
and
\begin{equation}\label{}
 \int_{M^{2n}}\omega(a')^n=\int_{M^{2n}}Ae^{f}\omega^n=\int_{M^{2n}}\omega^n,
\end{equation}
where $A>0$, then there exists a unique $\phi(f)\in A_+(k,\alpha,a')$ satisfying
\begin{equation}\label{}
(\omega+da'+dJ(a')d\phi(f))^n=Ae^f\omega(a')^n.
\end{equation}
Since $M^{2n}$ is a closed, oriented manifold of dimension $2n$, it is easy to see that
$g_J$ and $g_{J(a')}$ are quasi-isomery.
Let
\begin{equation}\label{}
  a=a'+J(a')d\phi(f),
\end{equation}
we have the following result, if
\begin{equation}\label{}
   \|f'-f\|_{C^{k,\alpha}(M^{2n},g_{J(a')})}<\varepsilon(\omega,J,a',J(a'),f',M^{2n}),
\end{equation}
then
\begin{equation}\label{}
  (\omega+da)^n=Ae^f\omega(a')^n,
\end{equation}
where $\omega+da=\omega+da'+dJ(a')d\phi(f)$ is a $C^{k,\alpha}$-symplectic form on $M^{2n}$ which tames the almost complex structure $J(a')$ (this means that $a-a'$ is very small),
and $a\in C^{k+1,\alpha}(T^*M^{2n})$ which is a $1$-form.
By Appendix \ref{App A}, we can construct an open dense submanifold $\mathring{M}^{2n}(a,P)$ of $M^{2n}$
where $M^{2n}\backslash \mathring{M}^{2n}(a,P)$ has Hausdorff dimension $\leq 2n-1$ with Lebesgue zero.
Also we can find a K\"{a}hler potential $\varphi(a)$ on $\mathring{M}^{2n}(a,P)$ such that
\begin{equation}\label{}
 -\sqrt{-1}da|_{\mathring{M}^{2n}(a,P)}=\partial\bar{\partial}\varphi(a),
\end{equation}
and $\varphi(a)\in C^{k+2,\alpha}(\mathring{M}^{2n}(a,P),g_J)$.
By Theorem \ref{main the2}, we have if
$$\|Ae^f\|_{C^{3,\alpha}_1}\leq Q_1, \,\,\, \|\varphi(a)\|_{C^0_1}\leq Q_2, \,\,\, \|da\|_{C^0_1}\leq Q_3, \,\,\, \|\nabla(da)\|_{C^0_1}\leq Q_4,$$
then $\|\varphi(a)\|_{C^{5,\alpha}_1}\leq Q_5$, and $\|a\|_{C^{4,\alpha}_1}\leq Q_6$.
Where $Q_1,\cdot\cdot\cdot,Q_6$ are defined in Section \ref{global CY}.
This completes the proof of Theorem \ref{main the3}.

\begin{rem}
Our proof of Theorem \ref{main the3} is similar to the proof of Theorem $C3$ in Joyce \cite[Chapter 6]{Joc}.
The key ingredient is the local existence of the symplectic potential for Gromov type Calabi-Yau equation.
\end{rem}

\begin{appendices}
\renewcommand{\theequation}{A.\arabic{equation}}
\setcounter{equation}{0}
\section{Measurable K\"{a}hler flat structures on symplectic manifolds}\label{App A}
     The appendix is devoted to constructing a measurable K\"{a}hler flat structure on symplectic manifold
     which is quasi isometric to an almost K\"{a}hler structure.

  We now devote to constructing measurable K\"{a}hler structures
  on closed almost K\"{a}hler manifolds off a Lebesgue measure zerosubset (cf. Fang-Wang \cite[Section 2]{FW}) which satisfy Lipschitz condition with respect to the given
  almost  K\"{a}hler metrics.
  Suppose that $(M^{2n}, \omega)$ is a closed symplectic manifold of dimension $2n$.
  Let $J$ be an $\omega$-compatible almost complex structure on $M^{2n}$, and $g_J(\cdot,\cdot):=\omega(\cdot,J\cdot)$.
  Then $(\omega,J,g_J)$ is an almost K\"{a}hler structure on $M^{2n}$.
  We choose a Darboux's coordinate atlas $\{U_\alpha,\varphi_\alpha\}_{\alpha\in\Lambda}$ such that
  $$\varphi_\alpha: U_\alpha\rightarrow \varphi_\alpha(U_\alpha)\subseteq\mathbb{R}^{2n}\cong\mathbb{C}^n$$ is a diffeomorphism
  and $\varphi^*_\alpha\omega_0=\omega|_{U_\alpha}$,
  where $$\omega_0=\frac{\sqrt{-1}}{2}\sum^n_{i=1}dz_i\wedge d\bar{z_i}$$ is the standard K\"{a}hler form
  on $\mathbb{C}^n$ with respect to the standard complex structure $J_{st}$ on $\mathbb{C}^n$ (more details, see McDuff-Salamon \cite{MS}).
  We may assume that, without loss of generality,
  $0\in \varphi_\alpha(U_\alpha)$ is contractible and relatively compact, strictly pseudoconvex domain (cf. L. H\"{o}mander \cite{Hor}).

  Let $\mathcal{A}(\omega)$ be the set of all associated metrics with respect to the given symplectic form on $M^{2n}$,
  $\mathcal{H}$ the set of all metrics with the same volume element $\frac{\omega^n}{n!}$ which is totally geodesics in the set $\mathcal{M}$ of
  all Riemannian metrics on $M^{2n}$.
  Note that on a compact $M^{2n}$, the set $\mathcal{M}$ of all Riemannian metrics may be given a Riemannian metric (cf. D. Ebin \cite{Ebin}).
  For symmetric tensor fields $S,T$ of type $(2,0)$,
  $$
  <S,T>_g=\int_{M^{2n}}S_{ij}T_{kl}g^{ik}g^{jl}dvol_g.
  $$
  The geodesics in $\mathcal{M}$ are the curves $ge^{St}$.
  $g_t:=ge^{St}$ is computed by $g_t(X,Y)=g(X,e^{g^{-1}St}Y)$, where $e^{g^{-1}St}$ acts on $Y$ as a tensor field of type $(1,1)$.
  Then $\mathcal{A}(\omega)$ is a totally  geodesics in  $\mathcal{H}$ in the sense that if $h$ is $J$-anti-invariant symmetric
  $(2,0)$-tensor field, where $J$ is an $\omega$-compatible almost complex structure on $M^{2n}$ and
  $g(\cdot,\cdot)=\omega(\cdot,J\cdot)$, then $g_t=ge^{ht}$ lies in $\mathcal{A}(\omega)$.
  In particular, two metrics in $\mathcal{A}(\omega)$ may be jointed by a unique geodesic (cf. D.E.  Blair \cite[Propositions in p.304,p.307]{Bla}).

  On the Darboux's coordinate charts, one can make local deformations of $\omega$-compatible almost complex structures
  (cf. D.E. Blair \cite{Bla}, also see Tan-Wang-Zhou-Zhu\cite{TWZ,TWZZ} and Fang-Wang \cite[Section 2]{FW}).
  Let $g_{\alpha}(0)(\cdot,\cdot)=\omega(\cdot,\varphi^*_{\alpha}J_{st}\cdot)$ on $U_\alpha$,
   where $\varphi^*_{\alpha}J_{st}=J_{\alpha}(0)$ is integrable,
   where $J_{st}$ is the standard complex structure on $\mathbb{C}^n$.
  Define $$g_{\alpha}(t)(\cdot,\cdot):= g_{\alpha}(0)e^{th_\alpha}= g_{\alpha}(0)(\cdot,e^{tg^{-1}_{\alpha}(0)h_\alpha}\cdot),$$
   that is,
   \begin{equation}\label{defor metric}
     g_{\alpha}(t)(X,Y)= g_{\alpha}(0)(X,Y)+g_{\alpha}(0)(X,\sum^{\infty}_{k=1}\frac{t^k(g^{-1}_{\alpha}(0)h_\alpha)^k(Y)}{k!})
   \end{equation}
   where $X,Y\in TU_\alpha$ and $h_\alpha$ is $J_{\alpha}(0)$-anti-invariant symmetric $(2,0)$-tensor field.
   By constructing a special $h_\alpha$, we will show that if $U_{\alpha_1}\cap U_{\alpha_2}\neq\varnothing$,
   then $g_{\alpha_1}(1)=g_{\alpha_2}(1)=g_J|_{U_{\alpha_1}\cap U_{\alpha_2}}$, that is, the compatibility condition holds which is a smooth Riemannian metric.
   But in general $g_{\alpha_1}(t)(x)\neq g_{\alpha_2}(t)(x)$ for $x\in U_{\alpha_1}\cap U_{\alpha_2}$ (the compatibility condition does not hold),
   and $g_{\alpha_1}(t)$ and $g_{\alpha_2}(t)$  are with Lipschitz equivalent norms on $U_{\alpha_1}\cap U_{\alpha_2}$
   with respect to the smooth Riemannian metric $g_{\alpha}(1)=g_J$, $t\in [0,1)$.
   One can define a family of measurable Riemannian metrics, $g_t$, $t\in [0,1]$,
   on $M^{2n}$ which are with equivalent norms\cite[Section 2]{FW},
    where $g_1=g_J$,
   and construct an open dense submanifold  $\mathring{M}^{2n}$ such that  $M^{2n}\setminus \mathring{M}^{2n}$ has Hausdorff dimension $\leq$ $2n-1$.
   Restricted to $\mathring{M}^{2n}$, $g_t$, $t\in [0,1]$, are smooth.
   By $\omega$ and $g_t$, $t\in [0,1]$, one can define a family of almost complex structures $J_t$, $t\in [0,1]$, on  $\mathring{M}^{2n}$,
   where $J_0$ is integrable, $J_1=J$. More details see the proof of the following lemma:

   \begin{lem}\label{star operator}
   (cf. Fang-Wang \cite[Section 2]{FW})
   Let $*_t$, $t\in [0,1]$, denote the Hodge star operator with respect to the metric $g_t$, $t\in [0,1]$, on $\mathring{M}^{2n}$.
   The $(\omega,J_t,g_t)$, $t\in[0,1]$, is a family of almost K\"{a}hler structures on $\mathring{M}^{2n}$,  in particular $t=1$, $(\omega,J,g_J)=(\omega,J_1,g_1)$, $t=0$, $(\omega,J_0,g_0)$ is a K\"{a}hler flat structures on $\mathring{M}^{2n}$.
    The coefficients of Hodge star operator $*_t$ are in $L^{\infty}_k(M^{2n},g_J)\cap C^\infty(\mathring{M}^{2n})$ with respect to smooth metric $g_J$ and its Levi-Civita connection, where $k$ is a nonnegative integer and $t\in [0,1]$.
    \end{lem}
    \begin{proof}
    Without loss of generality, we may assume that $inj(M^{2n},g_J)\geq 3$.
    Hence, for $\forall p\in M^{2n}$, there exists a Darboux's coordinate chart $$U_p\supset B(p,1)=\{x\in M^{2n}\,|\,\rho_{g_J}(p,x)<1\}$$
    and a diffeomorphism $\psi_p:U_p\rightarrow \mathbb{R}^{2n}\cong\mathbb{C}^n$ such that $\psi_p(U_p)$ is a strictly pseudoconvex domain,
     $\omega|_{U_p}=\psi^*_p\omega_0$,
    $\psi_p(p)=0$ in $\mathbb{C}^n$, where $\rho_{g_J}$ is the distance function defined by the metric $g_J$ and
     $\omega_0$ is the standard K\"{a}hler form on $\mathbb{C}^n$.
    It is easy to see that all derivatives of $\psi_p$ (measured with respect to the metric $g_J$ and its Levi-Civita connection) are bounded, independent of $p$ (cf. Fang-Wang \cite[Section 2]{FW}). Let $J_p(0)=\psi^*_pJ_{st}$, where $J_{st}$ is the standard complex structure on  $\mathbb{C}^n$,
    $g_p(0)(\cdot,\cdot)=\omega(\cdot,J_p(0)\cdot)$  on $U_p$.
    Then $(\omega|_{U_p},J_p(0),g_p(0))$ is a K\"{a}hler flat structure on $U_p$.
    Hence
    \begin{equation}\label{esmate eq}
      \|J_p(0)\|_{C^k(U_p,g_J)}, \,\,\,\|g_p(0)\|_{C^k(U_p,g_J)}, \,\,\,\|g^{-1}_p(0)\|_{C^k(U_p,g_J)}\leq C(M^{2n},\omega,J,k)<+\infty.
    \end{equation}
    Thus, one can construct a family of almost K\"{a}hler structures on $U_p$.
    Define $J_p(t)=J_p(0)e^{th_p}$, $t\in[0,1]$,
    $h_p$ is a $J_p(0)$-anti-invariant symmetric $(2,0)$ tensor field on $U_p$.
    Notice that set $S_p=-J_p(0)\circ J$ on $U_p$, then it is a symmetric, positive definite,
    and symplectic matrix on $U_p$ (cf. McDuff-Salamon \cite[Lemma 2.5.5]{MS}), that is,
    $S_pJ_p(0)S_p=J_p(0)$ on $U_p$ , and $J|_{U_p}=J_p(0)S_p$ on $U_p$.
    Let $h_p=g_p(0)\ln S_p$ on $U_p$.
    Thus the $C^k$-norms of $S_p$ and $h_p$ on $U_p$ with respect to $g_J$ and its Levi-Civita connection are bounded from above by $C(M^{2n},\omega,J,k)(>0)$
    depending only on $M^{2n}$, $\omega$, $J$, and $k$.
    It is easy to see that $J_p(t)$ are $\omega$-compatible for $t\in[0,1]$.
    If $U_{p_1}\cap U_{p_2}\neq\varnothing$, for any $x\in U_{p_1}\cap U_{p_2}$,
    $$
    J_{p_1}(1)(x)=J_{p_2}(1)(x)=J(x).
    $$
    Let $g_p(t)(\cdot,\cdot)=\omega(\cdot,J_p(t)\cdot)$ on $U_p$, $t\in [0,1]$,
    which are $J_p(t)$-invariant Riemannian metric on  $U_p$.
    If $U_{p_1}\cap U_{p_2}\neq\varnothing$, for any $x\in U_{p_1}\cap U_{p_2}$,
    $$
    g_{p_1}(1)(x)=g_{p_2}(1)(x)=g_J(x).
    $$
    Notice that
    \begin{equation}\label{metric t}
      g_p(t)(X,Y)=g_p(0)(X,Y)+g_p(0)(X,\sum^{\infty}_{k=1}\frac{t^k(g^{-1}_p(0)h_p)^k}{k!}Y).
    \end{equation}
    Thus, $(\omega|_{U_p},J_p(t),g_p(t))$, $t\in[0,1]$, is a family of almost K\"{a}hler structures on $U_p$.
     If $U_{p_1}\cap U_{p_2}\neq\varnothing$, then for any $x\in U_{p_1}\cap U_{p_2}$,
     $$
     (\omega,J_{p_1}(1),g_{p_1}(1))|_x=(\omega,J_{p_2}(1),g_{p_2}(1))|_x=(\omega,J,g_J)|_x.
     $$
     Since $[0,1]\times M^{2n}$ is compact, hence, for $t\in[0,1]$,
      \begin{equation}\label{uni bounded}
      \|J_p(t)\|_{C^k(U_p,g_J)}, \,\,\,\|g_p(t)\|_{C^k(U_p,g_J)}, \,\,\,\|g^{-1}_p(t)\|_{C^k(U_p,g_J)}\leq C(M^{2n},\omega,J,k)<+\infty,
     \end{equation}
     where $C(M^{2n},\omega,J,k)$ depends only on $M^{2n},\omega,J$, and $k$.
     Therefore, $\|*_{p,t}\|_{C^k(U_p,g_J)}$, $t\in[0,1]$, are bounded by $C(M^{2n},\omega,J,k)<+\infty$,
    independent of $p$, where $*_{p,t}$ is the Hodge star operator with respect to the metric $g_p(t)$ on $U_p$.
     Hence, $*_{p,t}\in L^{\infty}_k(U_p)$ with respect to smooth metric $g_J$ and its Levi-Civita connection.

      On $U_p\subset M^{2n}$, let $$S_p(t)=-J_p(t)J|_{U_p}=-J|_{U_p}J_p(t),\,\,t\in[0,1],$$ which are symmetric, positive definite, symplectic matrices.
    If $t=1$, $S_p(1)$ is $I_{2n}$.
    If $t=0$, $S_p(0)=S_p$, then $$J|_{U_p}=J_p(t)S_p(t)\,\,\,{\rm and}\,\,\, J_p(t)=J|_{U_p}S_p(t).$$
    Let $$h_p(t)=g_p(t)\ln S_p(t) \,\,({\rm resp.} \,\, h'_p(t)=g_p(1)\ln S_p(t)),\,\, t\in[0,1].$$
    Hence, for $X,Y\in TU_p$, $t\in[0,1]$,
    $$
    g_p(t)(X,Y)=g_p(1)(X,Y)+g_p(1)(X,\sum^{\infty}_{k=1}\frac{(t-1)^k(g^{-1}_p(1)h'_p(t))^{k}(Y)}{k!}),
    $$
    and
    $$
    g_p(1)(X,Y)=g_p(t)(X,Y)+g_p(t)(X,\sum^{\infty}_{k=1}\frac{(1-t)^k(g^{-1}_p(t)h_p(t))^{k}(Y)}{k!}).
    $$
    By (\ref{uni bounded}),  it is easy to see that on $U_p$, for $X\in T_xU_p$,
    \begin{equation}\label{metric equ1}
      g_p(t)(X,X)\leq C(M^{2n},\omega,J)g_p(1)(x)(X,X),
    \end{equation}
    and
     \begin{equation}\label{metric equ0}
      g_p(1)(X,X)\leq C(M^{2n},\omega,J)g_p(t)(x)(X,X),
    \end{equation}
    where $t\in[0,1]$ and $ C(M^{2n},\omega,J)>1$ depending only on $M^{2n},\omega,J$.

   When $(M^{2n},\omega)$ is a closed symplectc manifold of dimension $2n$,
   one can find an $\omega$-compatible almost complex structure $J$ on $M^{2n}$ such that $(\omega,J,g_J)$ is an almost K\"{a}hler structure on $M^{2n}$,
   and
    choose a finite Darboux's coordinate subatlas $\{U_{\alpha_i},\varphi_{\alpha_i}\}_{1\leq i\leq N}$.
    Hence, one can find a partition, $P(\omega,J)$, of $M^{2n}$.
   Set $W_1=U_{\alpha_1}, T_1=U_{\alpha_1}$, and as $2\leq i\leq N $,
    $$
   W_i=U_{\alpha_i}-\cup^{i-1}_{j=1}\overline{U}_{\alpha_j},
   T_i=U_{\alpha_i}-\cup^{i-1}_{j=1}U_{\alpha_j}.
   $$
   Where $W_i$ is open, $W_i\subset T_i\subset U_{\alpha_i}$, $T_i\cap T_j=\emptyset$ for $i\neq j$. Moreover, $T_i$, $1\leq i\leq N$, constitutes a partion of $M^{2n}$, that is, $M^{2n}=\cup^{N}_{i=1}T_i$.
  Thus, let
  $$  \mathring{M}^{2n}=\mathring{M}^{2n}(\omega,J,P)=\cup^{N}_{i=1}W_i .$$
  Then $\mathring{M}^{2n}(\omega,J,P)$ is an open dense submanifold of $M^{2n}$,
   $M^{2n}\setminus \mathring{M}^{2n}(\omega,J,P)$ has Hausdorff dimension $\leq2n-1$ with Lebesgue measure zero.
   Define $$g_t:=\{g_{\alpha_i}(t)|_{T_i}\}_{1\leq i\leq N},\,\,0\leq t\leq 1,$$ which are measurable Riemannian metrics on $M^{2n}$
    (cf. Dru\c{t}u-Kapovich \cite[p.34]{DK}), are also with Lipschitz equivalent norms with respect to the smooth metric $g_J$
      on $M^{2n}$ by (\ref{metric equ1})-(\ref{metric equ0}).
    One can make a family of  $g_t$, $\omega$ compatible almost complex structures $J_t:=\{J_{\alpha_i}(t)|_{T_i}\}_{1\leq i\leq N},\,\,0\leq t\leq 1$.
    Thus, one can define a family of almost K\"{a}hler  structures $(g_t,J_t,\omega)$, $t\in[0,1]$, on $\mathring{M}^{2n}(\omega,J,P)$,
    where $(g_0,J_0,\omega)$ is a K\"{a}hler flat structure on $\mathring{M}^{2n}(\omega,J,P)$ and $(g_1,J_1,\omega)=(g_J,J,\omega)$ on $M^{2n}$.
    By  (\ref{metric equ1}) and (\ref{metric equ0}), on  $\mathring{M}^{2n}(\omega,J,P)$ $g_t$, $t\in[0,1]$, are quasi isometry.
    It is easy to see that $g_t$, $t\in[0,1]$, is a family of measurable Riemannian metrics on $M^{2n}$ which are with Lipschitz equivalent norms with respect to the smooth metric $g_J$
      on $\mathring{M}^{2n}(\omega,J,P)$, that is, Lipschitz condition (\ref{metric equ1})-(\ref{metric equ0})  \cite{Tele3}.
      Thus, we call measurable metrics $\{g_t\}$, $t\in[0,1]$, a family of (measurable) Lipschitz almost K\"{a}hler metrics on $M^{2n}$.
    Let $*_t$, $t\in[0,1]$, be the Hodge star operator defined by the (measurable) Lipschitz Riemannian metric $g_t$, $t\in[0,1]$, on $M^{2n}$.
   Hence, by Lipschitz condition (\ref{metric equ1}) and (\ref{metric equ0}),
    $*_t\in L^{\infty}_k(M^{2n},g_J)\cap C^\infty(\mathring{M}^{2n}(\omega,J,P))$ with respect to the smooth metric $g_J$ and its Levi-Civita connection,
    where $*_t$ is the Hodge star operator with respect to the metric $g_t=\{g_{\alpha_i}(t)|_{T_i}\}_{1\leq i\leq N}$, $t\in[0,1]$, on $M^{2n}$.
    This completes the proof of Lemma \ref{star operator}.
    \end{proof}

   If $t=1$, $g_1=g_J$ is a smooth almost K\"{a}hler metric on $M^{2n}$;
   if $t=0$, it is easy to check that $g_0$ is a K\"{a}hler flat metric on $\mathring{M}^{2n}(\omega,J,P)$
    (cf. N. Teleman \cite[Lipschitz Hodge theory]{Tele3} and Dru\c{t}u-Kapovich \cite[\S 2.1]{DK}).
    On $\mathring{M}^{2n}(\omega,J,P)$, $g_t$, $t\in [0,1]$, is smooth by (\ref{defor metric}).
    Let $\nabla^1_t$ be the second canonical connection on $\mathring{M}^{2n}(\omega,J,P)$ with respect to the metric $g_t$ satisfying
    $$
    \nabla^1_tg_t=0,\,\,\, \nabla^1_tJ_t=0,\,\,\, \nabla^1_t\omega=0.
    $$
   By (\ref{defor metric}) and (\ref{uni bounded})-(\ref{metric equ0}), since $M^{2n}$ is closed, it is easy to see that
   \begin{equation}\label{metric equ}
     C^{-1}g_t(X,X)\leq g_1(X,X)\leq Cg_t(X,X),\,\,\,\forall X\in T\mathring{M}^{2n}(\omega,J,P),\,\,\,t\in[0,1],
   \end{equation}
  where $C>1$ is a constant depending only on $M^{2n},\omega,J$.
  By using $g_t$, $t\in [0,1]$, one can define an inner product $(\cdot,\cdot)_{g_t}$ on $\Lambda^pT^*\mathring{M}^{2n}(\omega,J,P)$,
  $0\leq p\leq 2n$, with respect to the metric $g_t$ as follows:
  $$(\alpha,\beta)_{g_t}dvol_{g_t}:= \alpha\wedge *_t\beta\in \Lambda^{2n}T^*\mathring{M}^{2n}(\omega,J,P),
    \,\,\ \forall \alpha, \beta\in \Omega^p(\mathring{M}^{2n}),$$
     where $*_t$ is the Hodge star operator with respect to the metric $g_t$.

     Notice that the volume density on $\mathring{M}^{2n}(\omega,J,P)$, for $t\in[0,1]$,
      $$ dvol_{g_t}=\frac{\omega^n}{n!}=dvol_{g_J},$$
      which is a smooth volume form on $M^{2n}$.
     In particular, by (\ref{metric equ}) $g_0$ and $g_1$ are equivalent (measurable) Lipschitz Riemannian metrics on $M^{2n}$ (cf. N. Teleman \cite{Tele3}).
     Hence, one can define an inner product on $\Omega^p(M^{2n})$ by using (measurable) Lipschitz Riemannian metric $g_t$ since $M^{2n}\setminus \mathring{M}^{2n}(\omega,J,P)$ has Hausdorff
    dimension $\leq2n-1$ with Lebesgue measure zero,
    $$
    <\alpha,\beta>_{g_t}:=\int_{M^{2n}} (\alpha,\beta)_{g_t}\frac{\omega^n}{n!},\,\,\,t\in [0,1], \,\,\,\forall \alpha,\beta\in\Omega^p(M^{2n}),\,\,\,0\leq p\leq 2n.
    $$
    Thus, by the same reason, we can define $L^2_k(t)$-norm on $\Omega^p(M^{2n})$ with respect to the metric $g_t$ on $M^{2n}$,
    \begin{equation}\label{norm}
     \|\alpha\|^2_{L^2_k(t)}:= \sum^k_{i=0}<(\nabla_t)^i\alpha|_{\mathring{M}^{2n}},(\nabla_t)^i\alpha|_{\mathring{M}^{2n}}>_{g_t},
  \,\,\, t\in [0,1], \,\,\,\forall \alpha\in\Omega^p(M^{2n}),\,\,\,0\leq p\leq 2n,
    \end{equation}
    where $\nabla_t$ is Levi-Civita connection with respect to the metric, $g_t$, on $\mathring{M}^{2n}(\omega,J,P)$.
   Hence we define Hilbert spaces of $\Omega^p(M^{2n})$, $ t\in [0,1],0\leq p\leq 2n$ as follows:
   \begin{defi}\label{measurable aKs}
   (cf. Fang-Wang \cite[Section 2]{FW})
   Let $(M^{2n}, \omega,J,g_J)$ be a closed almost K\"{a}hler manifold of dimension $2n$.
   Suppose that for any partition, $P(\omega,J)$, of $M^{2n}$, $(M^{2n},\omega,J_t,g_t)$,
   $t\in [0,1]$, is a family of almost K\"{a}hler metrics on $\mathring{M}^{2n}(\omega,J,P)$,
   where $J_1=J$, $g_1=g_J$ and $(\omega,J_0,g_0)$
   is a K\"{a}hler flat structure on $\mathring{M}^{2n}(\omega,J,P)$ (see Lipschitz condition (\ref{metric t})-(\ref{metric equ0})),
   and $g_t$, $t\in[0,1]$, are equivalent (measurable) Lipschitz Riemannian metrics on $M^{2n}$.
  Define $L^2_k\Omega^p(t)$ be the completion of $\Omega^p(M^{2n})\otimes_{\mathbb{R}}\mathbb{C}$ with respect to the norm $\|\cdot\|_{L^2_k(t)}$.
   \end{defi}

   By Lemma \ref{star operator} we see that $*_t\in L^{\infty}_k(M^{2n},g_J)\cap C^\infty(\mathring{M}^{2n}(\omega,J,P))$,
   $t\in[0,1]$, where $k$ is a nonnegative integer.
   Hence, by (\ref{metric equ}), $g_t$ on $\Omega^p(M^{2n})$, $t\in[0,1]$ are equivalent, $0\leq p\leq 2n$.
   Since $(\omega,J_1,g_1)=(\omega,J,g_J)$ is an almost K\"{a}hler structure on $M^{2n}$, it is not hard to see that $\Omega^p(M^{2n})\otimes_{\mathbb{R}}\mathbb{C}$
   is dense in $L^2_k\Omega^p(1)$ (cf. Aubin \cite[Theorem 2.6 and Remark 2.7]{Au2} or \cite{Au1}).
   By the definition of $L^2_k\Omega^p(M^{2n})(t)$, $t\in[0,1]$,
   $L^2_k\Omega^p(M^{2n})(t)$ and $L^2_k\Omega^p(M^{2n})(1)$ are quasi isometry (or with equivalent norms), that is, for $t\in[0,1)$,
    \begin{equation}\label{quasi isometry2}
   C^{-1}(k)\|\alpha\|_{L^2_k(t)}\leq \|\alpha\|_{L^2_k(1)}\leq C(k)\|\alpha\|_{L^2_k(t)}, \,\,\,\alpha\in\Omega^p(M^{2n}),\,\,\,0\leq p\leq 2n,
    \end{equation}
 where $k$ is a nonnegative integer and $C(k)>1$ is a constant depending only on $k,M^{2n},\omega,J$.
     It is well know that on $\Omega^p(M^{2n})\otimes_{\mathbb{R}}\mathbb{C}$,
  $d^{*_1}:= - *_1d*_1$ is the $L^2$-adjoint operator of $d$ and $\Delta_1:= dd^{*_1}+d^{*_1}d$ is $L^2$-self-adjoint operator on $\Omega^p(M^{2n})$ (cf. I. Chavel \cite{Cha}), $0\leq p\leq 2n$,
 where $*_1$ is the Hodge star operator with respect to $g_1$.

  In summary, we have the following lemma:
   \begin{lem}\label{quasi isometry}
   (cf. Fang-Wang \cite[Lemma 2.6]{FW})
    Let $(M^{2n},\omega,J,g_J)$ be a closed almost K\"{a}hler manifold of dimension $2n$.
    Suppose that for any partition, $P(\omega,J)$, of $M^{2n}$,
     $(\omega,J_t,g_t)$, $t\in[0,1]$ is a family of almost K\"{a}hler structures on $\mathring{M}^{2n}(\omega,J,P)$ constructed by using
     Darboux's coordinate charts which are quasi isometry (that is, Lipschitz condition (\ref{metric equ1})-(\ref{metric equ0})), where $J_1=J$, $g_1=g_J$ are smooth on $M^{2n}$
      and $(\omega,J_0,g_0)$ is a K\"{a}hler flat structure on $\mathring{M}^{2n}(\omega,J,P)$ (see Definition \ref{measurable aKs}).
      Moreover, $g_t$, $t\in[0,1]$, are with Lipschitz equivalent norms on $M^{2n}$ (see (\ref{metric equ})).
      By (\ref{norm})-(\ref{quasi isometry2}), the $L^2$-norms $\|\cdot\|_{L^2_k(t)}$ on $\Omega^p(M^{2n})$, where $t\in[0,1]$, $k$ is a nonnegative integer, $0\leq p\leq 2n$,
      are equivalent, that is,
      for any $\alpha\in\Omega^p(M^{2n})$, $0\leq p\leq 2n$, $t\in[0,1]$
      $$
       C^{-1}(M^{2n},\omega,J,k)\|\alpha\|_{L^2_k(t)}\leq \|\alpha\|_{L^2_k(1)}\leq C(M^{2n},\omega,J,k)\|\alpha\|_{L^2_k(t)},
      $$
      where $C(M^{2n},\omega,J,k)$ is a constant depending only on $k,M^{2n},\omega,J$.
      In fact, $L^2_k\Omega^p(t)$ and $L^2_k\Omega^p(1)$ are quasi-isometry, where $t\in [0,1)$, $0\leq p\leq 2n$, $k$ is a nonnegative integer.
      In particular, $L^2_k\Omega^p(t)$ can be approximated by $\Omega^p(M^{2n})$.
  \end{lem}

    \vskip 6pt

    I. M. Singer \cite{Sin} exposes a comprehensive program aimed to extend theory of elliptic operators and their index to more general situations.

    N. Teleman \cite{Tele1,Tele2,Tele3} obtained results for a Hodge theory on PL-manifolds and Lipschitz manifolds,
   and J. Cheeger \cite{Che} produces Hodge theory on a very general class of pseudo-manifolds.

   \vskip 6pt

    For a closed PL-manifold $(M^n,g)$, since there exists Stokes' formula on $(M^n,g)$ (cf. N. Teleman \cite[Chapter I, Proposition 3.2]{Tele1}),
    for PL-forms over a cell complex (see D. Sullivan \cite{Sull}), the PL-distributional adjoint operator of $d$ is $d^*=\pm *_gd*_g$
    (cf. N. Teleman \cite[Chapter II]{Tele2}) by using the method of harmonic analysis (cf. E.M. Stein\cite{Ste}).

    For a closed Lipschitz manifold $(M^n,g)$, a theory of signature operators (which are Dirac operators $d+d^*$ when $n=4m$) have been developed by
    N. Teleman \cite{Tele3}.
     $d^*$ is formally defined as $d^*=\pm*_gd*_g$ and it has sufficient good properties to mimic the usual linear elliptic analysis.
     Here $*_g$ is the Hodge star operator with respect to the metric $g$.
       As done in smooth case, we have the Rellich type lemma (cf. N. Teleman \cite[Theorem 7.1]{Tele3}) by using basic properties of classical
      elliptic pseudo-differential operators M.A. Shubin\cite{Shu}.

      \vskip 6pt

     In our situation, the underlying manifold $(M^{2n},\omega,J,g_J)$ is closed, smooth.
      Recall that Hodge star operator $*_t$ with respect to the metric $g_t$ is in $L^{\infty}_k(M^{2n})\cap C^\infty(\mathring{M}^{2n}(\omega,J,P))$ and $M^{2n}\setminus \mathring{M}^{2n}(\omega,J,P)$
       has Hausdorff dimension $\leq2n-1$ with Lebesgue measure zero,
       thus $$d: \Omega^{p-1}(M^{2n})\otimes_{\mathbb{R}}\mathbb{C} \longrightarrow L^2\Omega^p(t),$$
       $$
       -*_td*_t: \Omega^{p}(M^{2n})\otimes_{\mathbb{R}}\mathbb{C} \longrightarrow L^2\Omega^{p-1}(t)
       $$
     are well defined, where $*_t$ is the Hodge star operator with respect to the metric $g_t$, $t\in[0,1]$,
     since $\Omega^{p-1}(M^{2n})\otimes_{\mathbb{R}}\mathbb{C}$ and $\Omega^{p}(M^{2n})\otimes_{\mathbb{R}}\mathbb{C}$
     are dense in $L^2\Omega^{p-1}(t)$ and $L^2\Omega^{p}(t)$, $0\leq p\leq 2n$, respectively
     (Notice that $(M^{2n},\omega,J_1,g_1)=(M^{2n},\omega,J,g_J)$ is a closed almost K\"{a}hler manifold, $\Omega^p(M^{2n})$ is dense in $L^2_k\Omega^p(M^{2n})(1)$,
      by Lemma \ref{quasi isometry},
      it is easy to see that $L^2_k\Omega^p(M^{2n})(t)$, $t\in[0,1)$ can be approximated by $\Omega^p(M^{2n})$ since $g_t$,
       $t\in[0,1]$, are with equivalent norms, that is, with Lipschitz condition (\ref{metric equ}).).
      As done in smooth Riemannian metric, by Lipschitz condition (\ref{metric equ}), we can consider formal adjoint operator $d^{*_t}$ of $d$
      with respect to the (measurable) Lipschitz metric $g_t$.
          Let $d^{*_t}$ be $L^2$-adjoint operator of $d$ with respect to the metric $g_t$, $t\in [0,1]$.
          Since $\Omega^p(M^{2n})\otimes_{\mathbb{R}}\mathbb{C}$ are dense in $L^2\Omega^p(t)$, $0\leq p\leq 2n$, $t\in[0,1]$,
          $g_1=g_J$ and $d^{*_1}=-*_1d*_1$,
        it is easy to obtain the following lemma:
          \begin{lem}\label{adjoint lemma}
          (cf. Fang-Wang \cite[Lemma 2.7]{FW} or Tan-Wang-Zhou \cite[Lemma 2.7]{TWZ})
          The formal adjoint operators
          $d^{*_t}=-*_td*_t, \,\,\, (d^{*_t})^{*_t}=d$ in the sense of distributions.
          In particular, Stokes' formula holds, that is, for $\forall\alpha\in L^2_1\Omega^{2n-1}(t)$, $$\int_{M^{2n}}d\alpha=0,\,\,t\in[0,1].$$
          Thus, Hodge Laplacian $\Delta_t=dd^{*_t}+d^{*_t}d$, $t\in[0,1]$, on $M^{2n}$ is essentially self adjoint in the sense of distributions.
          \end{lem}
         \begin{proof}
          It is enough to choose $\alpha\in \Omega^{p-1}(M^{2n})\otimes_{\mathbb{R}}\mathbb{C}$ and
           $\beta\in \Omega^{p}(M^{2n})\otimes_{\mathbb{R}}\mathbb{C}$ due to Lemma \ref{quasi isometry}.
           Then $\alpha\wedge *_t\beta\in L^\infty(\Lambda^{2n-1}T^*M^{2n})$.
          Since $*_t$ is in $L^{\infty}_1(M^{2n},g_J)$,  $*_t\beta$ is in $L^2_1\Omega^{2n-p}(M^{2n})(t)$,
         hence $\alpha\wedge *_t\beta\in  L^2_1\Omega^{2n-1}(M^{2n})(t)$.
         Since $L^2_1(t)$ and  $L^2_1(1)$ are with equivalent norms,
         $d(\alpha\wedge *_t\beta)\in L^2\Omega^{2n}(M^{2n})(1)$.
        By Lemma \ref{quasi isometry},
         there is a sequence $\beta_k\in\Omega^{2n-p}(M^{2n})$ such that $\|*_t\beta-\beta_k\|_{L^2_1(1)}\rightarrow 0$
          as $k\rightarrow \infty$.
          Therefore
          $$
          \int_{M^{2n}}d(\alpha\wedge *_t\beta)=\lim_{k\rightarrow\infty}\int_{M^{2n}}d(\alpha\wedge\beta_k )=0.
          $$
          Thus,
         \begin{eqnarray}\label{int equ}
          \nonumber
            0 &=& \int_{M^{2n}}d(\alpha\wedge *_t\beta)  \nonumber\\
             &=& \int_{M^{2n}}d\alpha\wedge *_t\beta+(-1)^{p-1}\int_{M^{2n}}\alpha\wedge d(*_t\beta) \nonumber\\
             &=& \int_{M^{2n}} (d\alpha, \beta)_{g_t}dvol_{g_t}-\int_{M^{2n}}(\alpha,-*_td*_t\beta)_{g_t}dvol_{g_t} \nonumber\\
             &=&  <d\alpha,\beta>_{g_t}-<\alpha,-*_td*_t\beta>_{g_t}.
          \end{eqnarray}
          Therefore, formal adjoint $d^{*_t}$, $t\in [0,1]$ is $-*_td*_t$,
          and $(d^{*_t})^{*_t}=d$ in Hilbert spaces.
          So it is easy to see that Hodge-Laplacian
          $$
          \Delta_t=dd^{*_t}+d^{*_t}d,\,\,\,t\in[0,1],
          $$
          with respect to the metric $g_t$ is essentially self-adjoint (cf. P. R. Chernoff \cite{Cher}).
          This completes the proof of Lemma \ref{adjoint lemma}.
          \end{proof}

         By the above lemma, the measurable metrics $g_t$, $t\in [0,1]$, are with equivalent norms, that is, with Lipschitz condition (\ref{metric equ}),
         as in the smooth Riemannian metric, we can define $\Delta_t$-harmonic and $d+d^{*_t}$-harmonic forms on $M^{2n}$.
         Thus, we have the following lemma:
          \begin{lem}\label{har lemma}
         (cf. Fang-Wang \cite[Lemma 2.8]{FW} and Tan-Wang-Zhou \cite[Lemma 2.8]{TWZ})
          If a $L^2$ $p$-form $\alpha$ on $M^{2n}$ is $\Delta_t$-harmonic, then $\alpha$ is $d+d^{*_t}$-harmonic,
          $0\leq p\leq 2n$, $t\in [0,1]$.
          Where $\Delta_t=dd^{*_t}+d^{*_t}d$ is the Hodge Laplacian,
          $\alpha\in L^2\Omega^p(M^{2n})(t)$ is called $\Delta_t$-harmonic iff $\Delta_t\alpha=0$ in the sense of distributions;
          $d+d^{*_t}$ is Dirac operator, $\alpha\in L^2\Omega^p(M^{2n})(t)$ is called $d+d^{*_t}$-harmonic iff
           $d\alpha=0$ and $d^{*_t}\alpha=0$ in the sense of distributions.
          \end{lem}

          Let $L^2\Omega^p(M^{2n})(t)$ be $L^2$-space of $p$-forms on $M^{2n}$ with respect to the metric $g_t$, $0\leq p\leq 2n$, $t\in [0,1]$.
          We defined $L^2_k\Omega^p(M^{2n})(t)$ for a nonnegative integer $k$ be $kth$ Sobolev space of $p$-forms on $M^{2n}$, that is,
              the Hilbert square completion of $\Omega^p(M^{2n})$ with respect to the inner product
              or norm induced by (measurable) Lipschitz Riemannian metric $g_t$.
               Since Hodge Laplacian $\Delta_t=dd^{*_t}+d^{*_t}d$, $t\in[0,1]$, is essentially self adjoint in the sense of distributions
               (cf. J. Br\"{u}ning and M. Lesch \cite{BL}, P. R. Chernoff \cite{Cher} ), the $kth$ Sobolev space
                 of $p$-forms on $M^{2n}$ is equivalent to the following definition since $L^2_k\Omega^p(M^{2n})$ have the Rellich type lemma (cf. N. Teleman \cite{Tele1,Tele2,Tele3})
              \begin{eqnarray}\label{L norm}
                 <\alpha,\beta>_{L^2_k(t)}&=&<(1+\Delta_t)^{\frac{k}{2}}\alpha,(1+\Delta_t)^{\frac{k}{2}}\beta>_{L^2(t)}\nonumber\\
                                &=&<\alpha,(1+\Delta_t)^k\beta>_{L^2(t)},  \nonumber\\
              \|\alpha\|_{L^2_k(t)} &=& \sqrt{<\alpha,(1+\Delta_t)^k\alpha>_{L^2(t)}} .
              \end{eqnarray}

    By Lemma \ref{har lemma}, we have Hodge-Kodaira decomposition (cf. Fang-Wang \cite[Section 2]{FW}, Gromov \cite{Gro}, Cheeger \cite{Che} and Teleman \cite{Tele3})
         \begin{equation}
           L^2\Omega^p(M^{2n})(t)=\mathcal{H}^p(M^{2n})(t)\oplus \overline{dL^2\Omega^{p-1}(M^{2n})(t)}\oplus \overline{d^{*_t}L^2\Omega^{p+1}(M^{2n})(t)},
         \end{equation}
     where $$\mathcal{H}^p(M^{2n})(t)=\ker\Delta_t|_{L^2\Omega^p(M^{2n})(t)},$$
     $\overline{dL^2\Omega^{p-1}(M^{2n})(t)}$ and $\overline{d^{*_t}L^2\Omega^{p+1}(M^{2n})(t)}$ are closure of
     $dL^2\Omega^{p-1}(M^{2n})$ and $d^{*_t}L^2\Omega^{p+1}(M^{2n})$ with respect to $L^2(t)$-norm in the sense of distributions, respectively.
     For any $\alpha\in \mathcal{H}^p(M^{2n})(t)$, by the definition of the norm $\|\cdot\|_{L^2_k(t)}$ (cf. (\ref{L norm})),
     it is easy to get  $\|\alpha\|_{L^2_k(t)}=\|\alpha\|_{L^2(t)}$.
     By Lemma \ref{quasi isometry},
     $L^2_k\Omega^p(M^{2n})(t)$, $t\in[0,1)$ and $L^2_k\Omega^p(M^{2n})(1)$ are quasi isometry,
     hence $\alpha\in L^2_k\Omega^p(M^{2n})(1)$, for any $k\in\mathbb{N}$.
     By Sobolev embedding theorem (cf. T. Aubin \cite{Au2}),
      $$\alpha\in C^l(\Lambda^pT^*M^{2n}),\,\,\,
     {\rm for} \,\,\,{\rm any}\,\,\, l\in\mathbb{N}$$ such that $k-l>n$, hence $\alpha\in \Omega^p(M^{2n})$.
     Therefore, $$\mathcal{H}^p(M^{2n})(t)\subset\Omega^p(M^{2n}),\,\,t\in[0,1],\,\,0\leq p\leq 2n.$$
      Define
      \begin{equation}
        b^p(M^{2n})(t)=\dim \mathcal{H}^p(M^{2n})(t), \,\,\,0\leq p\leq 2n, \,\,\,t\in [0,1].
      \end{equation}

         One can define a Hilbert cochain complex
         $L^2_{l-*}\Omega^*(M^{2n})(t)$ as follows (cf. \cite{BL,Luck,Tele3})
         \begin{equation}
           0\rightarrow L^2_l\Omega^0(M^{2n})(t)\stackrel{d}{\rightarrow}
            L^2_{l-1}\Omega^1(M^{2n})(t)\stackrel{d}{\rightarrow}\cdot\cdot\cdot
          \stackrel{d}{\rightarrow}L^2_{l-2n}\Omega^{2n}(M^{2n})(t)\rightarrow 0,\,\,\, l\geq 2n.
          \end{equation}
      Let $$H^p(M^{2n})(t)=\ker d|_{L^2\Omega^p(M^{2n})(t)}/\overline{dL^2\Omega^{p-1}(M^{2n})(t)}.$$
      In summary, we have the following theorem:
      \begin{theo}
      (cf. N. Teleman \cite[Theorem 4.1]{Tele3} or Fang-Wang \cite[Theorem 2.9]{FW})
         Let $(M^{2n},\omega,J,g_J)$ be a closed almost K\"{a}hler manifold.
         By using Darboux's coordinate charts, construct a family of (measurable) Lipschitz almost K\"{a}hler structures $(\omega,J_t,g_t)$, $t\in[0,1]$, on $M^{2n}$ which are with equivalent norms,
         where $g_1=g_J$, $J_1=J$ are smooth on $M^{2n}$,
         and $(\omega,J_0,g_0)$ which is a (measurable) Lipschitz K\"{a}hler flat structure on $M^{2n}$ is smooth restricted on $\mathring{M}^{2n}(\omega,J,P)$.
         Then for any degree $p$, $0\leq p\leq 2n$,

         (1) there are the strong Kodaira-Hodge  decomposition
         $$
          L^2\Omega^p(M^{2n})(t)=\mathcal{H}^p(M^{2n})(t)\oplus \overline{dL^2\Omega^{p-1}(M^{2n})(t)}\oplus \overline{d^{*_{t}}L^2\Omega^{p+1}(M^{2n})(t)},
          \,\,\,t\in[0,1]
         $$
         and  $\mathcal{H}^p(M^{2n})(t)$ is smooth on $M^{2n}$;

         (2) the Hodge homomorphisms
         $$
         \chi^p(t):\mathcal{H}^p(M^{2n})(t)\longrightarrow H^p(M^{2n},\mathbb{C})(t),\,\,\,t\in[0,1]
         $$
         $$
         \alpha\longmapsto[\alpha]
         $$
          is an isomorphism.
      \end{theo}

      \begin{rem}
      (1) By Lemma \ref{quasi isometry}, $L^2_k\Omega^p(t)$, $t\in[0,1)$, and $L^2_k\Omega^p(1)$ are quasi-isometry (see (\ref{quasi isometry2})),
      hence, $$b^p(M^{2n})(t)=b^p(M^{2n})(1),\,\,t\in[0,1),$$ where $0\leq p\leq2n$, $k$ is a nonnegative integer.

      (2) Since $\Delta_t$, $t\in[0,1)$, are elliptic operators with bounded smooth (resp. $L^\infty_k$, $k\in\mathbb{N}$)
      coefficients on $\mathring{M}^{2n}(\omega,J,P)$ (resp. $M^{2n}$).
      But $\mathcal{H}^p(M^{2n})(t)$, $t\in[0,1]$, $0\leq p\leq 2n$, are smooth harmonic forms on $M^{2n}$ with respect to the metric $g_t$.

      (3)
      In general, by using different methods, E. de Giorgi \cite{deG} and J. Nash \cite{Nash} proved the following result:

      The weak solution of elliptic (parabolic) operators with measurable coefficients are continuous.
      \end{rem}

      \vskip 6pt

      In the remainder of this appendix, restricted to the open dense submanifold $\mathring{M}^{2n}(\omega,J,P)$ of $(M^{2n},\omega)$,
       we will reduce the $1$-form type Calabi-Yau equation
      \begin{equation}\label{appendix CYequ}
        (\omega+da)^n=e^F\omega^n,\,\, a\in \Omega^1(M^{2n}),\,\, \int_{M^{2n}}e^F\omega^n=\omega^n,
      \end{equation}
       to a complex Monge-Amp\`{e}re equation
       with respect to a  K\"{a}hler flat structure $(\omega,J_0,g_0)$ defined by Darboux's coordinate charts on $(M^{2n},\omega)$.
       In fact, $g_0$ can be regarded as a singular K\"{a}hler metric on $M^{2n}$.

       By a symplectic form $\omega$ on $M^{2n}$, one can find an almost K\"{a}hler structure $(\omega,J,g_J)$ on $M^{2n}$,
       where $J$ is an $\omega$-compatible almost complex structure on $M^{2n}$, $g_J(\cdot,\cdot)=\omega(\cdot,J\cdot)$.
       It is easy to see that $\omega(a):=\omega+da$ is a new symplectic form on $M^{2n}$ by (\ref{appendix CYequ}).
       Similarly, one can find a new almost K\"{a}hler structure $(\omega(a),J(a),g_{J(a)})$ on $M^{2n}$.
       As done in the before, one can consider K\"{a}hler flat structure on $\mathring{M}^{2n}(\omega,J,P)$ by almost K\"{a}hler structures $(\omega,J,g_J)$
       and $(\omega(a),J(a),g_{J(a)})$.

       For the almost K\"{a}hler manifold $(M^{2n},\omega,J,g_J)$, choose a Darboux's charts $\{U_{\alpha},\psi_{\alpha}\}_{\alpha\in\Lambda}$
       such that $$\psi_{\alpha}:U_{\alpha}\rightarrow \psi_{\alpha}(U_{\alpha})\subseteq\mathbb{R}^{2n}\cong\mathbb{C}^n$$ is a diffeomorphism
       and $\psi^*_{\alpha}\omega_0=\omega|_{U_{\alpha}}$,
       where $$\omega_0=\frac{\sqrt{-1}}{2}\sum^n_{i=1}dz_i\wedge d\bar{z}_i$$ is the standard K\"{a}hler
       form on $\mathbb{C}^n$ with respect to the standard complex structure $J_0$ on  $\mathbb{C}^n$ (more details, see McDuff-Salamon \cite{MS}).
       Without loss of generality, one may assume that $0\in\psi_{\alpha}(U_{\alpha})$ which is a contractible and relatively compact,
       strictly $(J_0-)$ pseudoconvex domain (cf. L. H\"{o}mander \cite{Hor}).
       On the Darboux's coordinate charts, one can make the deformation of $\omega$-compatible almost complex structures.
       On each $U_{\alpha}$, $\alpha\in\Lambda$, let $g_{\alpha}(0)(\cdot,\cdot):=\omega(\cdot,\psi^*_{\alpha}J_{st}\cdot)$ on $U_{\alpha}$.
       Define
       \begin{equation}\label{}
         g_{\alpha}(t)(\cdot,\cdot):= g_{\alpha}(0)(\cdot,e^{th_{\alpha}}\cdot),
       \end{equation}
         where $h_{\alpha}$ is $J_{\alpha}(0)=\psi^*_{\alpha}J_{0}$-anti-invariant symmetric $(2,0)$-tensor.
         It is easy to see that $g_{\alpha}(1)=g_J|_{U_\alpha}$ and $J_{\alpha}(0)=\psi^*_{\alpha}J_0$.
         Therefore, $\{(\omega|_{U_\alpha},J_{\alpha}(0),g_{\alpha}(0))\}_{\alpha\in\Lambda}$ defines a (measurable) Lipschitz
         K\"{a}hler flat structure on $M^{2n}$.
         By the local $\partial\bar{\partial}$-lemma (cf. D.D. Joyce \cite[p84]{Joc} or L. H\"{o}mander \cite[Chapter II]{Hor}),
         in terms of complex coordinates
         one can find a family of strictly plurisubharmonic functions $\{\varphi'_\alpha\}_{\alpha\in\Lambda}$ such that
         \begin{equation}\label{}
          \omega|_{U_\alpha}=\sqrt{-1}\partial_{J_{\alpha}(0)}\bar{\partial}_{J_{\alpha}(0)}\varphi'_\alpha,\,\,\,
          \omega|_{\psi_{\alpha}(U_\alpha)}=\sqrt{-1}\partial\bar{\partial}\varphi'_\alpha.
         \end{equation}

         Similarly, for the almost K\"{a}hler structure $(\omega(a),J(a),g_{J(a)})$ on $M^{2n}$,
         choose a Darboux's atlas $\{U_{\beta}(a),\psi_{\beta}(a)\}_{\beta\in\Lambda(a)}$ such that
         $$
        \psi_{\beta}(a):U_{\beta}(a)\rightarrow \psi_{\beta}(a)(U_{\beta}(a))\subseteq\mathbb{C}^n\cong\mathbb{R}^{2n}
        $$
        is a diffeomorphism and $ \psi^*_{\beta}(a)\omega_0=\omega(a)|_{U_{\beta}(a)}$, $\beta\in\Lambda(a)$.
       Where $0\in\psi_{\beta}(a)(U_{\beta}(a))$ which is a contractible, relatively compact and strictly pseudoconvex domain.
        On each $U_{\beta}(a)\subset M^{2n}$, $\beta\in\Lambda(a)$,
         let $$g_{\beta}(a)(0)(\cdot,\cdot)=\omega(a)|_{U_\beta(a)}(\cdot,J_{\beta}(a)(0)\cdot),\,\,
         J_{\beta}(a)(0)=\psi^*_{\beta}(a)J_0.$$
         Then $(\omega(a)|_{U_\beta(a)},J_{\beta}(a)(0),g_{\beta}(a)(0))$ is a
       K\"{a}hler flat structure on $\mathring{M}^{2n}(\omega(a),J(a),P(a))$.
         Then
         \begin{equation}\label{}
           g_{\beta}(a)(1)(\cdot,\cdot)=g_{J(a)}(\cdot,\cdot)|_{U_\beta(a)}=g_{\beta}(a)(0)(\cdot,\cdot)+g_{\beta}(a)(0)(\cdot,\sum^\infty_{k=1}(g^{-1}_{\beta}(a)(0)h_\beta(a))^k/k!\cdot),
         \end{equation}
        where $h_\beta(a)$ is a $J_{\beta}(a)(0)$-anti-invariant symmetric $(2,0)$-tensor.
        Hence, one can find a family of strictly plurisubharmonic functions $\{\varphi''_\beta(a)\}_{\beta\in\Lambda(a)}$ such that
        \begin{equation}\label{}
          \omega(a)|_{U_\beta(a)}=\sqrt{-1}\partial_{J_{\beta}(a)(0)}\bar{\partial}_{J_{\beta}(a)(0)}\varphi''_\beta(a),\,\,\,
          \omega(a)|_{\psi_{\beta}(a)(U_\beta(a))}=\sqrt{-1}\partial\bar{\partial}\varphi''_\beta(a).
         \end{equation}

        Since every symplectic form on $M^{2n}$ is locally diffeomorphic to the standard form $\omega_0$ on $\mathbb{C}^n$ by Darboux theorem (cf. McDuff-Salamon \cite[Theorem 3.2.2]{MS}) and $M^{2n}$ is closed,
        by the discussion above, for $(M^{2n},\omega)$ and $(M^{2n},\omega(a))$ one can choose a  Darboux's charts
        $\mathcal{V}(a)=\{V_{\alpha}(a),\bar{\psi}_\alpha,\bar{\psi}'_\alpha\}_{\alpha\in \Lambda}$ such that
        $$
        \bar{\psi}_\alpha: V_{\alpha}(a)\longrightarrow\bar{ \psi}_\alpha(V_{\alpha}(a))\subset \mathbb{R}^{2n}\cong\mathbb{C}^n
        $$
        and
        $$
        \bar{\psi}'_\alpha: V_{\alpha}(a)\longrightarrow\bar{\psi}'_\alpha(V_{\alpha}(a))\subset \mathbb{R}^{2n}\cong\mathbb{C}^n
        $$
        are diffeomorphisms. Moreover, $\bar{\psi}^*_\alpha\omega_0$ and $\bar{\psi}'^*_\alpha\omega_0$ are diffeomorphic to $\omega|_{V_{\alpha}(a)}$ and $\omega(a)|_{V_{\alpha}(a)}$, respectively,
        where
        $$
        \omega_0=\frac{\sqrt{-1}}{2}\sum^n_{i=1}dz_i\wedge d\bar{z}_i
        $$
        is the standard K\"{a}hler form on $\mathbb{C}^n$ with respect to the standard complex structure $J_0$ on $\mathbb{C}^n$.
        One may assume that $0\in \bar{\psi}_\alpha(V_{\alpha}(a))$ (resp. $\bar{\psi}'_\alpha(V_{\alpha}(a)$), which is a contractible and relatively compact and strictly ($J_0$-) pseudoconvex domain.
        On the  Darboux's coordinate charts, one can make the deformation.
        Therefore, $\{\omega|_{V_{\alpha}(a)},J_{\alpha}(0),g_{\alpha}(0)\}_{\alpha\in \Lambda}$ and $\{\omega(a)|_{V_{\alpha}(a)},J_{\alpha}(a)(0),g_{\alpha}(a)(0)\}_{\alpha\in \Lambda}$
        define two (measurable) Lipschitz  K\"{a}hler flat structures on $M^{2n}$.
        By the local $\partial\bar{\partial}$-lemma, one can find two families of strictly phurisubharmonic functions,
        $\{\varphi'_\alpha\}_{\alpha\in\Lambda}$ and $\{\varphi''_\alpha(a)\}_{\alpha\in\Lambda}$ such that
        $$
         \omega|_{\bar{\psi}_\alpha(V_\alpha)}=\sqrt{-1}\partial\bar{\partial}\varphi'_\alpha,\,\,\,
          \omega(a)|_{\bar{\psi}'_{\alpha}(V_\alpha)}=\sqrt{-1}\partial\bar{\partial}\varphi''_\alpha(a).
        $$

        Since $M^{2n}$ is closed, one can define a partition, $P(M^{2n},\omega(a),J(a),J)$, of $M^{2n}$ as follows.
        For $(M^{2n},\omega)$ and $(M^{2n},\omega(a))$, one can choose finite measurable K\"{a}hler flat
        subatlas (that is, $\cup^{N(a)}_{k=1}V_k(a)=M^{2n}$ ),
        $$\{V_k(a),\omega|_{\bar{\psi}_k(V_k(a))},J_k(0),\varphi'_k\}_{1\leq k\leq N(a)},\,\,\,
         \omega|_{\bar{\psi}_k(V_k(a))}=\sqrt{-1}\partial\bar{\partial}\varphi'_k,$$
        and  $$\{V_k(a),\omega(a)|_{\bar{\psi}'_k(V_k(a))},J_k(a)(0),\varphi''_k(a)\}_{1\leq k\leq N(a)},\,\,\,
         \omega(a)|_{\bar{\psi}'_k(V_k(a))}=\sqrt{-1}\partial\bar{\partial}\varphi''_k(a),$$
         where $\varphi'_k$, $\varphi''_k(a)$, $1\leq k\leq N(a)$ are strictly plurisubharmonic functions.
         Thus, the Calabi-Yau equation (\ref{appendix CYequ}), in $\bar{\psi}_k(V_{k}(a))$, can be written as the complex Monge-Amp\`{e}re eauation
         \begin{equation}\label{MA equ2}
           \det(h_{i\bar{j}}+\frac{\partial^2\varphi_k(a)}{\partial z_i\partial \bar{z}_j})=e^F\det(h_{i\bar{j}}),
         \end{equation}
         where $$\frac{\partial^2\varphi_k(a)}{\partial z_i\partial \bar{z}_j}=-\sqrt{-1}da|_{V_{k}(a)}
         =\sqrt{-1}(\partial\bar{\partial}\varphi'_k-\partial\bar{\partial}\varphi''_k(a)),$$
         $(h_{i\bar{j}})$ is given by $-\sqrt{-1}\omega$ on $\bar{\psi}_k(V_{k}(a))$ in terms of complex coordinates.

         In summary, we have the following proposition:
         \begin{prop}\label{CYequ Lipschitz}
         Let $(M^{2n},\omega,J,g_J)$ be a closed almost K\"{a}hler manifold of dimension $2n$.
         Suppose $F\in C^\infty(M^{2n},\mathbb{R})$ satisfies
         $$
         \int_{M^{2n}}\omega^n= \int_{M^{2n}}e^F\omega^n.
         $$
         If there exists a real $1$-form $a\in\Omega^1(M^{2n})$ which is a solution of the following Calabi-Yau equation
         $$
         (\omega+da)^n=e^F\omega^n.
         $$
         Then, one can define a partition, $P(M^{2n},\omega(a),J(a),J)$, of $M^{2n}$, and  find a finite Darboux's coordinate subatlas for symplectic  forms $\omega$ and $\omega(a)=\omega+da$
         as follows:
         $\mathcal{V}(a)=\{V_{k}(a)\}_{1\leq k\leq N(a)}$.
         Since $ V_{k}(a)$ is a relatively compact and pseudoconvex domain,
         then one can define $\omega(a)$-compatible almost K\"{a}hler structures on $V_{k}(a)$.
         Hence,
         on each $V_{k}(a)$, $1\leq k\leq N(a)$,
         find two families of $\omega$-compatible (resp. $\omega(a)$-compatible) almost K\"{a}hler structures $(\omega|_{V_{k}(a)},J_k(t),g_k(t))$
         (resp. $(\omega(a)|_{V_{k}(a)},J_k(a)(t),g_k(a)(t)))$, $t\in[0,1]$.
         When $t=1$,
         $$(\omega|_{V_{k_1}(a)},J_{k_1}(1),g_{k_1}(1))=(\omega|_{V_{k_2}(a)},J_{k_2}(1),g_{k_2}(1)),$$
         if $V_{k_1}(a)\cap V_{k_2}(a)\neq\varnothing$.
           When $t=0$,
           $(\omega|_{V_{k}(a)},J_{k}(0),g_{k}(0))$ is a K\"{a}hler flat structure on $V_{k}(a)$,
           but $(J_{k_1}(0),g_{k_1}(0))|_p\neq (J_{k_2}(0),g_{k_2}(0))|_p$ for $p\in V_{k_1}(a)\cap V_{k_2}(a)\neq \varnothing $,
      which satisfy Lipschitz condition (\ref{metric equ1})-(\ref{metric equ0}) (resp. for $\omega(a)$).

          On $V_{k}(a)$, $1\leq k\leq N(a)$, for $$(\omega|_{V_{k}(a)},J_k(0),g_k(0))\,\,\,{\rm and} \,\,\,(\omega(a)|_{V_{k}(a)},J_k(a)(0),g_k(a)(0)),$$
       choose a holomorphic chart $\{z_1,\cdot\cdot\cdot,z_n\}$
         such that the Calabi-Yau equation
         $$\omega(a)^n=e^F\omega^n,\,\,\int_{M^{2n}}\omega(a)^n=\int_{M^{2n}}e^F\omega^n=\int_{M^{2n}}\omega^n$$
         is reduced to the complex Monge-Amp\`{e}re equation on $V_{k}(a)$
         $$\det(h_{i\bar{j},k}+\frac{\partial^2\varphi_k(a)}{\partial z_i\partial \bar{z}_j})=e^F\det(h_{i\bar{j},k}).$$
         Where $\varphi_k(a)$ is the K\"{a}hler potential on $V_{k}(a)$, on the holomorphic chart $\{z_1,\cdot\cdot\cdot,z_n\}$,
         $$h_{i\bar{j},k}=-\sqrt{-1}\omega|_{V_{k}(a)},
         -\sqrt{-1}da|_{V_{k}(a)}=(\frac{\partial^2\varphi_k(a)}{\partial z_i\partial \bar{z}_j}),$$
         $$h(a)_{i\bar{j},k}=-\sqrt{-1}\omega(a)|_{V_{k}(a)}=h_{i\bar{j},k}+\frac{\partial^2\varphi_k(a)}{\partial z_i\partial \bar{z}_j}.$$
         ($h_{i\bar{j},k}$), ($h(a)_{i\bar{j},k})$ are Hermitian matrices.
         \end{prop}

         For any partition, $P(a)=P(M^{2n},\omega(a),J(a),J)$, of $M^{2n}$,  we constructed a K\"{a}hler potential $\varphi(a)$ on the  K\"{a}hler flat structure on $\{V_{k}(a)\}_{1\leq k\leq N(a)}$.
         Define $$W_1(a)=V_1(a),\,\,T_1(a)=V_1(a),$$ for $ 2\leq k\leq N(a)$
         $$W_k(a)=V_k(a)\setminus\cup^{k-1}_{j=1}\overline{V}_j(a),  T_k(a)=V_k(a)\setminus\cup^{k-1}_{j=1}V_j(a).$$
         Where, $W_k(a)$, $ 1\leq k\leq N(a)$, are open sets of $M^{2n}$,
         $$W_k(a)\subset T_k(a)\subset V_k(a),\,\,\,
          T_{k_1}(a)\cap T_{k_2}(a)=\varnothing$$
          for $k_1\neq k_2$ and $\{T_k(a),1\leq k\leq N(a)\}$ is a disjoint partition of $M^{2n}$.
        Then
           $$ \mathring{M}^{2n}(a,P)=\mathring{M}^{2n}(\omega(a),J(a),J,P(a))=\cup^{N(a)}_{k=1}W_k(a), \,\,\,  M^{2n}=\cup^{N(a)}_{k=1}T_k(a).$$
       $\mathring{M}^{2n}(a,P)$ is an open dense submanifold of $M^{2n}$ and $M^{2n}\setminus\mathring{M}^{2n}(a,P)$ has
       Hausdorff dimension $\leq 2n-1$ wiht Lebesgue measure zero.
       $\{ T_k(a)\}_{1 \leq k\leq N(a)}$ is a disjoint partition of $M^{2n}$.
       Let $$\varphi(a)=\sum^{N(a)}_{k=1}\varphi_k(a)|_{T_k(a)}\,\,{\rm on}\,\,M^{2n},$$ it is still denoted by $\varphi(a)$.
        Let $g_t:=\{g_{k}(t)|_{T_k(a)}\}_{1 \leq k\leq N(a)}$, $t\in[0,1]$, be a family of (measurable) Lipschitz Riemannian metrics
         with Lipschitz equivalent norms (\ref{metric equ})
         on $M^{2n}$.
        Then $g_0$ is a K\"{a}hler flat metric on $\mathring{M}^{2n}(a,P)$, $g_0$ is also a singular  K\"{a}hler metric on $M^{2n}$,
         and  $g_1=g_J$ is an almost K\"{a}hler metric on ${M}^{2n}$.

       One can define the norm $\|\cdot\|_{C^k_i(\mathring{M}^{2n}(a,P))}$ for $\varphi(a)$ on $\mathring{M}^{2n}(a,P)$ with respect to the metric $g_i$, $i=1,0$
       associated to $\omega$,
       where $k$ is a nonnegative integer.
       $$
       \|\varphi(a)\|_{C^k_i}=\sum^k_{j=1}\sup_{x\in\mathring{M}^{2n}(a,P)}|\nabla^j_i\varphi(a)(x)|,
       $$
      where $\nabla_i$ is the Levi-Civita connection with respect to the metric $g_i$, $i=1,0$.
      Notice that $M^{2n}\setminus\mathring{M}^{2n}(a,P)$ has Lebesgue measure zero,
      $g_1$ and $g_0$ are quasi isometry on $\mathring{M}^{2n}(a,P)$, and equivalent (measurable) Lipschitz Riemannian metrics on $M^{2n}$
      which satisfy Lipschitz condition (\ref{metric equ}).
      Hence,
       $C^{k}_i$, $i=0,1$, can be regarded as the norms on $M^{2n}$ since $g_1$ is smooth on closed manifold $M^{2n}$ and $\mathring{M}^{2n}(a,P)$
       is an open, dense submanifold of $M^{2n}$.
      In particular, $g_1=g_J$ is a smooth Riemann metric on $M^{2n}$.
      But, in general, $\varphi(a)$ is not continuous on $M^{2n}$, $\varphi(a)\in L^\infty(M^{2n})$, $\varphi(a)\in C_1^{k}(\mathring{M}^{2n}(a,P))$ only,
      where $k$ is a nonnegative integer.

      Consider the H\"{o}lder spaces $C^{k,\alpha}_i$, $i=1,0$ for a nonegative integer $k$ and $\alpha\in(0,1)$.
      We begin by defining $C^{0,\alpha}_i$.
      Let $d(x,y)_i$ be the distance with respect to metric $g_i$ between $x,y$ in an open connected component, $V_k(a)$, of $\mathring{M}^{2n}(a,P)$.
      Then the function $\varphi(a)$ on $V_k(a)$ is said H\"{o}lder continuous with exponent $\alpha$ if
            \begin{equation}\label{}
              [\varphi(a)]_{\alpha,i}(V_k(a))=\sup_{x\neq y\in U}\frac{|\varphi(a)(x)-\varphi(a)(y)|}{d(x,y)^{\alpha}_i}
            \end{equation}
            is finite.
    Hence the function is said  H\"{o}lder continuous with exponent $\alpha$ if
            \begin{equation}\label{}
              [\varphi(a)]_{\alpha,i}=\max_{1 \leq k\leq N(a)}[\varphi(a)]_{\alpha,i}(V_k(a))
            \end{equation}
           is finite.
           Let
           \begin{equation}\label{}
             \|\varphi(a)\|_{C^{0,\alpha}_i} = \|\varphi(a)\|_{C^0} +[\varphi(a)]_{\alpha,i}\,,\,\,\, i=1,0.
            \end{equation}
       The norm $\|\cdot\|_{C^{k,\alpha}_i}$, $i=1,0$, of $\varphi(a)$ is defined as follows
          \begin{equation}\label{defi norms}
             \|\varphi(a)\|_{C^{k,\alpha}_i} = \|\varphi(a)\|_{C^k_i} +[\nabla^k_i\varphi(a)]_{\alpha,i}\,,\,\,\, i=1,0.
            \end{equation}
  By (\ref{metric equ}), it is easy to get
              \begin{equation}\label{neq estimate}
     C^{-1}(k,\alpha,M^{2n},\omega,J)\|\cdot\|_{C^{k,\alpha}_0}\leq \|\cdot\|_{C^{k,\alpha}_1}\leq  C(k,\alpha,M^{2n},\omega,J)\|\cdot\|_{C^{k,\alpha}_0},
   \end{equation}
    where $k$ is a nonnegative integer, $\alpha\in(0,1)$, $C(k,\alpha,M^{2n},\omega,J)>1$ which is only depending on $M^{2n},\omega,J$ and $k,\alpha$.

    Similarly, for any $1$-form $a\in\Omega^1(M^{2n})$ and $da\in\Omega^2(M^{2n})$, we define $C^{k,\alpha}_i$, $i=0,1$, on $M^{2n}$, where $k$ is a nonnegative integer, $\alpha\in(0,1)$.
    Notice that $C^{k,\alpha}_1$ is well defined on $M^{2n}$, $C^{k,\alpha}_0$ is defined on $\mathring{M}^{2n}(a,P)$.

    \vskip 6pt

    By (\ref{neq estimate}), since $M^{2n}\setminus\mathring{M}^{2n}(a,P)$ has Lebesgue  measure zero, and $g_1$ is smooth metric on $M^{2n}$,
    we have the following lemma:

    \begin{lem}\label{quasi isometry1}
    Restricted to $\mathring{M}^{2n}(a,P)$, $C^{k,\alpha}_0$ and  $C^{k,\alpha}_1$ are quasi isometry, where $k$ is a nonnegative integer, $\alpha\in[0,1)$.
    Where $C^{k,0}_i=C^{k}_i$, $i=0,1$.
    Also, $C^{k,\alpha}_0$, $\alpha\in [0,1)$ can be extended to $M^{2n}$, (\ref{neq estimate}) still holds.
    \end{lem}

     \renewcommand{\theequation}{B.\arabic{equation}}
\setcounter{equation}{0}
\section{Calculations at a point}\label{App B}

          Let $(M^{2n},\omega)$ be a closed symplectic manifold of real dimension $2n$.
          It is clear that there exists an almost K\"{a}hler structure $(\omega,J,g_J)$ on $M^{2n}$
          by the given symplectic $2$-form $\omega$.
          Suppose that $a\in \Psi(M^{2n},[\omega])$, that is, $a\in\Omega^1(M^{2n})$ satisfies the following
          Calabi-Yau equation of Donaldson type
          \begin{equation}\label{CYE1}
            (\omega+da)^n=Ae^f\omega^n,
          \end{equation}
    where $A>0$, $f\in C^\infty(M^{2n},\mathbb{R})$ and
    \begin{equation}\label{CYE2}
           \int_{M^{2n}}Ae^f\omega^n= \int_{M^{2n}}\omega^n.
          \end{equation}
          Let $\omega(a)=\omega+da$, then $\omega(a)$ is a new symplectic form on $M^{2n}$.
          Thus, there exists a new almost K\"{a}hler structure $(\omega(a),J(a),g_{J(a)})$ on $M^{2n}$.
          Let $$ h(1)=g_J-\sqrt{-1}\omega\,\, ({\rm resp.}\,\, h(a)(1)=g_{J(a)}-\sqrt{-1}\omega(a))$$
           be the corresponding almost Hermitian metric on $M^{2n}$ (cf. Appendix \ref{App A}).

           By Proposition \ref{CYequ Lipschitz},
          one can define a partition, $P(\omega(a),J(a),J)$, of $M^{2n}$,
          and find a finite Darboux's coordinate subatlas for $\omega$ and $\omega(a)=\omega+da$ as $\mathcal{V}(a)=\{V_{k}(a)\}_{1\leq k\leq N(a)}$.
          $(\omega|_{V_{k}(a)},J_{k}(0),g_{k}(0))$ is a K\"{a}hler flat structure on $V_{k}(a)$.
           Define a measurable K\"{a}hler potential $\varphi(a)$ on $M^{2n}$ such that $da=dJd\phi=\sqrt{-1}\partial\bar{\partial}\varphi(a)$ which is a smooth $2$-form on $M^{2n}$, and also a $(1,1)$-form on $\mathring{M}^{2n}(a,P)$.
           Notice that by Theorem \ref{diff map} if $a$ is very small we always find a symplectic potential $\phi$ such that $da=dJd\phi$.
           On $V_{k}(a)$, Calabi-Yau equation (\ref{CYE1}) and (\ref{CYE2}) are reduced to the complex Mongue-Amp\`{e}re equation
           \begin{equation}\label{MAE}
             \det(h_{i\bar{j},k}+\frac{\partial^2\varphi_k(a)}{\partial z_i\partial \bar{z}_j})=Ae^F\det(h_{i\bar{j},k}).
         \end{equation}

          Let $p\in M^{2n}$.
          On a small neighborhood $V_p(a)$, there exists a Darboux's coordinate chart $(z_1,\cdot\cdot\cdot,z_n)$ for almost K\"{a}hler
          structures $(\omega,J,g_{J})$ and $(\omega(a),J(a),g_{J(a)})$ such that $z(p)=(0,\cdot\cdot\cdot,0)$.
          We may regard $h(0)(p)_{c\bar{d}}$ and $h(a)(0)(p)_{c\bar{d}}$ are invertible Hermitian $n\times n$ complex matrix due to the construction of $h(0)$
          and $h(a)(0)$ (cf. Proposition \ref{CYequ Lipschitz}).
          By elementary linear algebra, using simultaneous diagonalization, one can choose a basis $(v_1,\cdot\cdot\cdot,v_n)$ for $T_pM^{2n}$
          over $\mathbb{C}$ with respect to
            \begin{equation}
  h(0)(p)_{c\bar{d}}=\left\{
    \begin{array}{ll}
    1 & ~ if~ c=d \\
      &   \\
    0 & ~ if~ c\neq d
   \end{array}
  \right.
  \end{equation}
  and
  \begin{equation}
  h(a)(0)(p)_{c\bar{d}}=\left\{
    \begin{array}{ll}
    a_j & ~ if~ c=d=j \\
      &   \\
    0 & ~ if~ c\neq d,
   \end{array}
  \right.
  \end{equation}
  where $a_1,\cdot\cdot\cdot,a_n$ are strictly positive real numbers.
  Clearly, by Darboux's theorem, it is possible to find complex coordinates $z_1,\cdot\cdot\cdot,z_n$ on $M^{2n}$
  near $p$ such that $v_j=\partial/\partial z_j$ at $p$.
  Hence we have the following lemma:

  \begin{lem}\label{Joc lemma 1}
  (cf. Joyce \cite[Lemma 6.3.1]{Joc})

   $h(0)(p)=2(|dz_1|^2+\cdot\cdot\cdot+|dz_n|^2)$,

   $h(a)(0)(p)=2(a_1|dz_1|^2+\cdot\cdot\cdot+a_n|dz_n|^2)$,

   $\omega_p=\sqrt{-1}(dz_1\wedge d\bar{z}_1+\cdot\cdot\cdot+dz_n\wedge d\bar{z}_n)$,

      $\omega(a)_p=\sqrt{-1}(a_1dz_1\wedge d\bar{z}_1+\cdot\cdot\cdot+a_ndz_n\wedge d\bar{z}_n)$.
  \end{lem}

  \begin{lem}\label{Joc lemma 1'}
  (cf. Joyce \cite[Lemma 6.1.1]{Joc})
  Restricted to open dense submanifold, $\mathring{M}^{2n}(a,P)$, of $M^{2n}$,
  Calabi-Yau equation
  $$
  (\omega+da)^n=Ae^f\omega^n
  $$
  is reduced to  the complex Mongue-Amp\`{e}re equation
           $$
           (\omega+\sqrt{-1}\partial\bar{\partial}\varphi(a))^n=Ae^f\omega^n.
           $$
           Suppose that $\varphi_i(a)$, $\varphi_j(a)$ are two K\"{a}hler potential on $V_i(a)$ and $V_j(a)$ respectively.
           If $V_i(a)\cap V_j(a)\neq \varnothing$, then $dd^c\varphi_i(a)=dd^c\varphi_j(a)$, and $\varphi_i-\varphi_j$ is a constant on  $V_i(a)\cap V_j(a)$.
   Thus, $\omega+\partial\bar{\partial}\varphi(a)=\omega+da$ is positive on $M^{2n}$.
   \end{lem}
   \begin{proof}
   Choose holomorphic coordinates $z_1,\cdot\cdot\cdot,z_n$ on connected open set $V_k(a)$ of $M^{2n}$.
   Then in $V_k(a)$, the new metric $h(a)(0)$ is
   $$
   h_{c\bar{d}}(a)(0)=h_{c\bar{d}}(0)+\frac{\partial^2\varphi_k(a)}{\partial z_c \partial \bar{z}_{\bar{d}}}.
   $$
   As usual, we may interpret $ h_{c\bar{d}}(a)(0)$ as an $n\times n$ invertible Hermitian matrix indexed by
   $c,\bar{d}=1,\cdot\cdot\cdot,n$ associated to Hermitian metric $h(a)(0)$ in $V_k(a)$.

  Recall that if $V_i(a)\cap V_j(a)\neq \varnothing$, without loss of generality,  we may assume that $V_i(a)\cap V_j(a)$ is connected,
  then  on $V_i(a)\cap V_j(a)$
  \begin{eqnarray*}
    da &=& dd^c\varphi_i(a)=\sqrt{-1}\partial\bar{\partial}\varphi_i(a) \\
     &=&  dd^c\varphi_j(a)=\sqrt{-1}\partial\bar{\partial}\varphi_j(a).
  \end{eqnarray*}
  Hence, $\varphi_i(a)-\varphi_j(a)$ is constant on $V_i(a)\cap V_j(a) $.

  Note that $da=\sqrt{-1}\partial\bar{\partial}\varphi (a)$ is continuous on $M^{2n}$,
  and $\omega+da$ non degenerate on $M^{2n}$.
  Thus, $\omega+\partial\bar{\partial}\varphi (a)$ is positive on $M^{2n}$.
   \end{proof}

   By Theorem \ref{diff map}, if $a$ is very small we always find a symplectic potential $\phi\in C^\infty(M^{2n},\mathbb{R})$ such that $da=dJd\phi$.
   Without loss of generality, we may assume that there exists a symplectic potential $\phi\in C^\infty(M^{2n},\mathbb{R})$, $$\int_{M^{2n}}\phi\omega^n=0$$
  such that $da=dJd\phi$ (that is, $a=Jd\phi$).
  By using normal coordinate system \cite{Cha} and the result of Tosatti, Weinkov and Yau \cite[Lemma 2.5]{TWY},
  we have
  \begin{eqnarray}
    \Delta^1 \phi(p)&=&-2h(0)(p)^{c\bar{d}}\partial_c\bar{\partial}_d\phi(p)  \nonumber\\
     &=&\sqrt{-1}\sum^n_{j=1}(dJd\phi)^{(1,1)}(p)(v_j,\bar{v}_j)  \nonumber\\
     &=&  \sqrt{-1}\sum^n_{j=1}(da)^{(1,1)}(p)(v_j,\bar{v}_j).
  \end{eqnarray}
  $\Delta^1$ is the Laplacian of the second canonical connection with respect to the almost K\"{a}hler metric $g_J$ at $p$ (that is, $h(0)(p)$).
  Also by Tosatti-Weinkove-Yau \cite[Lemma 2.6]{TWY}, since $(M^{2n},\omega,J,g_J)$ is a closed almost K\"{a}hler manifold of dimension $2n$,
  then
  \begin{equation}\label{}
   \Delta^L \phi=\Delta^1 \phi,
  \end{equation}
  where $\Delta^L$ is the Laplacian of the Levi-Civita connection with respect to the almost K\"{a}hler metric $g_J$.
  Thus, we can relate $a_1,\cdot\cdot\cdot,a_n$ to $Ae^f$ and $\Delta^L \phi=\Delta^1 \phi$.
  We define $$\Delta^L \phi(p)=\Delta^1 \phi(p)=-2h(0)(p)^{c\bar{d}}\partial_c\bar{\partial}_d\phi(p),$$
  hence we have the following lemma:
  \begin{lem}\label{Joc lemma 2}
    (\cite[Lemma 6.3.2]{Joc})
   In the situation of the previous lemma, at $p\in\mathring{M}^{2n}(a,P)$,
   we have
   $$
   \Pi^n_{j=1}a_j=Ae^{f(p)},\,\,\,\frac{\partial^2\phi}{\partial z_j\partial \bar{z}_j}(p)=a_{j}-1
   $$
   and
   $$
   (\Delta^L) \phi(p)=n-\sum^n_{j=1}a_j.
   $$
    \end{lem}
     \begin{rem}
     By a result of Wang and Zhu \cite[Theorem 1.2]{WZ},
     in general, for a solution of the Calabi-Yau equation $(\omega+da)^n=Ae^f\omega^n$
     on a closed almost K\"{a}hler manifold $(M^{2n},\omega,J,g_J)$,
     we can not find a global symplectic potential $\phi$ such that $da=dJd\phi$.
     But for a small neighborhood $U$ (that is, $\phi$ is very small), we are always able to construct a local symplectic potential $\phi_U$
     such that $da|_U=dJd\phi_U=\sqrt{-1}\partial\bar{\partial}\varphi(a)|_{U}$, where $\varphi(a)$ is measurable K\"{a}hler potential defined in Appendix \ref{App A}.
     In the remainder of this appendix without loss of generality, we always assume that there is a $\phi$ such that $Jd\phi=a$.
     \end{rem}

     By the previous lemmas, it is easy to obtain the following lemma:
     \begin{lem}\label{Joc lemma 4}
     (\cite[Lemma 6.3.3]{Joc})
     In the situation of Lemma \ref{Joc lemma 1} and \ref{Joc lemma 1'}, at $p\in\mathring{M}^{2n}(a,P)$ we have
     $$
     |da(p)|^2_{h(0)}= |\partial\bar{\partial}\varphi(a)(p)|^2_{h(0)}=2\sum^n_{j=1}(a_j-1)^2,\,\,\,|h(a)_{c\bar{d}}(0)|^2_{h(0)}=2\sum^n_{j=1}a_j^2
     $$
     and
     $$
     |h(a)^{c\bar{d}}(0)|^2_{h(0)}=2\sum^n_{j=1}a_j^2.
     $$
     \end{lem}
     Here $(h(a)^{c\bar{d}}(0))$ is the matrix inverse of $(h(a)_{c\bar{d}}(0))$ in coordinates.

     \vskip 6pt

   By the first and last equations in Lemma \ref{Joc lemma 2} and the fact that the geometric mean $(a_1\cdot\cdot\cdot a_n)^{\frac{1}{n}}$
   is less than or equal to the arithmetic mean $\frac{1}{n}(a_1+\cdot\cdot\cdot +a_n)$, it is easy to obtain
   the following inequality:
   \begin{equation}\label{}
     (\Delta^L\phi)(p)\leq n-nA^{\frac{1}{n}}e^{\frac{f}{n}}<n.
   \end{equation}
   As $(\Delta^L\phi)(p)=n-\sum^n_{j=1}a_{j}$ by Lemma \ref{Joc lemma 2} and $a_j$ are positive,
   from Lemma \ref{Joc lemma 4} one can show that at $p$,
   $$|da(p)|^2_{h(0)}= |\partial\bar{\partial}\varphi(a)(p)|^2_{h(0)}\leq 2n+2(n-\Delta^L\phi)^2,$$
   and $$|h(a)_{c\bar{d}}(0)|^2_{h(0)}\leq 2n+2(n-\Delta^L\phi)^2.$$
   As these hold for all $p\in M^{2n}$, since $g_1$ and $g_0$ are quasi isometry on $\mathring{M}^{2n}(a,P)$ (cf. (\ref{metric equ}) and (\ref{defi norms})),
   we have the following estimates:
    \begin{equation}\label{Joc estimates 1}
      \|h(a)_{c\bar{d}}(0)\|^2_{C^0_1}\leq C_1,\,\,\,\|da\|^2_{C^0_1}= \|\partial\bar{\partial}\varphi(a)\|^2_{C^0_1}\leq C_3,
    \end{equation}
    where constant $C_1,C_3$ depending only on $n$ and $\|\Delta^L\phi\|_{C^0_1}$.

    Now if $\log A>-\inf_{M^{2n}}f$, then $Ae^f>1$ on $M^{2n}$, which contradicts the equation
     $$\int_{M^{2n}}Ae^f\omega^n= \int_{M^{2n}}\omega^n.$$
    Similarly $\log A<\sup_{M^{2n}}f$ leads to a contradition.
    Hence, $$-\sup_{M^{2n}}f\leq\log A\leq \inf_{M^{2n}}f,\,\,\,{\rm and}\,\,\, |\log A|\leq\|f\|_{C^0_1}.$$
    It follows that
    \begin{equation}\label{}
      e^{-2\|f\|_{C^0_1}}\leq Ae^f\leq e^{2\|f\|_{C^0_1}}
    \end{equation}
    on $M^{2n}$.
    Since $\Pi^n_{j=1}a_j=Ae^{f(p)}$ by Lemma \ref{Joc lemma 2}, we see that
    $$
    a^{-1}_j=A^{-1}e^{-f(p)}\Pi^n_{1\leq k\leq n, j\neq k}a_k.
    $$
    Using (\ref{Joc estimates 1}) to estimate $A^{-1}e^{-f(p)}$ and the inequality $$a_k\leq n-(\Delta^L\phi)(p)$$
    derived from the third equation in Lemma \ref{Joc lemma 2}, we find that
    \begin{equation}\label{Joc estimates 2}
      a^{-2}_j\leq e^{4\|f\|_{c^0}}(n-(\Delta^L\phi)(p))^{2n-2}.
    \end{equation}
    From (\ref{Joc estimates 2}) and Lemma \ref{Joc lemma 4}, we have
    \begin{equation}\label{Joc estimates 3}
     \|h(a)^{c\bar{d}}(0)\|_{C^0_1}\leq C_2,
    \end{equation}
    where the constant $C_2$ depending on $n$, $\|f\|_{C^0_1}$ and $\|\Delta^L\phi\|_{C^0_1}$.

    In summary, we have the following proposition:
    \begin{prop}\label{Joc prop 1}
    Let $(M^{2n},\omega,J,g_J)$ be an almost K\"{a}hler manifold of dimension $2n$.
    Let $f\in C^0(M^{2n},\mathbb{R})$, $a\in C^1(T^*M^{2n})$ and $A>0$.
    Set $\omega(a)=\omega+da=\omega+dJd\phi$, suppose
    $\omega(\phi)^n=Ae^f\omega^n$, and let $h(a)(0)=h(\phi)(0)$ be the K\"{a}hler metric with symplectic form $\omega(a)=\omega(\phi)$ on $\mathring{M}^{2n}(a,P)$.
    Then
    $$
    \Delta^L\phi\leq n-nA^{\frac{1}{n}}e^{\frac{f}{n}}<n,
    $$
    and there are constant $c_1$, $c_2$ and $c_3$ depending only on $n$ and upper bounds for $\|f\|_{C^0_1}$ and $\|\Delta^L\phi\|_{C^0_1}$,
    such that
    $$
     \|h(\phi)_{c\bar{d}}(0)\|_{C^0_1}\leq c_1, \|h(\phi)^{c\bar{d}}(0)\|_{C^0_1}\leq c_2
    $$
    and
    $$
    \|da\|_{C^0_1}=\|\partial\bar{\partial}\varphi(a)\|_{C^0_1}\leq c_3.
    $$
    Here all norms are with respect to the metric $g_1$ (that is, $g_J$).
    \end{prop}

    This shows that a priori bound for $\Delta^L\phi$ yields a priori bound for $h(a)(0)$ and
    $$dJd\phi=da=\sqrt{-1}\partial\bar{\partial}\varphi(a).$$
    Notice that $\partial\bar{\partial}\varphi(a)$ is smooth on $M^{2n}$.
    \begin{lem}\label{wedge equ}
    (D.D. Joyce\cite[Lemma 6.3.5]{Joc})
    In the situation of Proposition \ref{Joc prop 1}, we have
    $$
    d\phi\wedge Jd\phi\wedge\omega^{n-1}=\frac{1}{n!}|\nabla\phi|^2_{h(0)}\omega^n
    $$
    and
    $$
    d\phi\wedge Jd\phi\wedge\omega^{n-j-1}\wedge\omega(\phi)^j=G_j\omega^n
    $$
    for $j=1,2,\cdot\cdot\cdot,n-1$, where $G_j$ is a nonnegative real function on $M^{2n}$.
    \end{lem}
    \begin{proof}
    Notice that $\omega^{n-1}$ is a positive $J-(n-1,n-1)$ form on $M^{2n}$,
    $$*_{g_J}(\omega^{n-1})=(n-1)!\omega,\,\,\, {\rm and}\,\,\, \omega^n=n!dvol_{g_J},$$
    where $dvol_{g_J}$ is the volume element with respect the metric $g_J$.
   Since $$d\phi=\partial_J\phi+\bar{\partial}_J\phi,\,\,Jd\phi=\sqrt{-1}(\bar{\partial}_J\phi-\partial_J\phi),$$
   thus
   \begin{equation}\label{}
     d\phi\wedge Jd\phi\wedge \omega^{n-1}=(d\phi\wedge Jd\phi,*_{g_J}\omega^{n-1})dvol_{g_J}=\frac{1}{n!}(d\phi\wedge Jd\phi,\omega)\omega^n.
   \end{equation}
   It is easy to see that
   $$
   (d\phi\wedge Jd\phi,\omega)=|\nabla\phi|^2_{g_J}.
   $$

   The reason $ d\phi\wedge Jd\phi\wedge \omega^{n-1}$ is a nonnegative multiple of $\omega^n$ is that $*_{g_J}(\omega^{n-1})$ is a positive
   $J-(1,1)$ form. In the same way, if we can show $*_{g_J}(\omega^{n-j-1}\wedge\omega(\phi)^j)$ is a nonnegative multiple of $\omega^n$,
   which is what we have to prove.
   But using equations in Lemma \ref{Joc lemma 1} one can readily show that $*_{g_J}(\omega^{n-j-1}\wedge\omega(\phi)^j)$ is a positive
   $J-(1,1)$ form, and the proof is finished.
     \end{proof}

  \end{appendices}

  \vskip 6pt

  \noindent{\bf Acknowledgements.}\,
   The authors would like to thank Professors Teng Huang and Xiaowei Xu for their valuable comments and suggestions.


  \vskip 24pt

 \noindent Qiang Tan\\
School of Mathematical Sciences, Jiangsu University, Zhenjiang, Jiangsu 212013, China\\
 e-mail: tanqiang@ujs.edu.cn\\

 \vskip 6pt

 \noindent Hongyu Wang\\
 School of Mathematical Sciences, Yangzhou University, Yangzhou, Jiangsu 225002, China\\
 e-mail: hywang@yzu.edu.cn\\

\end{document}